\documentclass[10pt]{article}%

\usepackage{float}

\usepackage{tikz}
\usepackage{pgf}

\usetikzlibrary{intersections}
\usetikzlibrary{patterns}
\usepackage{xcolor}

\definecolor{lightgray}{rgb}{0.827, 0.827, 0.827}

\usetikzlibrary{arrows}
\usetikzlibrary{calc}
\usetikzlibrary{decorations.pathmorphing}
\usepackage{xfrac}    

\usepackage{faktor}
\usepackage{soul}
\setstcolor{red}
\setulcolor{red}

\usepackage{rotating} 

\usepackage{mathrsfs}

\usepackage{romannum}

\usepackage{color}
\usepackage{amsmath}
\usepackage{graphicx}
\usepackage{graphics}
\usepackage{float}
\usepackage{amsfonts}
\usepackage{amssymb}%
\usepackage{latexsym}
\usepackage{psfrag}
\usepackage{accents}
\usepackage{caption}
\usepackage{subcaption}

\newtheorem{theorem}{Theorem}[section]
\newtheorem{corollary}[theorem]{Corollary}

\newenvironment{proof}[1][Proof]{\noindent \emph{#1.} }
{\hfill \ \rule{0.5em}{0.5em}}
\newtheorem{lemma}[theorem]{Lemma}
\newtheorem{proposition}[theorem]{Proposition}

\newtheorem{assumption}[theorem]{Assumption}

\numberwithin{equation}{section}
\numberwithin{table}{section}
\numberwithin{figure}{section}

\usepackage[a4paper]{geometry}
\newtheorem{remark}[theorem]{Remark}
\newtheorem{example}[theorem]{Example}
\newcommand{\noi}{\noindent}

\newcommand{\R}{\mathbb{R}}

\newcommand{\cH}{{\cal H}}

\newcommand{\cR}{{\cal R}}
\newcommand{\cS}{{\cal S}}

\newcommand{\cW}{{\cal W}}

\newcommand{\cY}{{\cal Y}}

\newcommand{\cO}{{\cal O}}
\newcommand{\cZ}{{\cal Z}}
\newcommand{\bx}{x}

\newcommand{\bnu}{\nu}

    \newcommand\quotient[2]{
        \mathchoice
            {
                \text{\raise1ex\hbox{$#1$}\Big/\lower1ex\hbox{$#2$}}%
            }
            {
                #1\,/\,#2
            }
            {
                #1\,/\,#2
            }
            {
                #1\,/\,#2
            }
    }

\newcommand{\re}{{\rm e}}
\newcommand{\ri}{{\rm i}}
\newcommand{\rd}{{\rm d}}

\newcommand{\beq}{\begin{equation}}
\newcommand{\eeq}{\end{equation}}
\newcommand{\beqs}{\begin{equation*}}
\newcommand{\eeqs}{\end{equation*}}
\newcommand{\bit}{\begin{itemize}}
\newcommand{\eit}{\end{itemize}}
\newcommand{\ben}{\begin{enumerate}}
\newcommand{\een}{\end{enumerate}}
\newcommand{\bal}{\begin{align}}
\newcommand{\eal}{\end{align}}
\newcommand{\bals}{\begin{align*}}
\newcommand{\eals}{\end{align*}}
\newcommand{\bse}{\begin{subequations}}
\newcommand{\ese}{\end{subequations}}
\newcommand{\bpr}{\begin{proposition}}
\newcommand{\epr}{\end{proposition}}
\newcommand{\bre}{\begin{remark}}
\newcommand{\ere}{\end{remark}}
\newcommand{\bpf}{\begin{proof}}
\newcommand{\epf}{\end{proof}}
\newcommand{\ble}{\begin{lemma}}
\newcommand{\ele}{\end{lemma}}
\newcommand{\bco}{\begin{corollary}}
\newcommand{\eco}{\end{corollary}}
\newcommand{\bex}{\begin{example}}
\newcommand{\eex}{\end{example}}
\newcommand{\bth}{\begin{theorem}}
\newcommand{\enth}{\end{theorem}}

\newcommand{\Rea}{\mathbb{R}}

\newcommand{\GR}{{\partial B_R}}

\newcommand{\eps}{\varepsilon}

\newcommand{\pdiff}[2]{\frac{\partial #1}{\partial #2}}

\newcommand{\tendi}{\rightarrow \infty}

\def\XXint#1#2#3{{\setbox0=\hbox{$#1{#2#3}{\int}$}
     \vcenter{\hbox{$#2#3$}}\kern-.5\wd0}}

\usepackage{hyperref}
\counterwithin{figure}{section}
\definecolor{myblue}{rgb}{0,0,0.6}
\hypersetup{colorlinks=true,linkcolor=myblue,citecolor=myblue,filecolor=myblue,urlcolor=myblue}
\newcommand*{\N}[1]{\left\|#1\right\|}

\allowdisplaybreaks[4]

\newcommand{\tfa}{\text{ for all }}
\newcommand{\tfor}{\text{ for }}

\newcommand{\tin}{\text{ in }}
\newcommand{\ton}{\text{ on }}

\newcommand{\tand}{\text{ and }}
\newcommand{\tst}{\text{ such that }}

\newcommand{\vertiii}[1]{{\left\vert\kern-0.25ex\left\vert\kern-0.25ex\left\vert #1
    \right\vert\kern-0.25ex\right\vert\kern-0.25ex\right\vert}}

\newcommand{\DtN}{{\rm DtN}_k}

\definecolor{jwcol}{RGB}{27, 137, 18}  

\definecolor{dalcol}{rgb}{0.8,0,0}

\definecolor{escol}{rgb}{0,0,0.8}
\definecolor{estcol}{rgb}{0,0.5,0}
\definecolor{esnewcol}{rgb}{0,0.5,0}

\newcommand{\es}[1]{{\color{black}{#1}}}

\newcommand{\Ascat}{A_{\rm scat}}
\newcommand{\cscat}{c_{\rm scat}}
\newcommand{\Ascatout}{A_{\rm out}}
\newcommand{\cscatout}{c_{\rm out}}
\newcommand{\Ascatin}{A_{\rm in}}
\newcommand{\cscatin}{c_{\rm in}}
\newcommand{\Omegain}{\Omega_{\rm in}}
\newcommand{\Omegaout}{\Omega_{\rm out}}
\newcommand{\uout}{u_{\rm out}}
\newcommand{\uin}{u_{\rm in}}

\newcommand{\supp}{{\rm supp}}

\newcommand{\abs}[1]{{\left\lvert{#1}\right\rvert}}
\newcommand{\norm}[1]{{\left\lVert{#1}\right\rVert}}

\newcommand{\CPF}{{C_{\rm PF}}}

\newcommand{\Cres}{{C_{\rm res}}}

\newcommand{\Ccont}{{C_{\rm cont}}}
\newcommand{\Ccoer}{{C_{\rm coer}}}
\newcommand{\ccoer}{{c_{\rm coer}}}

\newcommand{\Cemb}{C_{\rm emb}}

\newcommand{\Ccom}{C_{\rm com}}
\newcommand{\Cell}{C_{\rm ell}}

\newcommand{\Cpw}{C_{\rm pw}}
\newcommand{\Cloc}{C_{\rm loc}}
\newcommand{\Capprox}{C_{\rm approx}}
\newcommand{\Ckappa}{\kappa}
\newcommand{\Cppw}{C_{\rm ppw}}
\newcommand{\Cp}{p}
\newcommand{\Cdagger}{C_\dagger}
\newcommand{\Csuper}{C'_{\rm super}}
\newcommand{\Csuperk}{{C_{\rm super}}}
\newcommand{\Cinv}{C'_{\rm inv}}
\newcommand{\Cinvk}{{C_{\rm inv}}}
\newcommand{\Cstar}{C_*}
\newcommand{\Cqu}{C_{\rm qu}}
\newcommand{\Cca}{C_{\rm ca}}

\newcommand{\mythmname}[1]{\textbf{\emph{(#1)}}}

\DeclareMathOperator{\Id}{Id}

\newcommand{\tr}{{\rm tr}}

\newcommand{\Amin}{A_-}

\newcommand{\Rscat}{R_{\rm scat}}

\newcommand{\RPMLo}{R_{\rm PML, -}}
\newcommand{\RPMLt}{R_{\rm PML, +}}

\newcommand{\Rtr}{R_{\tr} }

\makeatletter
\newcommand{\settheoremtag}[1]{
  \let\oldthetheorem\thetheorem
  \renewcommand{\thetheorem}{#1}
  \g@addto@macro\endtheorem{
    \addtocounter{theorem}{-1}
    \global\let\thetheorem\oldthetheorem}
  }
\makeatother

\usetikzlibrary{positioning}
\usepackage{lscape}

\definecolor{jeffColor}{RGB}{102, 0, 204}

\newcommand{\euanspace}{, \,}

\newcommand{\Ot}{\Omega_{\rm p}}
\newcommand{\OI}{\Omega_-}
\newcommand{\Gt}{\Gamma_{\rm p}}
\newcommand{\GI}{\Gamma_-}

\newcommand{\Zjdual}[3]{\norm{#1}_{(\mathcal{Z}_k^{#2,<}(#3))^*}}
\newcommand{\Zjdualintro}[3]{\norm{#1}_{(H_k^{#2,<}(#3))^*}}
\newcommand{\Wjdual}[3]{\norm{#1}_{(\mathcal{W}_k^{#2,<}(#3))^*}}
\newcommand{\Zj}[3]{\norm{#1}_{\mathcal{Z}_k^{#2}(#3)}}
\newcommand{\Wj}[3]{\norm{#1}_{\mathcal{W}_k^{#2}(#3)}}
\newcommand{\newell}{s}
\newcommand{\Omegasource}{\Omega_{\square}}
\newcommand{\Cover}{C_{\rm over}}

\title{
Helmholtz FEM solutions are locally quasi-optimal modulo low frequencies}

\author{
M.~Averseng\thanks{Department of Mathematical Sciences, University of Bath, Bath, BA2 7AY, UK, \tt M.Averseng@bath.ac.uk },
\,\,
J.~Galkowski\thanks{Department of Mathematics, University College London, 25 Gordon Street, London, WC1H 0AY, UK,   \tt J.Galkowski@ucl.ac.uk},
\,\, E.~A.~Spence\thanks{Department of Mathematical Sciences, University of Bath, Bath, BA2 7AY, UK, \tt E.A.Spence@bath.ac.uk }
}
\date{}

\begin{document}
\pagenumbering{arabic}

\maketitle

\begin{abstract}

For $h$-FEM discretisations of the Helmholtz equation with wavenumber $k$, we obtain $k$-explicit analogues of the classic local FEM error bounds of \cite{NiSc:74}, \cite[\S9]{Wa:91}, \cite{DeGuSc:11}, showing that these bounds hold with constants independent of $k$, provided one works in Sobolev norms weighted with $k$ in the natural way.

We prove two main results: 
(i) a bound on the local $H^1$ error by the best approximation error plus the $L^2$ error, both on a slightly larger set,  and 
(ii) the bound in (i) but now with the $L^2$ error replaced by the error in a negative Sobolev norm.
The result (i) is valid for shape-regular triangulations, and is the $k$-explicit analogue of the main result of \cite{DeGuSc:11}. 
The result (ii) is valid when the mesh is locally quasi-uniform on the scale of the wavelength (i.e., on the scale of $k^{-1}$) and is the $k$-explicit analogue of the results of \cite{NiSc:74}, \cite[\S9]{Wa:91}. 

Since our Sobolev spaces are weighted with $k$ in the natural way, the result (ii) indicates that the Helmholtz FEM solution is 
locally quasi-optimal modulo low frequencies (i.e., frequencies $\lesssim k$).
Numerical experiments confirm this property, and also highlight interesting propagation phenomena in the Helmholtz FEM error.

\

\noi\textbf{AMS subject classifications:} 35J05, 65N15, 65N30, 78A45\\
\noi\textbf{Keywords:} Finite Element Method, Helmholtz equation

\end{abstract}

%

\section{Introduction:~the main result in a simple setting}\label{sec:intro}

\subsection{The PML approximation to the Helmholtz exterior Dirichlet problem and its FEM discretisation.}\label{sec:introHelmholtz}

Let $\Omega_-\subset\Rea^d$ be a bounded Lipschitz open set with its open complement $\Omega_+ := \Rea^d\setminus \overline{\Omega_-}$ connected. 
Let $\widetilde{u}$ 
be the solution of the  variable-coefficient exterior Dirichlet problem for the  Helmholtz equation 
\[
-k^{-2}\nabla \cdot \big(A_{\rm scat}(x) \nabla \widetilde u(x)\big) - (c_{\rm scat}(x))^{-2} \widetilde u(x) = g(x) \quad \textup{in } \Omega_+,
\qquad
\widetilde{u}=0 \quad\ton\partial \Omega_+
\]
satisfying the Sommerfeld radiation condition, and with the supports of $g$, $I- A_{\rm scat}$, and $1-c_{\rm scat}$ compact. Let $u\in H^1_0(\Omega)$ be the radial perfectly-matched-layer (PML) approximation to $\widetilde{u}$, where $\Omega$ is the truncated domain; i.e., $u$ is the solution to the
variational problem 
\begin{equation}
	\label{eq:varfEDP}
	\textup{find } u \in H^1_0(\Omega) \textup{ such that } a(u,v) = G(v)\tfa  v \in H^1_0(\Omega),
\end{equation}
where $a$ is the sesquilinear form given by 
\[a(u,v) = \int_{\Omega} k^{-2} A\nabla u \cdot \overline{\nabla v} - c^{-2}u\overline{v}\,,\]
$G(v) = \int_{\Omega} g \overline{v}$, and the coefficients $A$ and $c$ are defined  in \S\ref{sec:PML} in terms of the PML scaling function and (respectively) $A_{\rm scat}$ and $c_{\rm scat}$.

We consider the Galerkin discretisation of \eqref{eq:varfEDP} 
using the standard conforming Lagrange finite element spaces $\{V_h\}_{h > 0}$ 
of continuous piecewise polynomials of degree $p$ on a family of shape-regular triangulations $(\mathcal{T}_h)_{h > 0}$ of $\Omega$; i.e., 
\begin{equation}
	\label{eq:varfEDPh}
	\textup{find } u_h \in V_h\textup{ such that } a(u_h,v_h) = G(v_h) \quad \tfa v_h \in V_h.
\end{equation}
Subtracting \eqref{eq:varfEDPh} from \eqref{eq:varfEDP} we find that Galerkin orthogonality holds:
	\begin{equation}
		\label{eq:GOGintro}
		a(u - u_h,v_h) = 0 \quad \tfa v_h \in V_h.
	\end{equation}

\subsection{First result:~bound on the local FEM error in $H^1_k$ with an $L^2$ error term.}

Given two subsets $\Omega_0 \subset \Omega_1 \subset \Omega$, let
\begin{equation}
	\label{eq:defDistInf}
	\partial_<(\Omega_0,\Omega_1) := \textup{dist}\big(\partial \Omega_0 \setminus \partial \Omega, \partial \Omega_1 \setminus \partial \Omega\big),
\end{equation}
with the convention $\partial_<(\Omega_0,\Omega_1) = +\infty$ when $\Omega_1 =  \Omega$; see Figure \ref{fig:distance}.
Working with this notion of distance allows us to consider subdomains that
go up to the boundary.

\begin{figure}[h]
\begin{center}
\begin{tikzpicture}

\def\radiusA{2cm}
\def\radiusB{1.5cm}
\def\radiusC{1cm}

\def\startangle{-70}
\def\endangle{70}
\begin{scope}[scale=2,rotate=90]

\begin{scope}

\coordinate (O1) at (0,0);
\coordinate (O2) at (2cm,0);
\coordinate (O3) at (2cm,0);

\draw[name path=circleA]  (O1) ++(\startangle:\radiusA) arc (\startangle:\endangle:\radiusA);

    \clip   (O1) ++(-90:\radiusA) arc (-90:90:\radiusA)-- (O1)++(-90:\radiusA);

    \draw[name path=circleB, fill=lightgray] (O2) circle (\radiusB);
            \filldraw[name path=circleC,pattern=north east lines] (O3) circle (\radiusC);
            
     \draw (O2) ++ (135:1.9cm)node[below] {$\Omega$};
     \draw (O2) ++ (135:1.2cm)node[below] {$\Omega_1$};

      \draw[->] (.9,0) node[right]{$\partial_{<}(\Omega_0,\Omega_1)$}--(1,0);
       \draw[->] (.9,0) --(.5,0);
            
\end{scope}            
      \draw[->]  (\radiusA+4,0)node[above] {$\Omega_0$} --(\radiusA-4,0);
\end{scope}




\end{tikzpicture}
\caption{\es{Illustration of the distance $\partial_<(\Omega_0,\Omega_1)$ defined by \eqref{eq:defDistInf}, with $\Omega_0$ hatched and $\Omega_1$ shaded.}}
\label{fig:distance}
\end{center}
\end{figure}

The Sobolev norms $\|\cdot\|_{H^s_k(D)}$ for $D$ a bounded Lipschitz domain are defined as for $\|\cdot\|_{H^s(D)}$ (via restriction of $\|\cdot\|_{H^s(\Rea^d)}$ to $D$, with this second norm defined by the Fourier transform), except that now we weight the $j$th derivative with $k^{-j}$; see \S\ref{sec:norms}.

\begin{theorem}[Local quasioptimality in $H^1_k$ up to an $L^2$ error term]
\label{thm:intro1}
Suppose that $A$ and $c$ are $L^\infty(\Omega)$ and the PML scaling function $f_\theta$ (defined in \S\ref{sec:PML}) is $W^{1,\infty}(\Omega)$. Given $C_0>0$ there exists $C_1,\Cstar>0$ such that the following is true.
Let $\Omega_0\subset\Omega_1\subset \Omega$ be such that $\Omega_0\neq \Omega_1$,
\beq\label{eq:intro_condition}
 d:=\partial_<(\Omega_0,\Omega_1)
  \geq \frac{C_0}{k},
 \quad\tand \quad 
 \max_{K\cap \Omega_1\neq \emptyset} h_K \leq \frac{C_1}{k}.
 \eeq
Given $k>0$, let $u \in H^1_0(\Omega)$ and $u_h \in V_h$ satisfy the Galerkin orthogonality \eqref{eq:GOGintro}.
Then 
	\begin{align}
		&\norm{u - u_h}_{H^1_k(\Omega_0)} \leq \Cstar		\Big(\min_{w_h \in V_h} 
		\norm{u - w_h}_{H^1_k(\Omega_1)} 
		+ \norm{u - u_h}_{L^2(\Omega_1)} 
		\Big).\label{eq:intro1}
	\end{align}
\end{theorem}
We make two remarks (which also hold for Theorem \ref{thm:intro2} below).
\ben
\item 
Recall that, for standard finite-element spaces with $p$ fixed, $h_K k$ must be chosen as a decreasing function of $k$ to maintain accuracy (see \cite{GS3} and the references therein); thus the second condition in \eqref{eq:intro_condition} is not restrictive.
\item The assumption that $u$ and $u_h$ satisfy the Galerkin orthogonality \eqref{eq:GOGintro} can be weakened to $u$ and $u_h$ satisfying a local version of Galerkin orthogonality on $\Omega_1$; see \eqref{eq:GOG1} below. 
\een

\subsection{Second result:~bound on the local FEM error in $H^1_k$ with an error term in a negative norm}

Theorem \ref{thm:intro1} holds for all shape-regular meshes $\mathcal{T}_h$ (i.e., the mesh elements are uniformly ``well-shaped"), and thus the mesh can in principle be highly non-uniform. 
However, if the mesh $\mathcal{T}_h$ is \emph{locally quasi-uniform on a scale of $k^{-1}$} (i.e., in every ball of radius proportional to $k^{-1}$, the mesh elements diameters are comparable) \es{and the PDE coefficients and boundary of the domain have sufficient regularity}, then the $L^2$ norm of the error on the right-hand side of \eqref{eq:intro1} can be replaced by a negative Sobolev norm.

\begin{theorem}[Local quasioptimality in $H^1_k$ up to an error term in a negative norm]
\label{thm:intro2}
Suppose that, for some $\ell \in\mathbb{Z}^+$, $A$ and $c$ are $C^{\ell,1}(\overline{\Omega})$ \es{(i.e., their $\ell$th derivatives are Lipschitz)}, the PML scaling function $f_\theta$ (defined in \S\ref{sec:PML}) is $C^{\ell+1,1}(\overline{\Omega})$,
and $\partial \Omega$ is $C^{\ell+1,1}$. 
Given $C_0, \Cqu>0$ there exists $C_1,C_2,\Cstar>0$ such that the following is true.
Let $\Omega_0\subset\Omega_1\subset \Omega$ be such that $\Omega_0\neq \Omega_1$ and 
\eqref{eq:intro_condition} holds.
Assume further that 
$\mathcal{T}$ is {\em quasi-uniform on scale $k^{-1}$}, in the sense that, for every ball $B$ of radius at most $C_2 k^{-1}$,
	\begin{equation}
		\label{eq:AssQuThm}
		\frac{\max_{K\cap B \neq \emptyset} h_K}{\min_{K \cap B \neq \emptyset} h_K} \leq \Cqu.
	\end{equation} 	
Given $k>0$, let $u \in H^1_0(\Omega)$ and $u_h \in V_h$ satisfy \eqref{eq:GOGintro}.
Then	
\begin{align}
&\norm{u - u_h}_{H^1_k(\Omega_0)} \leq \Cstar 
\Big(\min_{w_h \in V_h} \norm{u - w_h}_{H^1_k(\Omega_1)} 
+ \Zjdualintro{u - u_h}{\min\{\ell+1,p\}}{\Omega_1}
\Big).
\label{eq:intro2}
\end{align}
\end{theorem} 

Note that meshes satisfying \eqref{eq:AssQuThm} can be highly non-uniform on $\Omega$; indeed, the ratio between the largest and smallest mesh elements on an $\cO(1)$ scale can be proportional to $\exp(\alpha k)$ for some $\alpha\in \Rea$ (depending on $\Cqu$).

We now define the norm $\Zjdualintro{\cdot}{s}{\Omega_1}$ appearing in the last term on the right-hand side of \eqref{eq:intro2}, but highlight that when $\Omega_1$ is an interior subset of $\Omega$, $\Zjdualintro{\cdot}{s}{\Omega_1}$ is equivalent to $\|\cdot\|_{H_k^{-s}(\Omega_1)}$. For $s\geq 0$ and $D\subset \Omega$ Lipschitz, let 
\beqs
H^{s,<}(D):= \overline{\Big\{
v \in H^1_0(\Omega)\,:\, v|_D\in H^s(D), \,
\supp \,v \subset \overline{D},\, \partial_< (\supp\, v, D)>0
\Big\}}
\eeqs
(where the closure is taken with respect to the $H^s$ norm) and 
\beqs
\Zjdualintro{v}{s}{D}:= \sup_{w \in H^{s,<}(D),\, \|w\|_{H^s_{\es{k}}(D)}=1}
\big| \es{v(w)}
\big|.
\eeqs
When $\partial D \cap \partial \Omega= \emptyset$ (i.e., $D$ is an interior subset of $\Omega$), $\Zjdualintro{\cdot}{s}{D}$ is equivalent to $\|\cdot\|_{H_k^{-s}(D)}$ by \eqref{eq:neg_norm_intro} (see also \cite[Equation 9.18]{Wa:91}).
When $\partial D\cap \partial\Omega\neq \emptyset$, 
\beqs
H_0^s(D) \subset H^{s,<}(D) \subset H^s(D)\cap H^1_0(D)
\eeqs
(where $H_0^s(D)$ is the closure of $C^\infty_{\rm comp}(D)$ in $H^s(D)$), and so 
\beqs
\N{v}_{H^{-s}_k(D)} \leq \Zjdualintro{v}{s}{D}\leq C\max \big\{ \N{v}_{H^{-1}_k(D)}, \N{v}_{\widetilde{H}^{-s}_k(D)}\big\}.
\eeqs
To informally understand the differences between these norms, we note that $\N{\cdot}_{H^{-s}_k(D)}$ doesn't ``see'' the boundary of $D$, $\N{v}_{\widetilde{H}^{-s}_k(D)}$ sees the boundary of $D$, and $\Zjdualintro{v}{s}{D}$ sees only the parts of $\partial D$ that coincide with $\partial\Omega$.

\subsection{The relationship of Theorems \ref{thm:intro1} and \ref{thm:intro2} to other results in the literature.}

Estimates on the local FEM error for second-order linear elliptic PDEs were pioneered by Nitsche and Schatz in \cite{NiSc:74}; see also \cite{De:75}, \cite{ScWa:77}, \cite{ScWa:82}, \cite[Chapter 9]{Wa:91}, and \cite{DeGuSc:11}. 
These arguments use that the sesquilinear forms of second-order linear elliptic PDEs are coercive (i.e., sign definite) on sufficiently-small balls. This property is used to prove, again on small balls, a discrete analogue of the classic Caccioppoli estimate (bounding the $H^1$ norm of the PDE solution in terms of the $L^2$ norm and the data on a slightly larger set). The Caccioppoli estimate is the main ingredient required to prove a bound of the form \eqref{eq:intro1} on small balls. A covering argument is then used to obtain the bound on an arbitrary domain from the bound on small balls. These arguments combined with a duality argument and elliptic regularity then produce a bound of the form \eqref{eq:intro2}. 

Although these classic results apply to the Helmholtz equation, they don't use norms weighted with $k$ and the constants in the bounds are not explicit in $k$. The main motivation for the present paper was to obtain the analogues of the results in \cite{NiSc:74}, \cite[Chapter 9]{Wa:91}, and \cite{DeGuSc:11} applied to the Helmholtz equation in $k$-weighted norms, and with constants explicit in $k$. 
Roughly speaking, we show that the results of \cite{NiSc:74}, \cite[Chapter 9]{Wa:91}, and \cite{DeGuSc:11} hold for the Helmholtz equation with constants independent of $k$, provided that one works in $k$-weighted norms. In more detail,
\bit
\item
Theorem \ref{thm:intro1} is a $k$-explicit version of \cite[Theorem 3.4]{DeGuSc:11}, with both proved under the assumption that the mesh is shape-regular. This result (in non $k$-explicit form) for quasi-uniform meshes appears as \cite[Theorem 4]{De:75}, \cite[Theorem 4.1]{ScWa:82}, and \cite[Theorem 9.1]{Wa:91}.
\item  Theorem \ref{thm:intro2} is, roughly-speaking, a $k$-explicit version of 
 \cite[Theorem 5.1(i)]{NiSc:74} and \cite[Theorem 9.2]{Wa:91}. 
The differences are 
\bit
\item in 
\cite[Theorem 5.1(i)]{NiSc:74} the best approximation error on $\Omega_1$ (i.e., the first term on the right-hand sides of \eqref{eq:intro1} and \eqref{eq:intro2}) is estimated by the standard polynomial approximation result (\eqref{eq:ap} below) and the subdomains are assumed not to touch the boundary, and 
\item the results in \cite{NiSc:74} and \cite[Chapter 9]{Wa:91} are geared towards quasi-uniform meshes (see \cite[Assumption A3]{NiSc:74}, 
\cite[Equation 9.6]{Wa:91}) whereas Theorem \ref{thm:intro2} requires quasi-uniformity only on the scale of the wavelength (i.e., on a scale of $k^{-1}$).
\eit
\eit

An additional difference between the results of the present paper and  existing results is that we cover Helmholtz transmission problems, i.e., those with discontinuous $A_{\rm scat}$ and $c_{\rm scat}$, and \es{the analogue of} Theorem \ref{thm:intro2} in this context appears to be new (independent of the $k$-explicitness); indeed, \cite{NiSc:74, De:75} consider second order linear elliptic PDEs with smooth coefficients and \cite{ScWa:82, Wa:91} cover Poisson's equation.

\subsection{Interpreting the results of Theorems \ref{thm:intro1} and \ref{thm:intro2}.}\label{sec:interpretation}

The standard interpretation of the non-$k$-explicit versions of the bounds \eqref{eq:intro1} and \eqref{eq:intro2} is that the FEM solution is locally quasi-optimal up to a lower-order term that allows error to propagate into $\Omega_0$ from the rest of the domain. (This second term is sometimes called the ``pollution'' or ``slush'' term in the literature; later in the paper we refer to it as the ``slush'' term to avoid confusion with the pollution effect for the Helmholtz $h$-FEM.)

The fact that \eqref{eq:intro1} and \eqref{eq:intro2} are proved in $k$-weighted norms with $k$-independent constants allows this interpretation to be refined in the Helmholtz context to \emph{Helmholtz FEM solutions are locally quasi-optimal modulo low frequencies}, where ``low frequencies'' here means ``frequencies $\leq C k$ for some $C>1$''.

We now show (albeit heuristically) how this property can be inferred from the bounds \eqref{eq:intro1} \eqref{eq:intro2}, with this property illustrated by numerical experiments in \S\ref{sec:numerical}. 

The key point is that a bound on a high $k$-weighted Sobolev norm of a function in terms of a low $k$-weighted Sobolev norm, with the constant independent of $k$, implies that the function is controlled by its 
frequencies $\lesssim k$. This is illustrated by the following simple lemma. 
\es{
This lemma uses the $k$-weighted Fourier transform $\mathcal F_k$, defined by \eqref{eq:FT} below, with Fourier variable $\xi$. The weighting by $k$ implies that ``low frequencies'' are now $\{\xi: |\xi| \leq C\}$ for some $C>1$.}

\ble\mythmname{Bound on high Sobolev norm by low Sobolev norm implies function controlled by its low frequencies}
\label{lem:freq}
Suppose a family of functions $(f(k))_{k>0}$ is such that $f(k) \in H^1_k(\Rea^d)$ and there exists $C_2>0$ such that given $s>0$ there exists $C_1>0$ such that 
\beq\label{eq:model1}
\N{f}^2_{H^1_k(\Rea^d)}\leq C_1\Big( C_2 + \N{f}^2_{H^{-s}_k(\Rea^d)}\Big)\tfa k\geq k_0.
\eeq
If 
\beq\label{eq:freqCond}
\frac{C_1}{\es{\langle}R\es{\rangle}^{2(s+1)}}\leq \frac{1}{2},
\eeq
then, for all $k\geq k_0$,
\beqs
\bigg(\frac{k}{2\pi}\bigg)^d\int_{\Rea^d} \langle\xi\rangle^{2} |\mathcal{F}_k f(\xi)|^2 \, \rd \xi 
\leq 
2 C_1
\Bigg( C_2 + \bigg(\frac{k}{2\pi}\bigg)^d\int_{|\xi|\leq R} \langle \xi\rangle^{-2s} |\mathcal{F}_k f(\xi)|^2 \, \rd \xi \Bigg).
\eeqs
\ele

In the ``ideal'' situation that $C_1 = \es{\langle}C\es{\rangle}^{2(s+1)}$, \eqref{eq:freqCond} can be satisfied by taking $R> C$ and $s$ large; we call this the ``ideal" situation, since \eqref{eq:model1} holds with this value of $C_1$ (and $C_2=0$) when $|\mathcal{F}_k f(\xi)|= 0$ for $\es{|} \xi\es{|} \geq C$.

\

\bpf[Proof of Lemma \ref{lem:freq}]
By the definition of $\|\cdot\|_{H^s_k(\Rea^d)}$ \eqref{eq:Sobolev}, \es{\eqref{eq:model1} implies that}
\begin{align*}
&\int_{\Rea^d} \langle \xi\rangle^{2} |\mathcal{F}_k f(\xi)|^2 \, \rd \xi \\
&\qquad\leq C_1 \bigg(  \bigg(\frac{2\pi}{k}\bigg)^d C_2  + 
\int_{|\xi|\leq R} \langle \xi\rangle^{-2s} |\mathcal{F}_k f(\xi)|^2 \, \rd \xi
+\frac{1}{\es{\langle}R\es{\rangle}^{2(s+1)}}\int_{|\xi|\geq R} \langle \xi\rangle^{2} |\mathcal{F}_k f(\xi)|^2 \, \rd \xi \bigg),
\end{align*}
and the result follows.
\epf

\

The bound \eqref{eq:intro2} is conceptually similar to the setting of Lemma \ref{lem:freq} of a high Sobolev norm of $u-u_h$ being bounded by one of its low Sobolev norms, except that (i) 
the norm on the right-hand side is over a slightly larger domain ($\Omega_1$ vs $\Omega_0$) 
(ii) the order of the negative norm (i.e., $s$ in \eqref{eq:model1}) cannot be arbitrarily large.

Nevertheless we expect from \eqref{eq:intro2} that 
$u-u_h$ is controlled locally by the local best approximation error and the low frequencies of $u-u_h$; i.e., the Galerkin solution is locally quasi-optimal, modulo low frequencies.

\subsection{Outline of the paper.}

Section \ref{sec:numerical} contains numerical experiments illustrating 
local quasi-optimality modulo low frequencies of Helmholtz FEM solutions.

Section \ref{sec:assumptions} describes a general framework, with Section \ref{sec:examples} then showing how a variety of Helmholtz problems fit in this framework (including the exterior Dirichlet problem and transmission problem discussed above). 
 
  Section \ref{sec:statement} states the main results applied to the general framework, with Theorem \ref{thm:WDGS1} the generalisation of Theorem \ref{thm:intro1} and Theorem \ref{thm:WDGS2} the generalisation of Theorem \ref{thm:intro2}

Section \ref{sec:proofs} proves Theorems \ref{thm:intro1} and \ref{thm:intro2}, with Section \ref{sec:Caccioppoli} proving auxiliary results (Caccioppoli estimates) used in the proofs.

The appendix (\S\ref{sec:norms}) recaps the definitions of Sobolev spaces weighted by $k$.

\section{Numerical experiments illustrating local quasi-optimality modulo low frequencies of Helmholtz FEM solutions}
\label{sec:numerical}

This section presents numerical experiments where the error in the Helmholtz FEM solution behaves differently in different subsets of the domain due to either 
\ben
\item a non-uniform mesh, see \S\ref{sec:num:mesh},  or 
\item the solution being zero in some part of the domain (so that the local best approximation error is immediately zero in this part of the domain), see \S\ref{sec:num:source}. 
\een
These situations are manufactured so that each of the two terms on the right-hand side of \eqref{eq:intro2} is dominant in a different subset of the domain. 

Our ultimate aim is to study the local behaviour of the FEM error for scattering problems. 
However, a proper understanding of this situation requires understanding the $k$-dependence of the two terms on the right-hand side of \eqref{eq:intro2} as a function of the position of $\Omega_1$ relative to the scatterer; this is work in progress and will be reported elsewhere. 

In all the experiments, the Galerkin equations \eqref{eq:varfEDPh} are formulated with the software FreeFem++ \cite{He:12} using continuous Lagrange elements of degree $p = 1,\ldots, 4$. The resulting linear systems are then solved using the parallel domain decomposition toolbox HPDDM. The code used to produce these numerical results is available at \url{https://github.com/MartinAverseng/local_qo_experiments}.

In the experiments we compute ``low'' and ``high" frequency components of the FEM error; how we do this is described in \S\ref{sec:lowPassFilter} below, with the ``low'' and ``high'' frequencies corresponding to the components of the solution with the absolute value of the \es{(unweighted)} Fourier variable $\leq 2k$ and $\geq 2k$, respectively.

\subsection{Experiments with a non-uniform mesh}\label{sec:num:mesh}

We solve the interior impedance problem in a rectangular domain $\Omega = [0,2.1]\times[0,1]$, i.e., 
\beqs
	(k^{-2}\Delta+1) u =0 \quad \tin \Omega \quad\tand\quad k^{-1}\partial_n u - \ri u = g \quad\ton \partial\Omega,
\eeqs
with data $g$ is chosen so that the exact solution is the plane wave
$u = \exp\left(\ri k (\cos(\theta) x + \sin(\theta)y)\right)$.

This problem falls into the class of Helmholtz problems described in \S\ref{sec:scattering}, with 
the sesquilinear form given by \eqref{eq:sesqui2} with $\Ascat\equiv I$, $\cscat\equiv1$, $\Omega_-=\emptyset$ (no impenetrable obstacle), $\Ot=\emptyset$ (no penetrable obstacle), and $\Omega_{\tr}= \Omega$.

\paragraph{Experiment 1.} We consider two different mesh sizes $h_1> h_2$ and consider the following three different meshes on $\Omega$:
\bit
\item a globally-uniform mesh of size $h_1$, 
\item a globally-uniform mesh of size $h_2$, 
\item a mesh with 
sizes $h_1$ in the left-hand square $[0,1]\times[0,1]$, $h_2$ in the right-hand square $[1.1,2.1]\times[0,1]$, and some non-constant mesh size in the transition region $[1,1.1]\times[0,1]$.\footnote{More specifically, the mesh is designed by first meshing the boundaries of the two squares and transition region, which is partitioned into a total of $10$ segments. On the $4$ segments bounding the left-hand (respectively right-hand) square, the mesh size is uniform equal to $h_1$ (respectively $h_2$) and on the two segments on top and bottom of the transition region, the mesh size is constant equal to $\sqrt{h_1h_2}$. The command {\tt buildmesh} from FreeFem++ then generates a mesh for the global domain respecting the boundary mesh. As a result, the mesh size is uniform equal to $h_1$ in the left-hand square, $h_2$ in the right-hand square, and non-uniform in the transition region.}
\eit 
Figure \ref{fig:meshH1H2} shows an example of the third type of mesh.

\begin{figure}[H]
	\centering
	\includegraphics[width=0.5\textwidth]{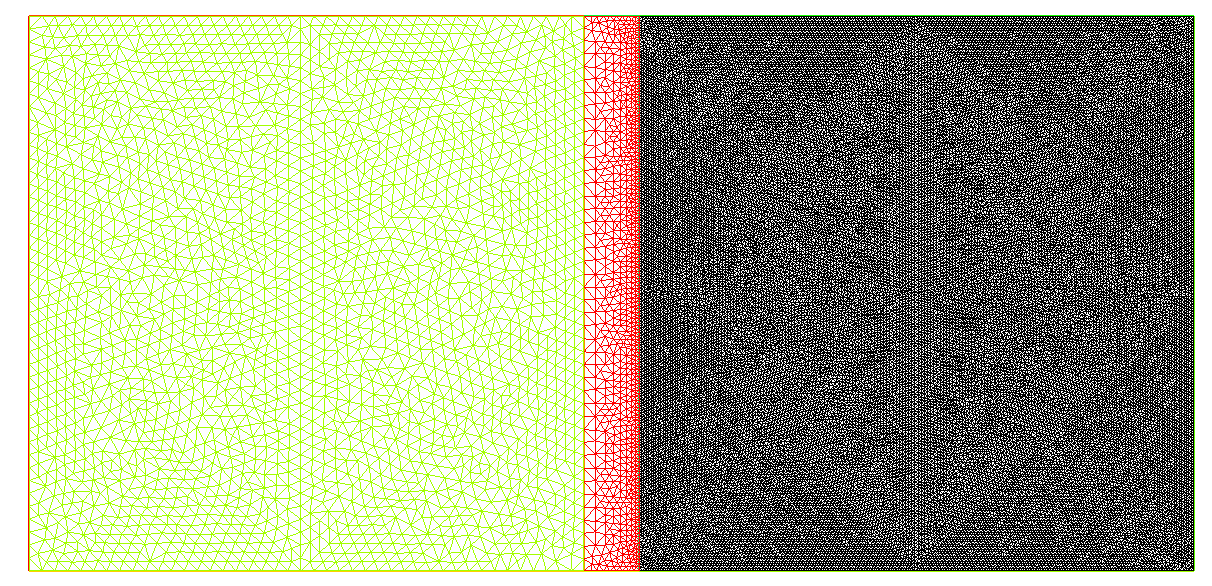}
	\caption{A mesh with two square regions $S_1$ and $S_2$ each with a uniform mesh size (different to each other).	\label{fig:meshH1H2}}
\end{figure}

\begin{figure}[p]
	\centering
	\includegraphics[width=0.45\textwidth]{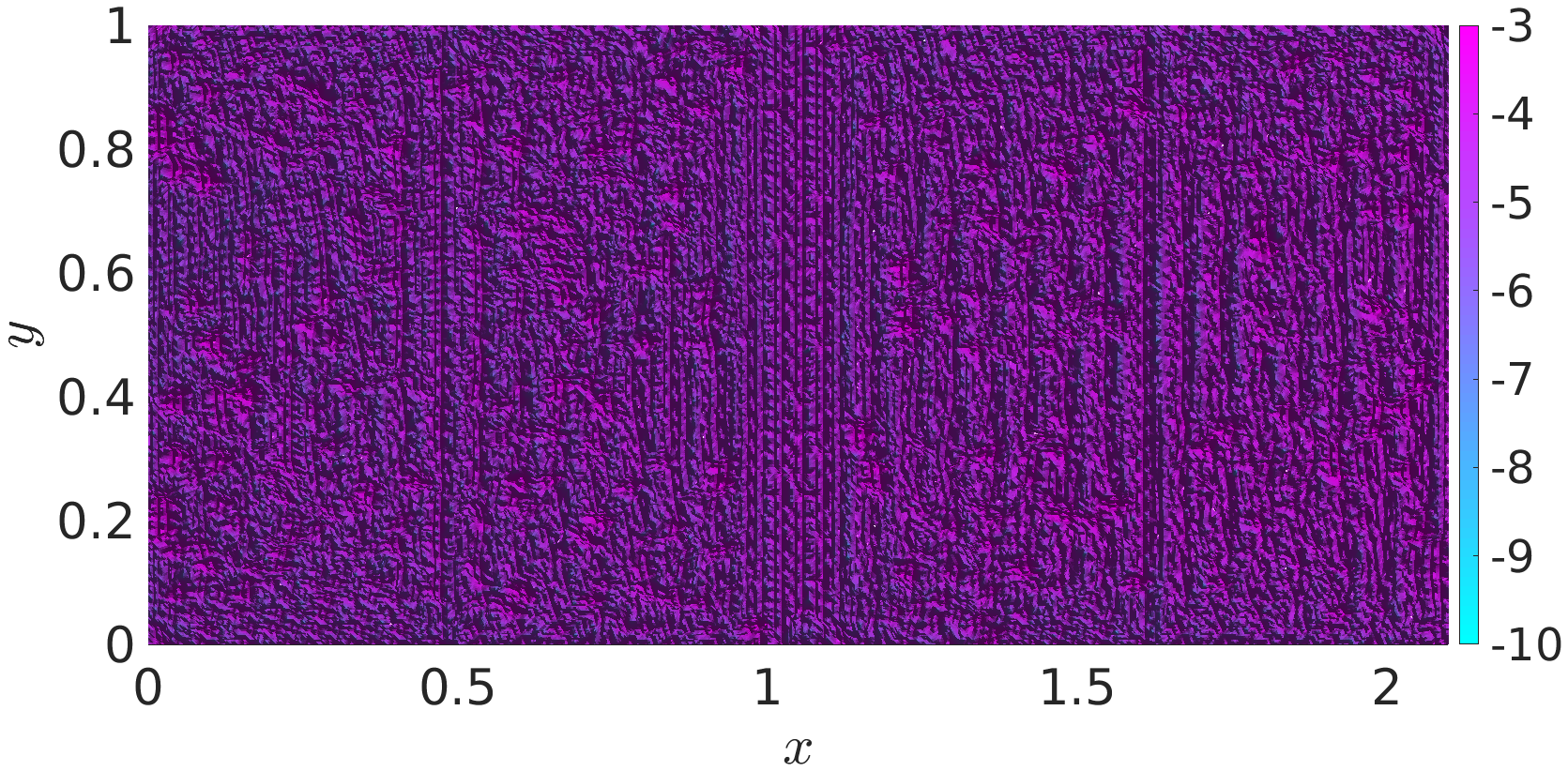}
	\includegraphics[width=0.45\textwidth]{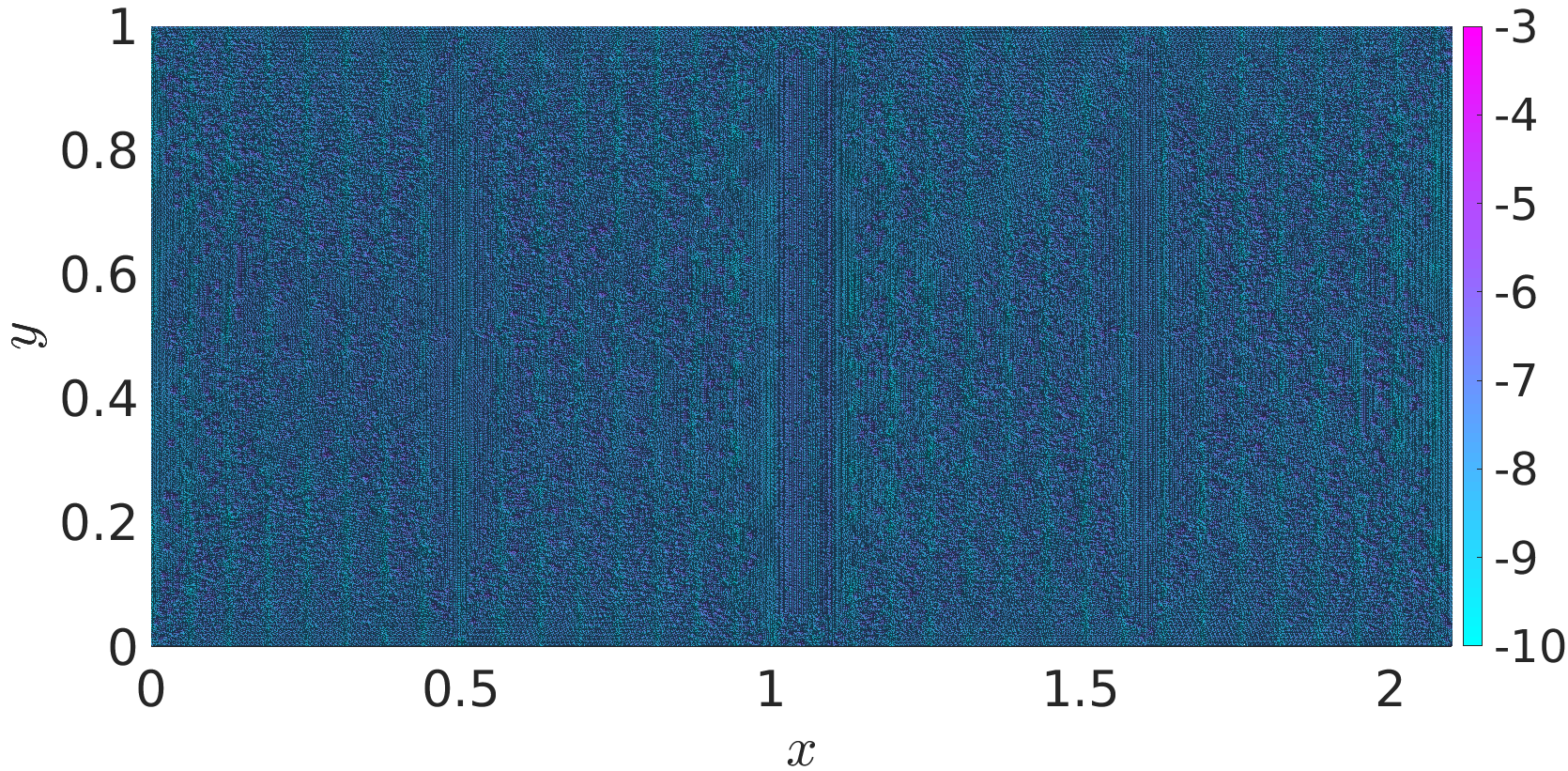}
	\includegraphics[width=0.9\textwidth]{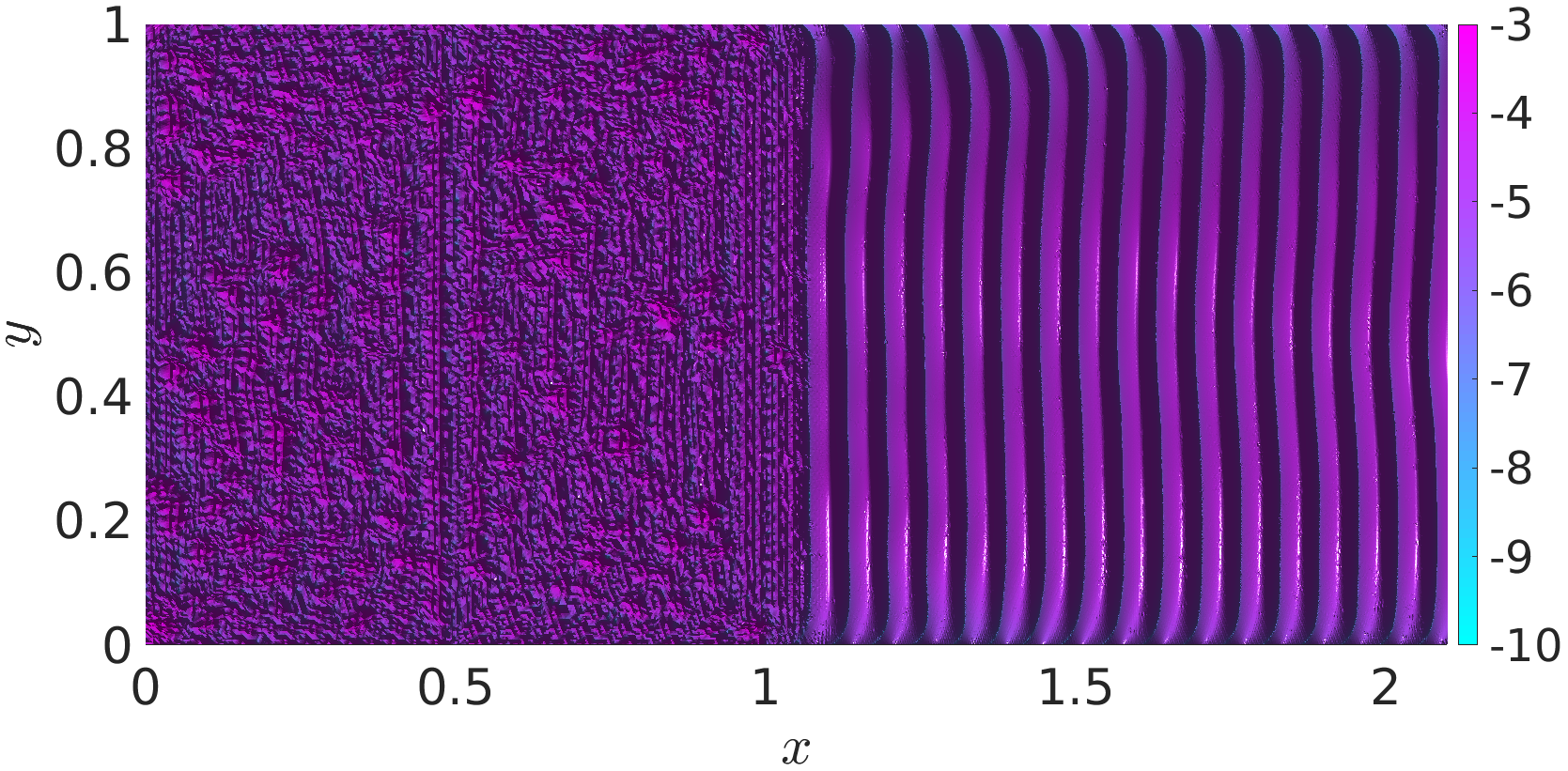}
	\caption{For Experiment 1 in \S\ref{sec:num:mesh}, plot of the quantity $\log(10^{-12} + \abs{\Re(u - u_h)})/\log(10)$, for $k = 50$, $p = 4$ and $\theta = 0$. Top left:~globally uniform mesh with $h =  1.78k^{-1}$. Top right:~globally uniform mesh with $h = 0.38k^{-1}$. Bottom:~non-uniform mesh, with a coarse region $h_1= 1.78k^{-1}$ and a fine region with $h_2 = 0.38k^{-1}$. }
	\label{fig:LFpollutionHorizontal}
\end{figure}

\begin{figure}[p]
	\centering
	\includegraphics[width=0.45\textwidth]{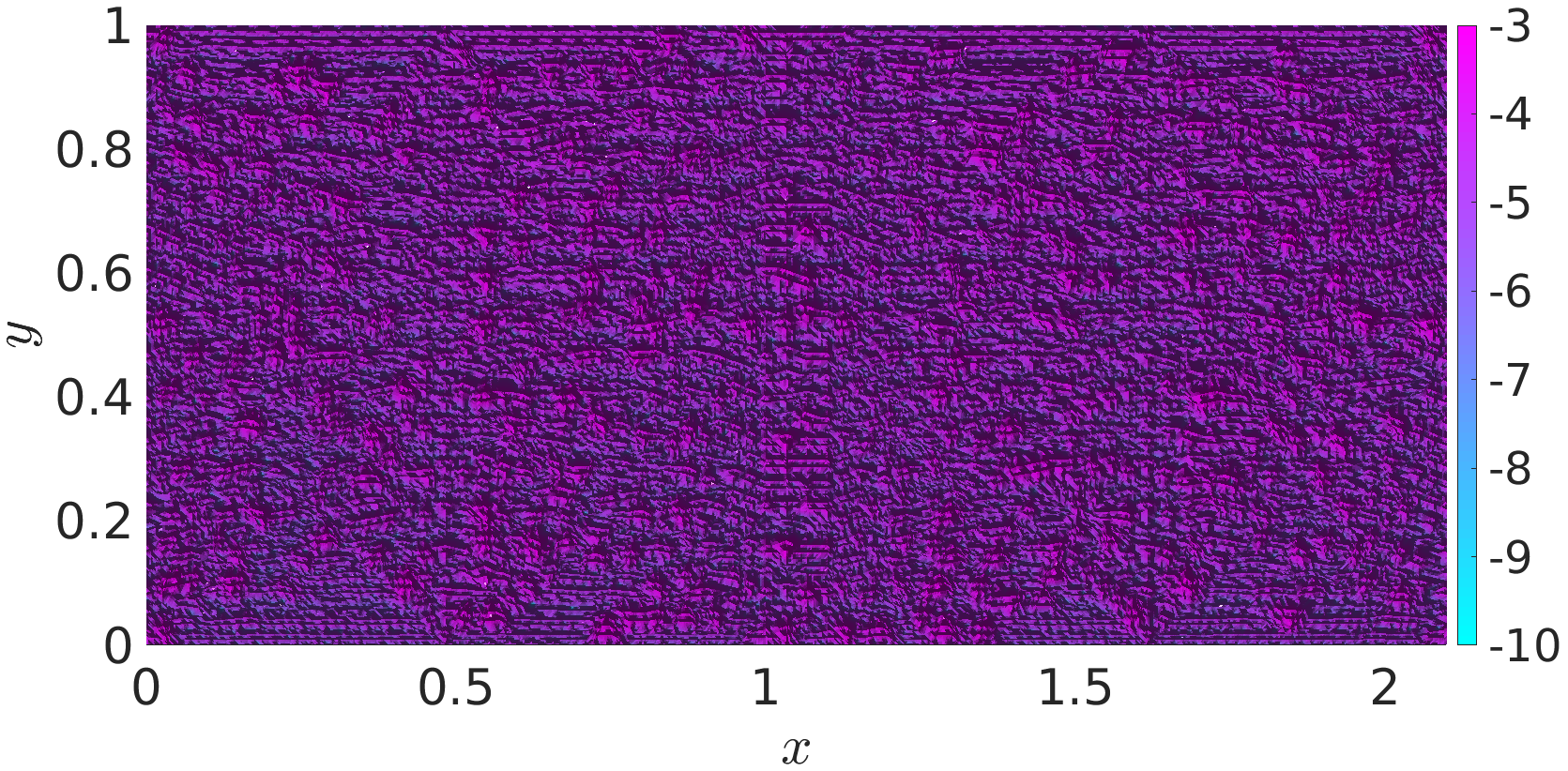}
	\includegraphics[width=0.45\textwidth]{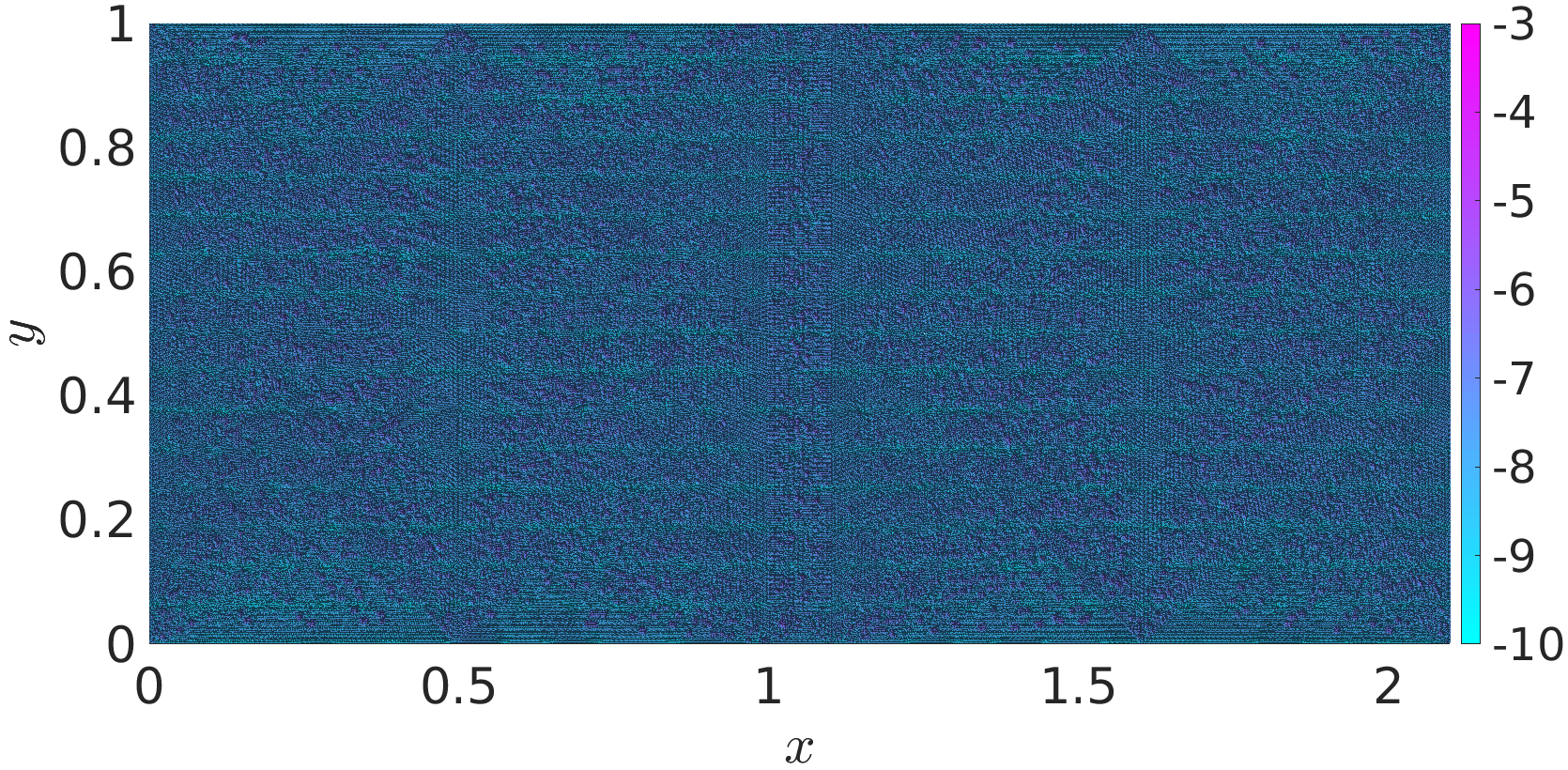}
	\includegraphics[width=0.9\textwidth]{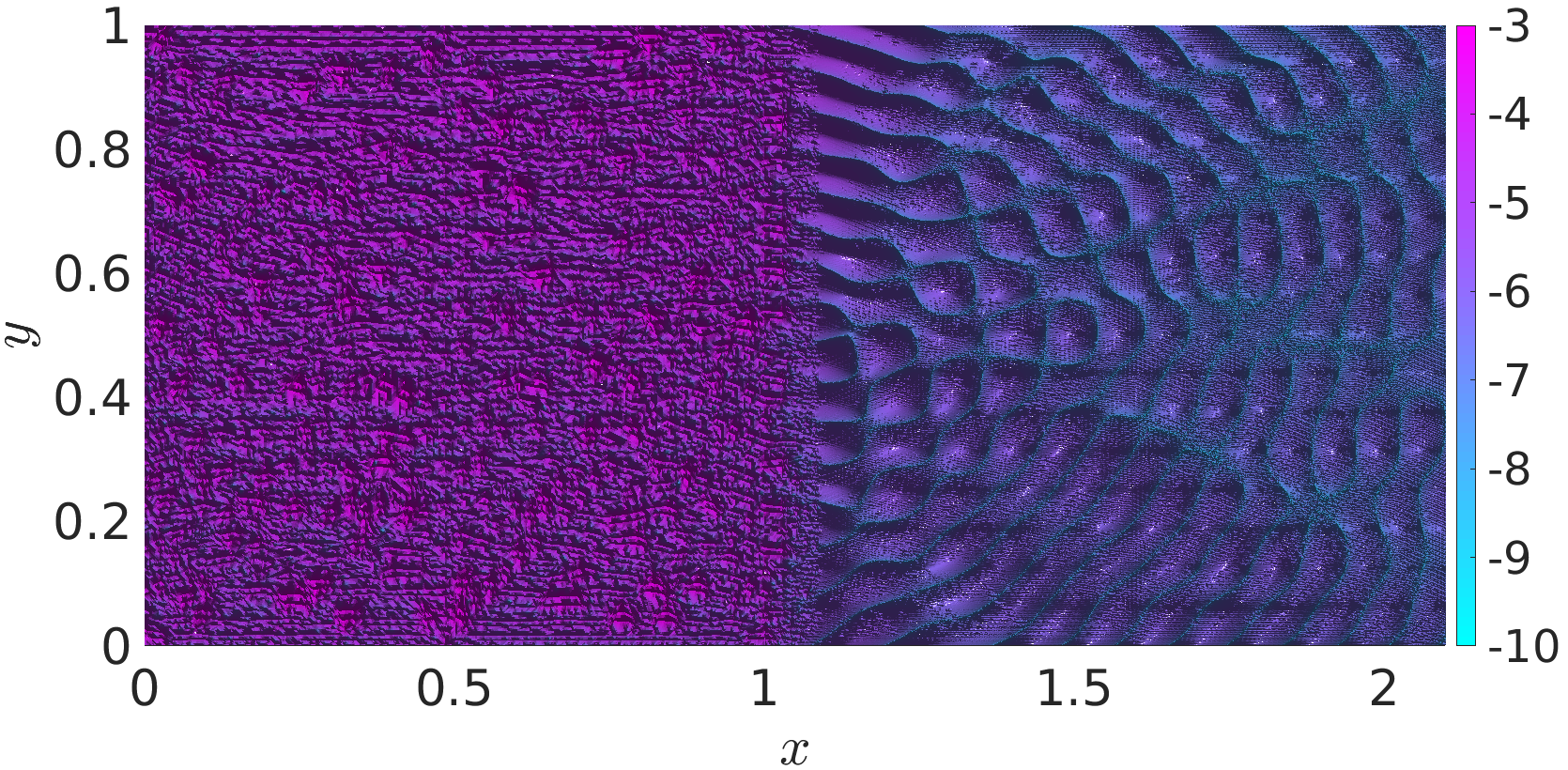}
	\caption{Same as Figure \ref{fig:LFpollutionHorizontal} (including all parameter values) but with $\theta = \pi/2$.}
	\label{fig:LFpollutionVertical}
\end{figure}

Figures \ref{fig:LFpollutionHorizontal} and \ref{fig:LFpollutionVertical} plot the FEM error $u-u_h$ for all three meshes with $k=50$ \es{(i.e., the wavelength $\approx 0.13$)}, $p=4$, $h_1= 1.78 k^{-1}$, $h_2= 0.38 k^{-1}$, $\theta=0$ in Figure \ref{fig:LFpollutionHorizontal}, and $\theta=\pi/2$ in Figure \ref{fig:LFpollutionVertical}. The errors are plotted on a logarithmic scale, with the scale kept the same for all the plots.

In both Figures \ref{fig:LFpollutionHorizontal} and \ref{fig:LFpollutionVertical}, for the third mesh the error in the right-hand square (with mesh width $h_2$) is between 10 to 100 times larger than the error in the right-hand square for the second (globally $h_2$) mesh; i.e., in the right-hand square for the third mesh, the error is dominated by the ``slush'' term on the right-hand side of \eqref{eq:intro2} and not the local best-approximation error. Furthermore, by eye, this ``slush'' error is low frequency.

The difference between Figures \ref{fig:LFpollutionHorizontal} and \ref{fig:LFpollutionVertical} is that in the first, the exact solution propagates from left to right, whereas in the second the exact solution propagates upwards. 
These figures show that the error is affected by the direction of propagation of the solution (or more precisely, its localisation in Fourier space), but does not  ``inherit'' these properties; indeed, in 
Figure \ref{fig:LFpollutionVertical} the error still propagates from left to right, even though the exact solution propagates upwards.

\paragraph{Experiment 2.} 
This experiments considers the non-uniform mesh from Experiment 1 with 
$h_1$ and $h_2$ now chosen to decrease with $k$ at different rates, with the FEM error computed at a sequence of values of $k$. We do this so that the terms on the right-hand side of \eqref{eq:intro2} have different $k$-dependence in the left and right squares.

We take $p=2$, $h_1= \sqrt{2}/k$, $h_2 = 1/k^{3/2}$, and $k\in [5,60]$, and 
Figures \ref{fig:sequenceImpedLeft} and \ref{fig:sequenceImped} plot the $H^1_k$ norm of the error on the left and right, respectively, as well as  their high- and low-frequency components (see \S\ref{sec:lowPassFilter} below for how these components are computed).

\begin{figure}[H]
	\centering
	\includegraphics[width=0.7\textwidth]{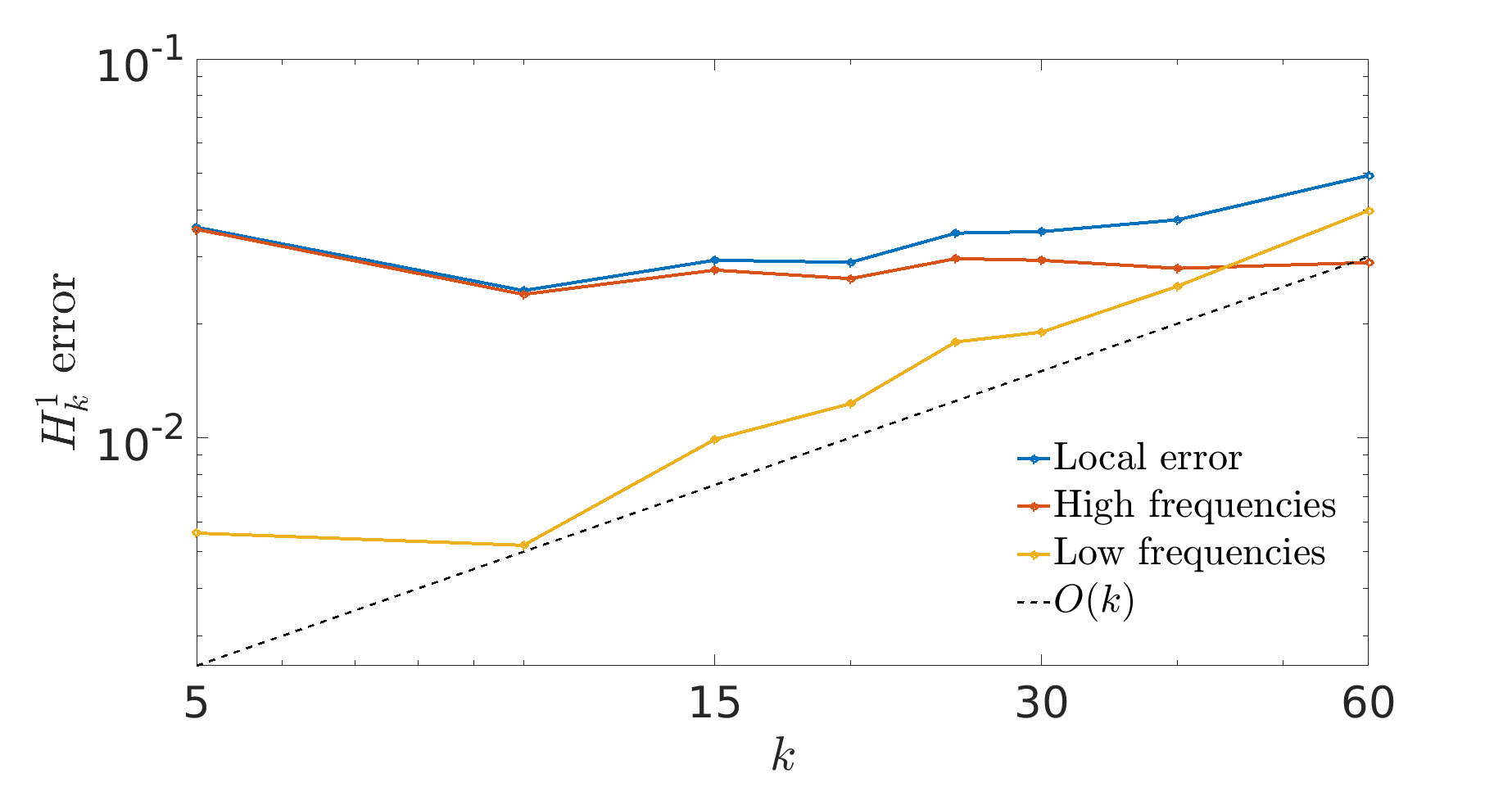}
	\caption{
	For Experiment 2 in \S\ref{sec:num:mesh}, plots of the $H^1_k$ error and its low- and high-frequency components in the left square (i.e., the square with the coarser mesh).} 
	\label{fig:sequenceImpedLeft}
\end{figure}

\begin{figure}[H]
	\centering
	\includegraphics[width=0.7\textwidth]{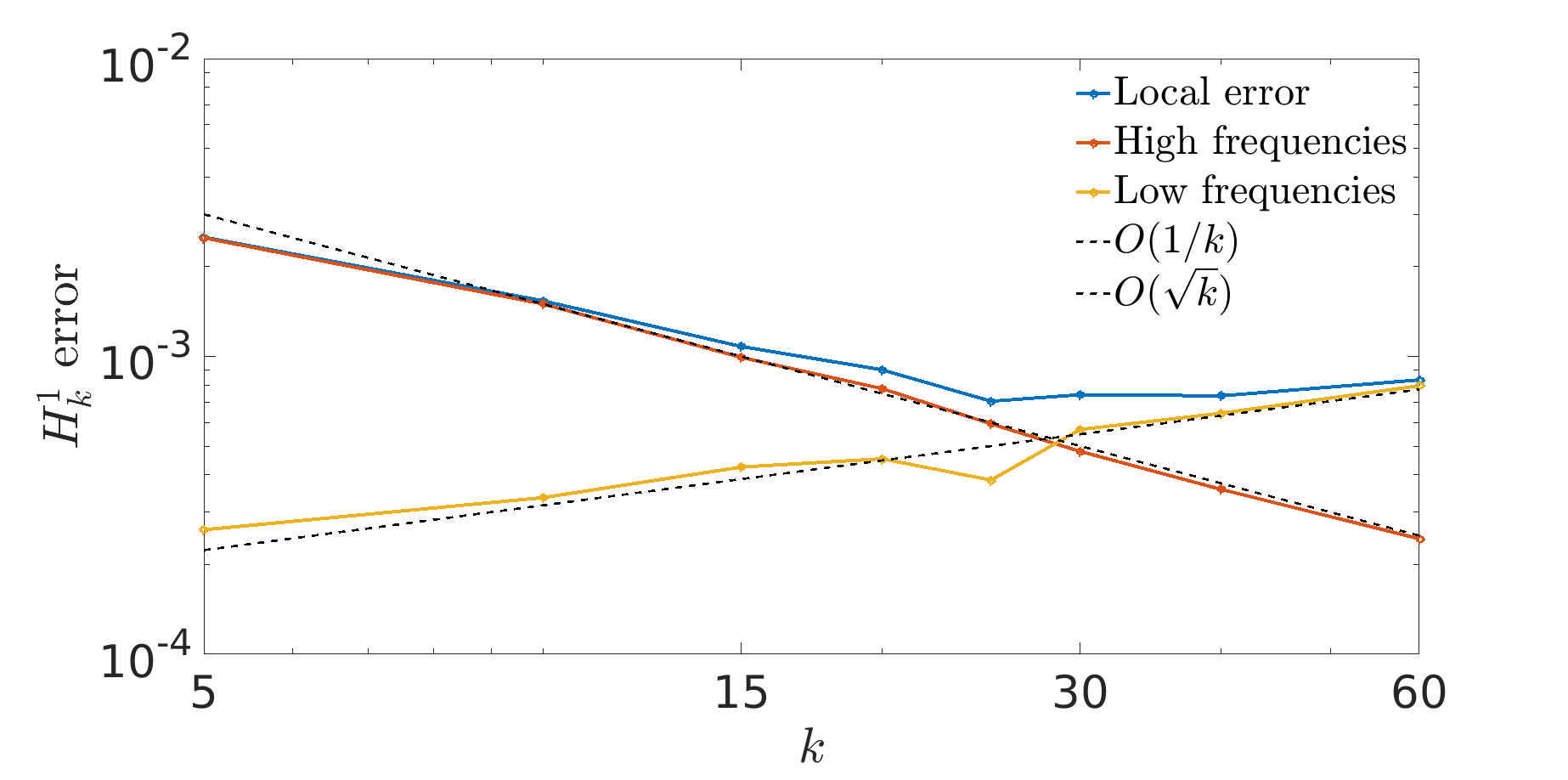}
	\caption{
		For Experiment 2 in \S\ref{sec:num:mesh}, plots of the $H^1_k$ error and its low- and high-frequency components in the right square (i.e., the square with the finer mesh).} 
	\label{fig:sequenceImped}
\end{figure}

The key point is that, \emph{on both the left and right, the FEM solution is locally quasi-optimal, modulo low frequencies, with the low frequencies on the left caused by the pollution effect, and the low frequencies on the right coming from the ``slush'' term propagating from the left.}

In more detail:~in Figure \ref{fig:sequenceImpedLeft} we see that 

(i) the error in the left square grows with $k$, 

(ii) the high frequencies of the error \es{in the left square} are roughly constant in $k$, and 

(iii) the growth in the error is caused by the low-frequency components, which are proportional to $k$. 

Points (i) and (iii) are expected \es{from} \cite[Corollary 3.2 and Point 1 in the following discussion]{IhBa:97} which proves that 
\beqs
\| u-u_h\|_{H^1_k}  \sim k(hk)^{2p} \|u\|_{H^1_k} \sim k,
\eeqs
for the 1-d impedance problem with $hk$ sufficiently small. 
 Point (ii) is expected because the local best approximation error on the left $\sim (hk)^p \| u\|_{H^{p+1}} \sim (hk)^2 \sim 1$; i.e., modulo low frequencies the FEM solution is locally quasi-optimal.

In Figure  \ref{fig:sequenceImped} we see that 

(iv) The high frequencies \es{of the error in the right square} decrease like $k^{-1}$.

(v) The low frequencies \es{of the error in the right square} grow like $k^{1/2}$. 

Point (iv) is expected since the local best approximation error on the right $\sim (hk)^p \| u\|_{H^{p+1}} \sim (hk)^2 \sim k^{-1}$. 

A heuristic argument for Point (v) is the following. Let $\chi$ be a compactly supported cutoff function supported on the left square, and let $e := \chi (u - u_h)$. 
For sufficiently large $k$, Points (i) to (iii) above show that the low frequencies dominate, and thus we assume that, first, 
$\mathcal{F}_k {e}(\xi)$ is negligible except  
\es{for $|\xi|\leq C$}
for some constant $C >1$ and, second, 
$\mathcal{F}_k{e}(\xi)$ is approximately evenly distributed in $\es{\{ \xi: |\xi| \leq C\}}$. 
Since the local $H^1_k$ error on the left $\sim k$, 
\[k^d\int_{\es{|\xi|\leq C}} \langle \xi\rangle^2 \abs{\mathcal{F}_k e(\xi)}^2 \,\rd\xi \sim k^2\]
which, \es{under the assumptions}, implies that $\abs{\mathcal{F}_k e(\xi)} \sim  k^{1-d/2}$ uniformly for 
\es{$|\xi|\leq C$}.
Within $\es{\{ \xi: |\xi| \leq C\}}$, only the components in \es{$\{ \xi: 1 - \eta/k \leq |\xi| \leq 1 + \eta/k\}$} propagate, 
where $\eta = O(1)$; this is because  the Helmholtz operator is (semiclassically) elliptic away from frequency $k$ \es{(see, e.g., the discussion in \cite[\S1.8]{GLSW1})}.
Therefore, the \es{squared} $H^1_k$ norm of the error propagated to the right square is approximately
\[
k^d\int_{\es{1 - \eta/k \leq |\xi| \leq 1 + \eta/k}}
 \langle \xi\rangle^2 \abs{\mathcal{F}_k e(\xi)}^2 \,\rd\xi 
\sim\es{k^2\int_{\es{1 - \eta/k \leq |\xi| \leq 1 + \eta/k}}
 \langle \xi\rangle^2  \,\rd\xi 
\sim k^2 \frac{\eta}{k}}
 \sim k; 
\]
i.e., Point (v).

\subsection{Experiments with an artificial source term}\label{sec:num:source}

\paragraph{Geometry of the obstacles considered in the simulations}

\es{We work with the Helmholtz problem in \S\ref{sec:introHelmholtz}, i.e., the exterior Dirichlet problem, with $A_{\rm scat} \equiv I$ and $c_{\rm scat} \equiv 1$.}
We consider the following four obstacles $\Omega_-$.  
\begin{itemize}
	\item No obstacle i.e. $\Omega_- = \emptyset$.
	\item Two flat mirrors, as shown in Figure \ref{fig:flat mirrors}. 
 	\item One flat mirror, i.e., the obstacle in Figure \ref{fig:flat mirrors} with the left mirror removed.
	\item Two curved mirrors with an inscribed ellipse, as shown in Figure \ref{fig:elliptic mirrors}. 
\end{itemize}

\begin{figure}[H]
	\centering
	\begin{subfigure}[t]{0.3\textwidth}
		\centering
		\includegraphics[height=4cm]{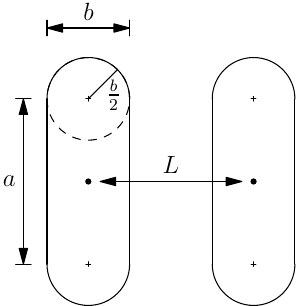}
		\caption{Flat mirrors}
		\label{fig:flat mirrors}
	\end{subfigure}
	\begin{subfigure}[t]{0.3\textwidth}
		\centering
		\includegraphics[height=4cm]{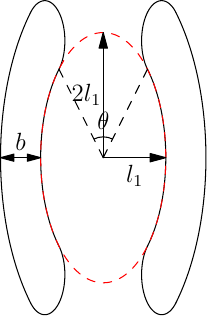}
		\caption{Elliptic mirrors}
		\label{fig:elliptic mirrors}
	\end{subfigure}
	\caption{Flat mirror and elliptic mirror geometries.}
\end{figure}
In all experiments, the computational domain is a disk of radius $1.5$, and the PML region is the annulus $\{1 \leq r \leq 1.5\}$. 
For the one flat mirror, $a=0.4, b=0.2,$ and $L=0.6$. 
For the two flat mirrors, $a=0.6, b=0.2,$ and $L=0.8$.
For the two curved mirrors, $l_1= 0.265, l_2 =2 l_1,$ $b=0.17$, and $\theta=\arctan(2 \tan (\pi/3))$.

\paragraph{Quasi-resonant wavenumbers for the trapping obstacles}

Whereas there is no trapping for the case of no obstacle or one flat mirror, the two flat mirrors and two curved mirrors trap rays. 
For these latter two geometries, the norm of the solution operator grows faster through an increasing sequence of $k$s (often called ``quasi resonances'') than the nontrapping solution operator. 

The experiments below for the two flat mirrors and two curved mirrors are conducted at quasi resonances for these obstacles.

For the flat mirrors, the quasi resonances are 
\[k_n := n\frac{\pi}{L - b}\,, \quad n=1,2,\ldots.\]
Indeed, for each of those frequencies, one can manufacture a ``quasi-mode" of the Laplace operator, in the form 
$u_n(x,y) = \chi(y) \psi(x)\sin(k_n x)$
where $\chi$ is a cutoff function supported in $(-a/2,a/2)$, and $\psi$ is a cutoff function with $\psi \equiv 1$ on $[-L/2+b/2, L/2-b/2]$ and supported in $[-L-b/2,L+b/2]$.

For the curved mirrors, the Helmholtz solution operator grows exponentially through
  the square roots of eigenvalues of the Laplace operator with Dirichlet conditions in the domain $U$ equal to the ellipse inscribed between the mirrors. 
By separation of variables, these quasi resonances 
  can be expressed as zeros of some special Mathieu functions (see \cite[Appendix E]{MGSS1}) and computed accordingly 
  (see, e.g.,\cite{wilson2007computing} and associated Matlab toolbox).

\paragraph{The artificial source term.}
Let $g: = k^{-2}\Delta u + u$, 
where 
\begin{equation*}
	u(x,y) := \chi(x) \chi(y) \exp(\ri k x)
\quad\tand\quad
\chi(x) = \begin{cases}
	0 & \textup{for } |x| \geq 0.1 \\
	\exp\left(\frac{5 x^2}{x^2 - 0.01}\right) & \textup{otherwise.}
\end{cases}
\end{equation*}
Observe that $u$ is supported in $\Omegasource := [-0.1,0.1]^2$. 

\paragraph{The FEM error.}

Figure \ref{fig:aliens} plots the FEM error (on a logarithmic scale) with $hk = 1$ and $p = 4$. For the two nontrapping obstacles (i.e., no obstacle and the one flat mirror) 
$k=100$. For the two trapping obstacles $k$ is taken to be the closest quasi resonance to $100$, namely $k= 104.72$ for the two flat mirrors and $k=95.838$ for the two curved mirrors.	

\begin{figure}[h]
	\centering
	\includegraphics[width=0.4\textwidth]{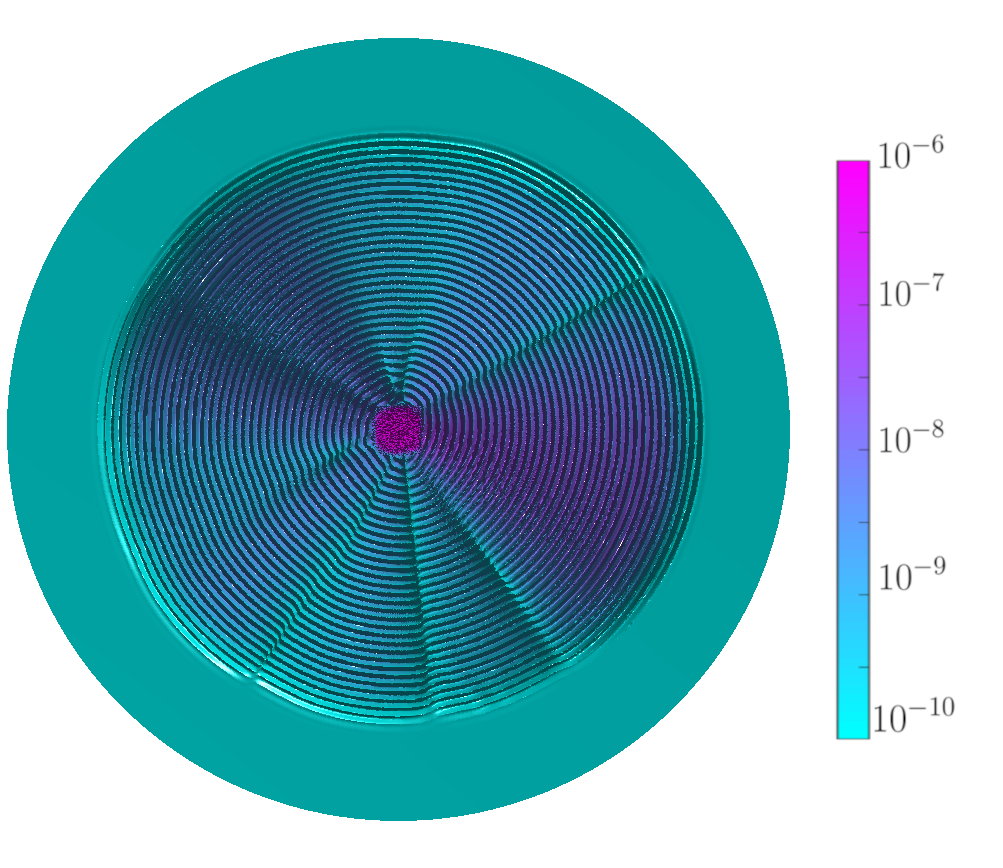}
	\includegraphics[width=0.4\textwidth]{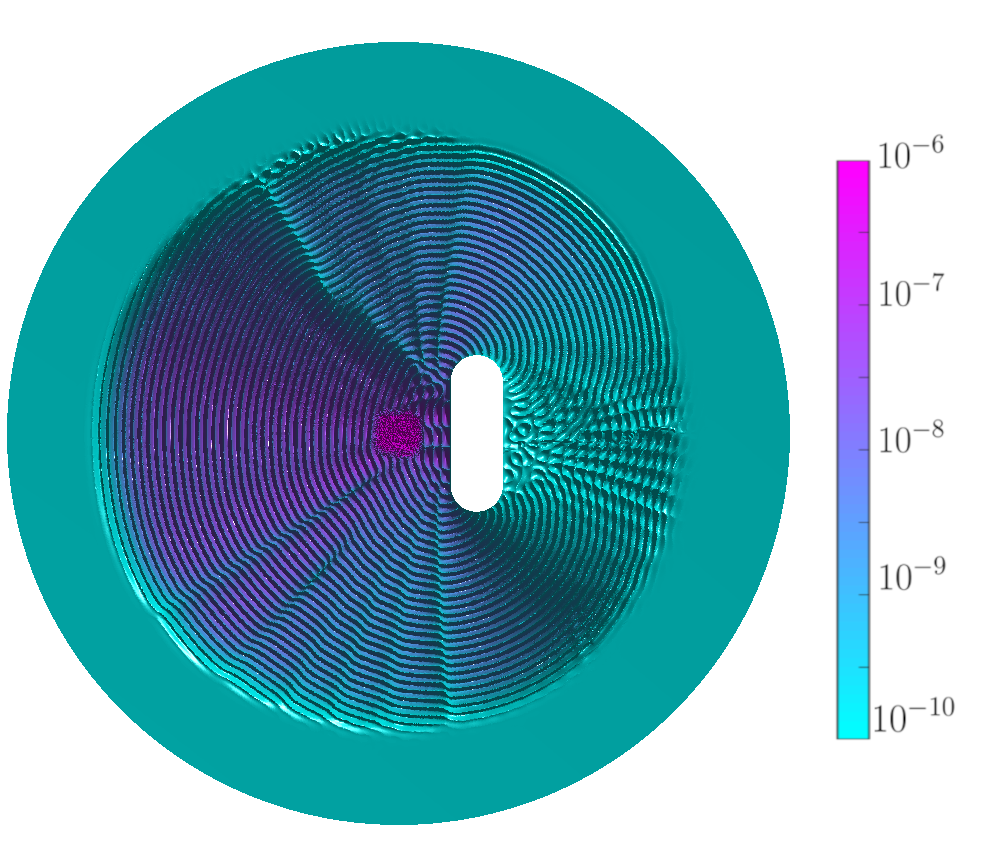}\\
	\includegraphics[width=0.4\textwidth]{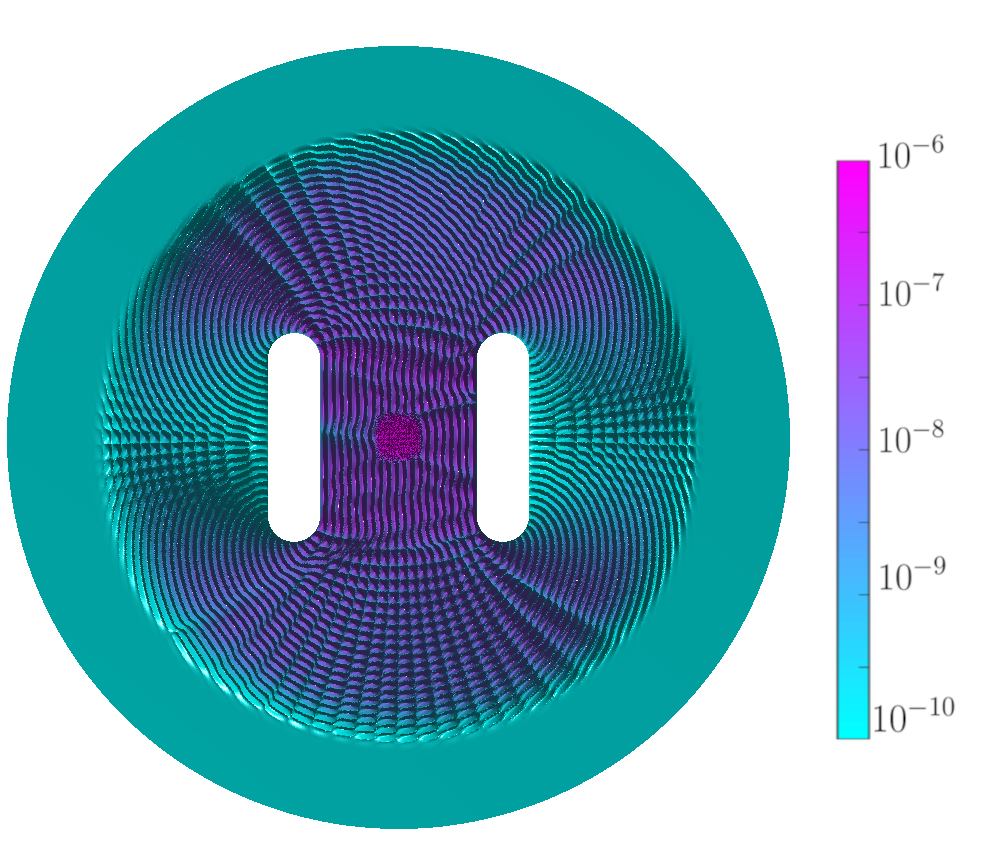}
	\includegraphics[width=0.4\textwidth]{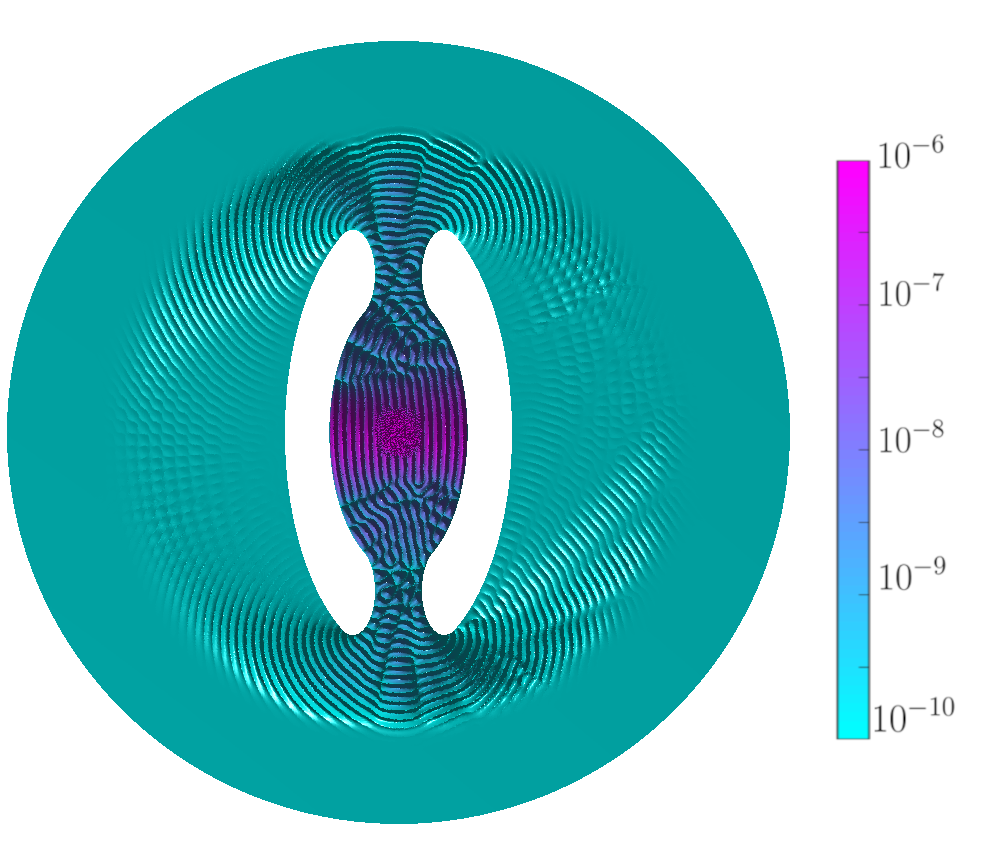}
	\caption{
	For the experiment in \S\ref{sec:num:source} with $k \approx 100$, $hk = 1$ and $p = 4$, the plots show $\log(\abs{\Re(u - u_h)})/\log(10)$ (the rationale for plotting the real part and not the absolute value is to give a sense of the wave-length). Plots obtained with FFMATLIB.}
	\label{fig:aliens}
\end{figure}

All four of the figures show a large high-frequency error on the support of $u$ (i.e., on $\Omegasource$) and a small low-frequency error away from the support of $u$. Since the best approximation error is zero away from the support of $u$, this is consistent with the claim that Helmholtz FEM solutions are 
quasi optimal  modulo low frequencies.

The four figures also show that, whereas the high-frequency error in $\Omegasource$ is roughly the same in all four cases, the low-frequency error is affected by the shape of the obstacles. Indeed, 
for the two trapping obstacles the ``slush" is larger inside the trapping regions than elsewhere, and for the three non-empty obstacles the ``slush'' is smaller in the shadow regions. 

To quantify the high- and low-frequency behaviour more precisely, for each obstacle we consider two boxes:~$\Omegasource$ and 
 a location away from the source (described for each obstacle in the caption of Table \ref{tab:rhoAway}). 
In each box we 
let $v_h = \chi (u-u_h)$, where $\chi$ is a cut-off function compactly supported in each box, and we compute the high- and low-frequency components of $v_h$ as described in \S\ref{sec:lowPassFilter} below.

For the ``one mirror'' obstacle, $v_h$ in the two different boxes is plotted in  Figure \ref{fig:illustrationDFT} for $k=50$. Figure \ref{fig:powerSpectra} plots the Discrete Fourier Transform (DFT) of $v_h$ in base $10$ log scale; these plots confirm that the error away from $\Omegasource$ (in the left plot) is dominated by low frequencies, compared to the error in $\Omegasource$ (in the right plot). 

\begin{figure}[h]
	\centering
	\includegraphics[width=0.45\textwidth]{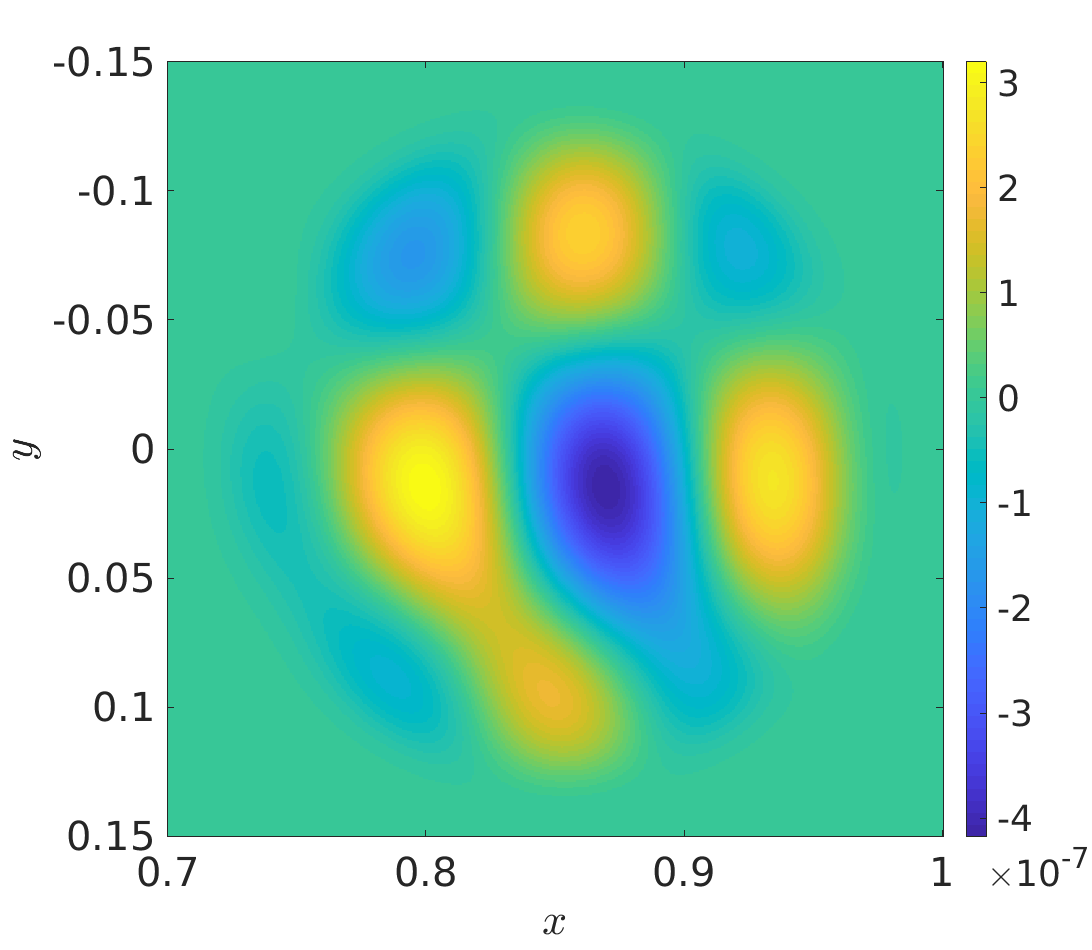}
	\includegraphics[width=0.45\textwidth]{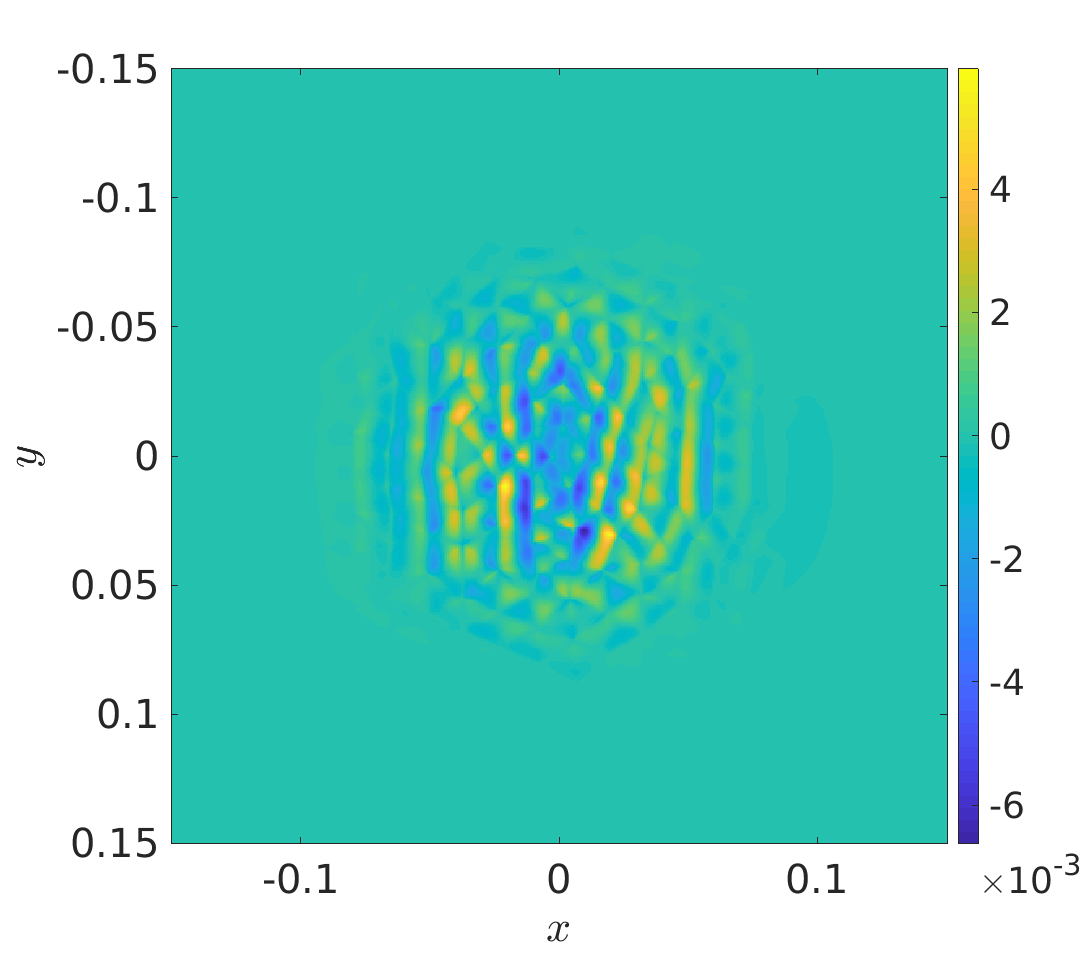}
	\caption{The signal $V$ corresponding to the scattering by ``one wall" (see Figure \ref{fig:aliens}, top right) with $k = 50$ and $\Delta = 10^{-3}$. Left: $x_0 = 0.7$, $y_0 = -0.15$ (thus, the box is on the right of the wall, and the error is ``pure slush". Right: $x_0 = y_0 = -0.15$ (thus, the box contains $\Omegasource$). }
	\label{fig:illustrationDFT}
\end{figure}

\begin{figure}[h]
	\centering \includegraphics[width=0.45\textwidth]{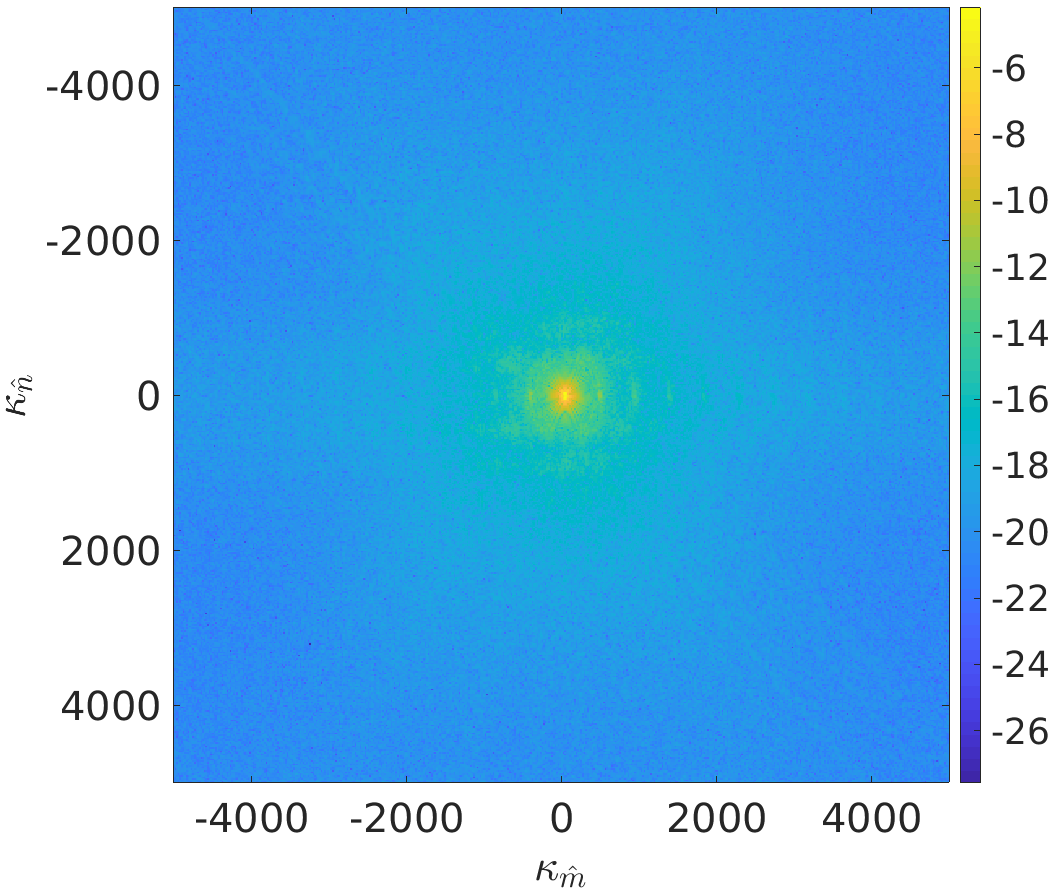}
	\includegraphics[width=0.45	\textwidth]{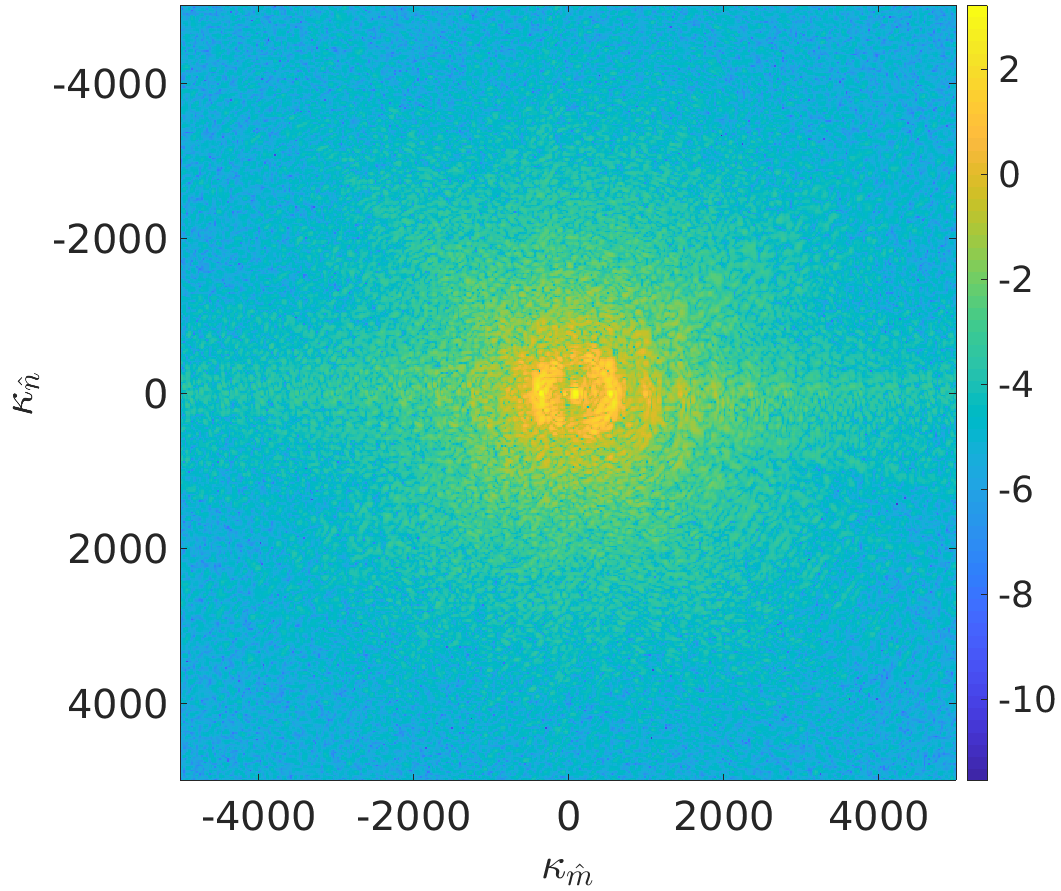}	
	\caption{Graph of $\log(|\widehat{V}|)$, where in each figure, $\widehat{V}$ is the Discrete Fourier Transform of the corresponding $V$ in Figure \ref{fig:illustrationDFT}. This confirms that error away from $\Omegasource$ (left) is dominated by low frequencies, compared to the error in  $\Omegasource$ (right).}
	\label{fig:powerSpectra}
\end{figure}

Tables \ref{tab:rhoAway} and \ref{tab:rhoBox} plot the quantity $\rho$ defined by \eqref{eq:defrho} below, which measures the proportion of the high-frequencies of $v_h$. As expected, the values of $\rho$ are much smaller in the ``away'' region than in $\Omegasource$.

\begin{table}[H]
	\centering
	\begin{tabular}{l|ccc}
	 Experiment	 & $k = 50$ & $k = 75$ & $k = 100$ \\\hline 
	 a) no obstacle  & 0.0593& 0.0165 &0.0091\\
	 b) one flat mirror & 0.0625& 0.0183& 0.0103\\
	 c) two flat mirrors & 0.0471 & 0.0251&0.0122\\ 
	 d) two curved mirrors &0.0460 & 0.0236& 0.0186
	\end{tabular}
\caption{Values of $\rho$ (defined by \eqref{eq:defrho}) in experiments a)-d) for a box away from $\Omegasource$. For the experiments a)-b) , the box is defined by $x_0 = 0.7$, $y_0 = -0.15$. In c) and d), we use $x_0 = -0.15$ and $y_0 = 0.2$ (thus the boxes are inside the cavities). We take $\Delta = (20k)^{-1}$. In all cases, $N \Delta = 0.3$ (i.e., the box sides are 0.3). The polynomial degree of the FEM is set to $p=2$ in these experiments. 
}\label{tab:rhoAway}
\end{table}

\begin{table}[H]
	\centering
	\begin{tabular}{l|ccc}
		Experiment	 & $k = 50$ & $k = 75$ & $k = 100$ \\\hline
		a) no obstacle & 0.9379& 0.8421& 0.8142\\
		b) one flat mirror &0.9375 & 0.8700& 0.8001\\
		c) two flat mirrors & 0.9342& 0.8726& 0.8043\\ 
		d) two curved mirrors &0.9532& 0.9177&0.5817
	\end{tabular}
	\caption{Values of $\rho$ (defined by \eqref{eq:defrho}) in experiments a)-d) for the box $[-0.15,0.15]^2 \supset \Omegasource$.
	}
	\label{tab:rhoBox}
\end{table}

\subsection{Computing the high- and low-frequency components}

\label{sec:lowPassFilter}

To compute the high- and low-frequency components of a finite-element function $v_h$ restricted to a square, 
we compute the 2-dimensional discrete Fourier transform of the matrix ${V}$ defined by 
\[{V}_{m,n} = v_h(x_m,y_n)\,, \quad 0 \leq m,n \leq N-1\]
where $x_m = x_0 + m\Delta$, $y_n = y_0 + n\Delta$
and the sampling rate $\Delta^{-1}$ is chosen sufficiently large (in our tests, we use $\Delta = 
2\pi/(200 k)$ 
and compare with $\Delta = 2\pi/(400 k)$ 
 to confirm that the results are not very sensitive to this value). This interpolation from a triangular mesh to a Cartesian grid is performed with the help of the FFMATLIB toolbox.\footnote{\url{https://github.com/samplemaker/freefem_matlab_octave_plot/blob/public/README.md}} 

Let ${\widehat{V}}$ be the discrete Fourier Transform of ${V}$, that is 
\[\widehat{V}_{\widehat{m},\widehat{n}} = \sum_{m = 0}^N \sum_{n = 0}^N V_{m,n} \re^{- \frac{2\ri\pi}{N}(\widehat{m} m + \widehat{n} n)}\,,\]
so that
\[v_h(x_m,y_n)  = \frac{1}{N^2}\sum_{\widehat{m} = 1}^N \sum_{\widehat{n} = 1}^N \widehat{W}_{\widehat{m},\widehat{n}}  \re^{\ri \kappa_{\widehat{m}} x_m} \re^{\ri \kappa_{\widehat{n}} y_n}\,, \quad 0 \leq m,n \leq N-1\,, \]
with
\[\widehat{W}_{\widehat{m},\widehat{n}} = \widehat{V}_{\widehat{m},\widehat{n}} \re^{-\ri (\kappa_{\widehat{m}} x_0 + \kappa_{\widehat{n}} y_0)}\,, \quad \kappa_{\widehat{m}} = \frac{2\pi}{\Delta} \frac{\widehat{m}}{N}\,.\]
That $v_h$ is low frequency means that one can represent it accurately as a linear combination of waves of the form $\re^{\ri \kappa_1 x }\re^{\ri \kappa_2 y}$ with $\kappa_1, \kappa_2\lesssim k$. Here we check this by computing the discrete signal
\[\widetilde{v}_h(x_n,y_n) := \frac{1}{N^2} \sum_{\widehat{m} = 1}^{N} \sum_{\widehat{n} = 1}^{N} \widehat{H}_{\widehat{m}}\widehat{H}_{\widehat{n}}\widehat{W}_{\widehat{m},\widehat{n}} \re^{\ri \kappa_{\widehat{n}} x_n } \re^{\ri \kappa_{\widehat{m}} y_n}\,,\]
where $H$ is a ``low-pass filter", i.e.
\[\widehat{H}_{\widehat{m}} = \begin{cases}
	1 & \textup{if } \kappa_{\widehat{m}} \leq \alpha k \textup{ or } \frac{2\pi}{\Delta} - \kappa_{\widehat{m}} \leq \alpha k\,, \\
	0 & \textup{otherwise.}
\end{cases}\]
Note that, by periodicity, for all $0 \leq m, \widehat{m} \leq N-1$, one has 
\[\re^{\ri \kappa_{\widehat m} x_m} = \re^{-\ri (\frac{2\pi}{\Delta} - \kappa_{\widehat m}) x_m}\,,\]
hence $\widehat{H}$ effectively removes all frequencies outside the interval $[-\alpha k,\alpha k]$. Here, $\alpha$ is a parameter, set to $2$ in our tests. 

The relative $l^2$ norm of the high-frequency components of $v_h$ is 
\begin{equation}
	\label{eq:defrho}
	\rho := \frac{\norm{v_h - \widetilde{v}_h}_{l^2}}{\norm{v_h}_{l^2}}\,,
\end{equation}
where $\norm{\cdot}_{l^2}$ is the discrete $l^2$ norm. By Parseval's theorem for discrete Fourier transforms, this ratio is equal to
\[\rho = \frac{\sum_{\widehat{m},\widehat{n} = 1}^N (1 - \widehat{H}_{\hat m} \widehat{H}_{\widehat{n}})^2\abs{\widehat{V}_{\widehat{m},\widehat{n}}}^2}{\sum_{\widehat{m},\widehat{n} = 1}^{N} \abs{\widehat{V}_{\widehat{m},\widehat{n}}}^2}\,.\]

\section{Abstract framework}\label{sec:assumptions}

\subsection{Function spaces}\label{sec:function}

Given a Hilbert space $\mathcal{Y}$, let $\mathcal{Y}^*$ be the Hilbert space of bounded {\em anti-linear} functionals $G: \mathcal{Y} \to \mathbb{C}$ (i.e. $G(\lambda y) = \overline{\lambda}G(y)$) equipped with the norm
\beq\label{eq:dual}
\norm{G}_{\mathcal{Y}^*} := \sup_{y \in \mathcal{Y} \setminus \{0\}} \frac{|G(y)|}{\norm{y}_{\mathcal{Y}}}.
\eeq
Let $\Omega \subset \R^d$ be a bounded open set and let $\mathcal{H} := L^2(\Omega)$. As usual, $\mathcal{H}$ is identified with its dual $\mathcal{H}^*$. Furthermore, let $k_0 > 0$ and let $(\mathcal{Z}_k)_{k \geq k_0}$ be a family of Hilbert spaces such that for each $k \geq k_0$, the inclusion
$\mathcal{Z}_k\subset \mathcal{H}$ is dense.
In all the examples below, $\cZ_k$ is either $H^1_k(\Omega)$ or this space with a \es{zero} Dirichlet boundary condition prescribed on a subset of its boundary. 
To quantify abstractly the ``regularity" of elements of $\mathcal{H}$, \es{given $\ell \in \mathbb{N}$}, we introduce a scale of Hilbert spaces $(\mathcal{Z}_{k}^j)_{0 \leq j \leq \ell+2}$ with  $\mathcal{Z}_k^0 =\mathcal{H}$, $\mathcal{Z}_k^1 =\mathcal{Z}_k$, and with dense inclusions $\mathcal{Z}_k^j\subset \mathcal{Z}_k^{j-1}$ for $j=1,\dots,\ell+2$.

For $j \geq 0$, each anti-linear functional  $g$ on $\mathcal{Z}_k^j$ also defines an anti-linear functional $g'$ on $\mathcal{Z}_k^{j'}$ for $j' \geq j$, via restriction since $\mathcal{Z}_k^{j'} \subset \mathcal{Z}_k^{j}$. This restriction map is moreover continuous and injective, by density of the previous inclusions and hence we identify $g$ and $g'$. This identification is compatible with the identification of $\mathcal{H}$ to its dual, and gives the chain of continuous and dense inclusions
\begin{equation}
	\label{eq:embeddings}
	(\mathcal{Z}_k^{\ell +1})^* \supset (\mathcal{Z}_k^{\ell })^* \supset \ldots \supset (\mathcal{Z}_k^{1})^* \supset \underbrace{\mathcal{Z}_k^0}_{= \mathcal{H}} \supset   \underbrace{\mathcal{Z}_k^1}_{= \mathcal{Z}_k} \supset \ldots \supset \mathcal{Z}_k^{\ell} \supset \mathcal{Z}_{k}^{\ell + 1}\,.
\end{equation}
We assume that there exists $\Cemb > 0$ such that, for all $0 \leq j \leq j' \leq \ell + 2$,
\begin{equation}
	\label{eq:continuous_embed}
	\norm{u}_{\mathcal{Z}_k^j} \leq \Cemb \norm{u}_{\mathcal{Z}^{j'}_k}\quad \tfa \es{u \in \mathcal{Z}_k^{j'} \tand} k \geq k_0.
\end{equation}
which also implies that $\norm{u}_{(\mathcal{Z}_k^{j'})^*} \leq C_{\rm emb} \norm{u}_{(\mathcal{Z}_k^{j})^*}$ for $0 \leq j \leq j' \leq \ell + 2$. 

For technical reasons (to be able to treat transmission problems), we introduce another scale $(\mathcal{W}_k^j)_{0 \leq j \leq \ell+1}$ of Hilbert spaces with the property that $\mathcal{Z}_k^j \subset \mathcal{W}_k^j \subset \mathcal{H}$ with continuous inclusions (in particular $\mathcal{W}_k^0 = \mathcal{H}$), and, for all 0 $\leq j\leq  \ell + 2$,
\begin{equation}
\label{e:embedW}
\norm{u}_{\mathcal{W}_k^j} \leq \Cemb \norm{u}_{\mathcal{Z}^{j}_k}\quad \tfa \es{u \in \mathcal{Z}_k^j \tand} k \geq k_0.
\end{equation}

\begin{example}
For the Helmholtz transmission problem with the outgoing condition approximated by a perfectly-matched layer, 
\beqs
\cZ_k = H^1_0(\Omega), \quad \cW^j_k = 
L^2(\Omega)\cap \big( H^j(\Omega_{\rm in})
\oplus H^j(\Omega_{\rm out}\cap \Omega)
\big),
\quad\tand\quad
\cZ^j_k =\cW^j_k \cap H^1_0(\Omega),
\eeqs
where $\Omegain$ is the penetrable obstacle, $\Omegaout$ its exterior, and $\Omega$ the (truncated) computational domain containing $\Omegain$;   
see \S\ref{sec:spaces} below.
\end{example}

\subsection{Local properties}\label{sec:local}

\ble\label{lem:balls}
Let $B_j, j=1,\ldots,N$ be open sets, and let 
\beq
\label{eq:defC}
\es{\Cover}:= \max\Big\{ |J| : B_{j_1} \cap \ldots \cap B_{j_J} \neq \emptyset \text{ with $j_1,\ldots,j_J$ distinct}\Big\}.
\eeq
Then
\beq\label{eq:lem_balls1}
(\es{\Cover})^{-1} \sum_{j=1}^N \N{v}^2_{\cH(B_j)} \leq \N{v}^2_{\cH(\cup_{j=1}^N B_j)} \leq \sum_{j=1}^N \N{v}^2_{\cH(B_j)}
\eeq
for all $v \in \cH(\cup_{j=1}^N B_j)$.
\ele

The notation $B_j$ is used because we use Lemma \ref{lem:balls} below with the $B_j$ either balls or balls intersected with some larger (fixed) open set. 

\

\bpf[Proof of Lemma \ref{lem:balls}]
 The second inequality in \eqref{eq:lem_balls1} follows immediately from the fact 
that the $\cH$ norm is the $L^2$ norm.
Given $x \in \cup_{j=1}^N B_j$ let $m(x)$ be the number of distinct $B_1,\ldots,B_N$ that contain $x$. Then
\beqs
\sum_{j=1}^N \N{v}^2_{\cH(B_j)}= \int_{\cup_{j=1}^N B_j} m(x) |v(x)|^2 \, \rd x 
\eeqs
and the first inequality in \eqref{eq:lem_balls1} follows since $|m(x)|\leq \es{\Cover}$.
\epf

\begin{assumption}\label{ass:balls}
The following holds with $\cY$ equal to either $\cW$ or $\cZ$. 
Given open sets $B_1,\ldots,B_N$, with $\es{\Cover}$ as in \eqref{eq:defC}, 
\beq\label{eq:lem_balls2}
(\es{\Cover})^{-1} \sum_{j=1}^N \N{v}^2_{\cY_k^j(B_j)} \leq \N{v}^2_{\cY_k^j(\cup_{j=1}^N B_j)} \leq \sum_{j=1}^N \N{v}^2_{\cY_k^j(B_j)}
\eeq
for all $v \in \cY_k^j(\cup_{j=1}^N B_j)$.
Furthermore, if $\supp \,u\cap \supp \,v=\emptyset$, then
\beq\label{eq:ZNormDisjoint}
\|u+v\|_{\cY_k^j}^2=\|u\|_{\cY_k^j}^2+\|v\|_{\cY_k^j}^2. 
\eeq
\end{assumption}

With $\partial_<$ defined by~\eqref{eq:defDistInf}, for $U\subset \Omega$ open, let 
\begin{align}\nonumber
\mathcal{Z}_k^{j,<}(U) &:= \overline{\big\{v \in \mathcal{Z}_k^j \textup{ s.t. } \supp \,v\subset\overline{U},\,\partial_<(\supp\,v, \overline{U}) > 0\big\}
}\\ \label{eq:<Space}
&\,\,=\overline{\big\{v \in \mathcal{Z}_k^j \textup{ s.t. } \supp \,v\subset U\cup (\partial U\cap\partial\Omega)\big\}}
\end{align}
(where the closures are taken with respect to the $\cZ^j_k$ norm).
Observe that the convention that $\partial_<(A,B)=+\infty$ when $B=\Omega$ implies that $\mathcal{Z}_k^{j,<}(\Omega) = \mathcal{Z}_k^j$.

For any $U\subset \Omega$, let
\beq\label{eq:subset_norm}
\|u\|_{\cZ_{k}^j(U)}:=\inf\Big\{ \|v\|_{\mathcal{Z}_k^j}\,:\, v|_U=u|_U,\, v\in \mathcal{Z}_k^j\Big\};
\eeq
observe that this definition implies that $\|u\|_{\cZ_{k}^j(U)}\leq \|u\|_{\cZ_{k}^j(V)}$ for $U\subset V$. 

For an open set $U \subset \Omega$ and $j \geq 0$, observe that \eqref{eq:dual} implies that 
\begin{equation}
	\label{eq:negNorm}
 	\Zjdual{u}{j}{U} = \sup_{v \in \mathcal{Z}_k^{j,<}(U) \setminus \{0\}} \frac{ |u(v)|}{\norm{v}_{\mathcal{Z}_k^j}}\,.
\end{equation} 
Since $\mathcal{Z}_k^{j,<}(\Omega) = \mathcal{Z}_k^j$, $\Zjdual{\cdot}{j}{\Omega}=\norm{\cdot}_{(\mathcal{Z}^j_k)^*}$.

We define $\cW^{j,<}_k(U)$ and $\Wjdual{\cdot}{j}{U}$ analogously to \eqref{eq:<Space} and \eqref{eq:negNorm}.

\es{
\ble\label{lem:final}
For any open set $U \subset \Omega$ and any function $u \in \mathcal{Z}_k$,
\beq\label{eq:final1}
 	\Zjdual{u}{j}{U} \leq\Cemb \Wjdual{u}{j}{U}.
	\eeq
\ele

\begin{proof}
	Let $v \in \mathcal{Z}_k^j$ be non-zero, supported on $\overline{U}$, and such that $\partial_<(\textup{supp}\,v,U) > 0$. Then by \eqref{e:embedW} and the definition of $\mathcal{W}_k^{j,<}(U)$, $v \in \mathcal{W}_k^{j,<}(U)$ and furthermore
	\[\frac{\abs{\langle u,v\rangle_{\mathcal{H}}}}{\|v\|_{\mathcal{Z}_k^j}} \leq \Cemb \frac{\abs{\langle u,v\rangle_{\mathcal{H}}}}{\norm{u}_{\mathcal{W}^j_k}}\,.\]
	By definition of the dual $(\mathcal{W}^{j,<}_k(U))^*$ norm (defined analogously to \eqref{eq:negNorm}), 
	\[\frac{\abs{\langle u,v\rangle_{\mathcal{H}}}}{\|v\|_{\mathcal{Z}_k^j}} \leq \Cemb \Wjdual{u}{j}{U}\,,\]
	and the result follows since the set $\big\{v \in \mathcal{Z}_k^j \textup{ s.t. } \supp \,v\subset\overline{U},\,\partial_<(\supp\,v, \overline{U}) > 0\big\}$ is dense in $\mathcal{Z}_k^{j,<}(U)$ by its definition \eqref{eq:<Space}. 
\end{proof}
}

\subsection{Sesquilinear forms} We consider a family $(a_k)_{k \geq k_0}$ of {\em sesquilinear} forms $a_k:\mathcal{Z}_k\times \mathcal{Z}_k\to \mathbb{C}$ (i.e. $a_k(\lambda u,\mu v) = \lambda \overline{\mu} a_k(u,v)$), satisfying the following assumptions.

\begin{assumption}[Continuity and local coercivity]
	\label{ass:cont_coer}
	There exist positive constants $\Ccont$, $\ccoer$, and $\Ccoer$ such that
	\begin{equation}
		\label{eq:contAb}
		|a_k(u,v)|\leq \Ccont \N{u}_{\mathcal{Z}_k}\N{v}_{\mathcal{Z}_k} \quad\tfa u,v\in \mathcal{Z}_k \tand k\geq k_0
	\end{equation}
and if  $x_0\in \Omega$ and $r\leq \ccoer k^{-1}$ then
\beq\label{eq:coercivity1}
\Re \big\{a_k(v,v)\big\} \geq \Ccoer \N{v}^2_{\cZ_k}
\quad \tfa v\in \cZ^{<}_k(B(x_0,r)\cap\Omega) \tand k\geq k_0.
\eeq
\end{assumption}

\begin{assumption}[Elliptic regularity up to $\partial\Omega$ for the adjoint problem on $\cO(k^{-1})$ balls]\label{ass:er}
Given $x_0\in \Omega$, $r>0$, and $d>0$, let 
\beq\label{eq:Uballs}
U_0 := B(x_0,r)\cap \Omega\quad\tand\quad U_1 := B(x_0,r+d)\cap \Omega
\eeq
(so that $\partial_< (U_0,U_1)=d$).

Given $c>0$, and $\ell \in \mathbb{Z}^+$, there exists $\Cell>0$ such that
if $r+d \leq \ccoer k^{-1}$ and $r,d \geq c k^{-1}$,  then,
for all $u\in \mathcal{Z}_k^<({U}_1)$, 
	\begin{equation}
		\label{eq:ellipticRegB}
		\N{u}_{\mathcal{Z}_k^{j+2}(U_0)}\leq \Cell\Big(\N{u}_{\mathcal{H}}+\sup_{\substack{v\in \mathcal{Z}_k^<({U}_1) \euanspace \|v\|_{(\mathcal{W}_k^{j})^ *}=1}}| a_k(v,u)|\Big),\quad
		j=0,\dots,\ell.
	\end{equation}
\end{assumption}

The following assumption involves ``localizing" operators that commute with $a_k$ in a weak sense. 
In all the specific examples in \S\ref{sec:examples}, these operators are cut-off functions, but for transmission problems these functions must 
be defined piecewise and satisfy certain properties across the interface; see Lemma \ref{lem:ass_commut}.

\begin{assumption}[Compatible localisers]
	\label{ass:commut}
Given $x_0\in \Omega$, $r>0$, and $d>0$, let  $U_0$ and $U_1$ be as in \eqref{eq:Uballs}.
		There exist constants $\Cdagger > 0$ and $\Ccom > 0$ and a family $\{\psi(U_0,U_1)\}_{U_0 \subset U_1}$ of localisers indexed by $U_0 \subset U_1 \subset \Omega$, and hence indexed by $x_0, r, d$, 
		where each $\psi = \psi(U_0,U_1): \mathcal{H} \to \mathcal{H}$ is a self-adjoint operator with the following properties
	\begin{itemize}
		\item[(i)] 
With $U_0':= B(x_0, r+d/4) \cap \Omega$ and $U_1':= B(x_0,r+3d/4)\cap \Omega$, 
for all $u \in \mathcal{H}$, 
\[ u \equiv \psi u \textup{ on } U'_0\,, \quad \psi u \equiv 0 \textup{ on } (U'_1)^c\,.\]
		\item[(ii)] $\psi$ maps $\mathcal{Z}_k^j$ to itself continuously, with
		\begin{equation}
			\label{eq:improvedLeibniz}
			\norm{\psi  u}_{\mathcal{Z}_k^j} \leq C_\dagger\sum_{m=0}^{j} (kd)^{-(j-m)} \norm{u}_{\mathcal{Z}_k^m(U_1)} \quad \tfa j \in \{0,1,\ldots,\ell+2\},
\end{equation} 
\item[(iii)] for all $j = 1,\ldots,\ell+1$, and $u,v \in \mathcal{Z}_k$, 
	\end{itemize}
\begin{align}\nonumber
&\abs{a_k(\psi u,v) - a_k(u,\psi v)}\\ 
&\leq \frac{\Ccom}{kd} \left( \sum_{m= 0}^{j-1} (kd)^{-m} \right) \min\Big(\norm{u}_{\mathcal{Z}_k^j(U_1\setminus \overline{U_0})}\Wjdual{v}{(j-1)}{U_1\setminus \overline{U_0}},\Wjdual{u}{(j-1)}{U_1\setminus \overline{U_0}} \norm{v}_{\mathcal{Z}_k^j(U_1\setminus \overline{U_0})}\Big).
\label{eq:commute}
		\end{align}
		\end{assumption}

\begin{corollary}[Mapping properties of the adjoint solution operator on $\cO(k^{-1})$ balls]
	\label{cor:solution}
	
Let $U_0$ and $U_1$ be given by \eqref{eq:Uballs}. Suppose that $r+ d\leq \ccoer k^{-1}$. 
	Let $\mathcal{R}^*:(\cZ^{<}_k({U}_1))^* \to \mathcal{Z}_k^<({U}_1)$ be the operator defined by the variational problem
	\[a_k(v,\cR^* g) = \overline{\langle g,v\rangle} \quad \tfa v \in \mathcal{Z}_k^<({U}_1)\,.\]
	Then there exists $\Cres > 0$ such that
	\[\norm{\cR^* g}_{\mathcal{Z}_k^{j+2}(U_0)} \leq \Cres \norm{g}_{\cW^j_k} \quad \tfa g \in \mathcal{W}_k^{j,<}({U}_1)\,, \,\, j = 0\,,\ldots\,, \ell \,.\]
\end{corollary}
\begin{proof}
Since $r+d\leq \ccoer k^{-1}$, $\mathcal{R}^*$ is well-defined by Assumption \ref{ass:cont_coer} and the Lax--Milgram lemma, and 
	\beq\label{eq:hungry1}
	\norm{\cR^* g}_{\mathcal{Z}_k} \leq \frac{1}{\Ccoer} \norm{g}_{(\mathcal{Z}_k^<(U_1))^*}\,.
\eeq
	Let $g \in \mathcal{W}_k^{j,<}(U_1)$. By \eqref{eq:ellipticRegB}, the definition of $\mathcal{R}^*$, and the triangle inequality, 
	\begin{align*}
	\norm{\mathcal{R}^*g}_{\mathcal{Z}^{j+2}_k(U_0)} &\leq \Cell \Big(\norm{\mathcal{R}^*g}_{\mathcal{H}} + 
	\sup_{v\in \mathcal{Z}_k^<({U}_1),\,\norm{v}_{(\mathcal{W}^{j}_k)^*} = 1} 
	\abs{\langle g,v\rangle} \Big)\leq \Cell  \big(\norm{\mathcal{R}^*g}_{\mathcal{H}} + \norm{g}_{\mathcal{W}_k^j}\big),
	\end{align*}
and the result follows from \eqref{eq:hungry1}.
\end{proof}

\subsection{Triangulation and finite-dimensional subspaces}

Let $\mathcal{T}$ be a regular triangulation (in the sense of, e.g., \cite[Page 61]{Ci:91}) of $\Omega$. For each element $K \in \mathcal{T}$, let $h_K:= \textup{diam}(K)$. For simplicity, we assume that 
\beq\label{eq:Cppw}
h_K k\leq \Cppw
\eeq
for some $\Cppw>0$ (where ``ppw'' standard for ``points per wavelength"); 
as discussed after Theorem \ref{thm:intro1} for standard finite-element spaces with $p$ fixed, $h_K k$ must be chosen as a decreasing function of $k$ to maintain accuracy and thus this assumption is not restrictive.

\begin{assumption}[Broken norms]\label{ass:bn}
	For each $u \in \mathcal{Z}_k$  and each element $K$ of $\mathcal{T}$, the restriction of $u$ to $K$ belongs to $H^1_k(K)$, and 
	\begin{equation}
		\label{eq:normsOnElements}
		\norm{u}_{\mathcal{Z}_k(K)} = \norm{u}_{H^1_k(K)} \quad \tfa K \in \mathcal{T}\,.
	\end{equation}
	For each $K \in \mathcal{T}$, we assume that $C^\infty_0(K) \subset \mathcal{Z}_k^j$ and that
	\begin{equation}
		\label{eq:controlHONorms}
		\norm{u}_{\mathcal{Z}_k^j} = \norm{u}_{H^j_k(K)} \quad \tfa u \in C^\infty_0(K)
	\end{equation} 
(where the first norm is defined by \eqref{eq:subset_norm}). 
	Furthermore, there exists a constant $C_{\rm loc}$ such that 
	\begin{equation}
		\label{ass:loc}
		\abs{a_k(u,v)} \leq \Cloc \sum_{K \in \mathcal{T}} \norm{u}_{H^1_k(K)} \norm{v}_{H^1_k(K)} \quad \tfa u,v \in \mathcal{Z}_k\,.
	\end{equation}
\end{assumption}

We fix a finite-dimensional space $V_h \subset \mathcal{Z}_k$ consisting of functions whose restrictions to each $K \in \mathcal{T}$ is in $C^\infty(\overline{K})$. For any open subset $U \subset \Omega$, let 
\[V_h^<(U) :=\mathcal{Z}_k^{<}(U) \cap V_h \,.\]
We introduce the following standard assumptions on $V_h$:
\begin{assumption}[Approximation property]
	\label{ass:ap}
	There exist constants $\Ckappa > 0$, $\Cp\in \mathbb{Z}^+$ and $\Capprox > 0$ such that the following holds.
	 For each $j \in \{1,\ldots,p+1\}$, given $u \in \mathcal{Z}_k^j$, there exists $u_h \in V_h$ such that
	\begin{equation}
		\label{eq:ap}
		\sum_{K \in \mathcal{T}} (h_Kk)^{2(m - j)}\norm{u - u_h}^2_{H^m_k(K)} \leq \Capprox\norm{u}_{\mathcal{Z}_k^j}^2 \quad\tfa u_h \in V_h\,, \,\, 0 \leq m \leq j \leq \Cp+1.
	\end{equation}
	Furthermore, if $U_0 \subset U_1$ are such that 
	$$\partial_<(U_0,U_1) > \Ckappa \max_{K \cap U_1 \neq \emptyset} h_K$$ and $\textup{supp}\,u \subset U_0$, then $u_h$ can be chosen in $V_h^<(U_1)$. 
\end{assumption}

\begin{assumption}[Super-approximation property]
	\label{ass:sa}
	For each $\Cdagger>0$, there exists constant $\Csuperk >0$ such that, with $\Cp$ and $\Ckappa$ as in Assumption \ref{ass:ap}, the following property holds for sets $U_0 \subset U_1 \subset \Omega$ 
given by \eqref{eq:Uballs} and
	satisfying
	\begin{equation}
		\label{eq:dU0U1}
d = \partial_<(U_0,U_1) > 4\Ckappa \max_{K \cap U_1 \neq \emptyset} h_K
	\end{equation}
	Let $\chi = \psi(U_0,U_1)$ be the localiser associated to $U_0$ and $U_1$ (given by Assumption \ref{ass:commut}).
 	Then for each $u_h \in V_h$, there exists $v_h \in V_h^<(U_1)$ such that, for all $K\in \mathcal{T}$, 
	\begin{equation}
		\label{eq:saH1k}
		\norm{\chi^2 u_h - v_h}_{H^1_k(K)} \leq \Csuperk \frac{h_K}{d}\left[\left(1 + \frac{1}{kd}\right) \norm{u_h}_{L^2(K)} + \norm{\chi u_h}_{H^1_k(K)}\right].
	\end{equation}
\end{assumption}

The constant $\Ckappa$ in Assumptions \ref{ass:ap} and \ref{ass:sa} is related to the ``stencil" of the chosen finite element, i.e., how large the support the finite element basis functions is.
For Lagrange finite-elements $\kappa=1$; see \S\ref{sec:verifySubspace} below.

For any open subset $U \subset \R^d$ and any $s\in \Rea$, we recall that $H^s(U)$ is defined as the set of restrictions to $U$ of elements of $H^s(\R^d)$, with a Hilbert structured inherited via 
\beq\label{eq:negativeNorm}
\norm{v}_{H^s_k(U)} := \inf_{V\in H^s_k(\Rea^d)\,:\, V|_{U} = v} \norm{V}_{H^s_k(\R^d)}\,.
\eeq
When $U$ is a Lipschitz domain,
\beq\label{eq:negativeNormEquiv}
\norm{u}_{H^{-s}_k(U)} \sim \sup_{\substack{\norm{v}_{H^s_k(\R^d)} = 1,\\ \supp v\subset U}} |(u,v)_{L^2}|,
\eeq
where $\sim$ denotes norm equivalence; i.e., $H^{-s}_k(U)$ is dual to $\widetilde{H}^s_k(U)$ defined as the closure of $C^\infty_{\rm comp}(U)$ in $H^s(\Rea^d)$; see 
 \cite[Page 77 and Theorem 3.30(i), Page 92]{Mc:00}.

\begin{assumption}[Inverse inequalities]
	\label{ass:ii}
Given $p\in \mathbb{Z}^+$ as in Assumption \ref{ass:ap}, 
  there exists 
$\Cinvk > 0$ such that, for all $K \in \mathcal{T}$ and $u_h \in V_h$,
	\begin{equation}
		\label{eq:iiH1k1}
		\norm{u_h}_{H^1_k(K)} \leq \frac{\Cinvk}{h_Kk} \norm{u_h}_{L^2(K)}
	\end{equation}
 and, for $0\leq s \leq p$,
	\begin{equation}
		\label{eq:iiH1k2}
\norm{u_h}_{L^2(K)} \leq  \frac{\Cinvk}{(h_Kk)^{s}} \norm{u_h}_{H_k^{-s}(K)}.
	\end{equation}
\end{assumption}

\section{Statement of the main results}\label{sec:statement}

\begin{theorem}[General version of Theorem \ref{thm:intro1}]\label{thm:WDGS1}
	Given positive constants $\Ccont$, $\Ccoer$, $\ccoer$, $\Ckappa$, $\Cp$, $\Cinvk$, $\Cpw$, $\Csuperk$, $\Ccom> 0$, and some $C_0 > 0$, there exists a constant $C_\star > 0$ such that the following holds. Let $(a_k)_{k \geq k_0}$ and $V_h \subset \mathcal{Z}_k$ satisfy Assumptions \ref{ass:cont_coer}, \ref{ass:commut}, \ref{ass:ap}, \ref{ass:sa}, \ref{ass:ii}, and \ref{ass:loc}, with the constants above. Let $\Omega_0 \subset \Omega_1 \subset \Omega$ be arbitrary subsets such that 
	\beq
		\label{eq:WDGS1:assd}
d:= \partial_<(\Omega_0,\Omega_1)\geq \frac{C_0}{ k} \quad \tand \quad 
\max_{K \cap \Omega_1 \neq \emptyset} h_K \leq \frac{C_1}{k}
\eeq
where
\beqs
C_1:=
\frac{1}{4\max \big\{1,8\Ckappa\big\}}
\min\bigg\{ \frac{C_0}{2}, \frac{\es{4}\ccoer}{3}\bigg\}. 
\eeqs
If $k \geq k_0$ and $u \in \mathcal{Z}_k$ and $u_h \in V_h$ are such that 
	\beq\label{eq:GOG1}
	a_k(u-u_h,v_h) = 0 \quad \tfa v_h \in V_h^<(\Omega_1),
	\eeq
	then
\begin{align}
	&\norm{u - u_h}_{\mathcal{Z}_k(\Omega_0)} \leq \Cstar \Big(\min_{w_h \in V_h}\norm{u - w_h}_{\mathcal{Z}_k(\Omega_1)} 
	+ \norm{u - u_h}_{\mathcal{H}(\Omega_1)}\Big).\label{eq:WDGS1}
	\end{align}
\end{theorem}
If the triangulation $\mathcal{T}$ is furthermore locally quasi-uniform and Assumption \ref{ass:er} (local elliptic regularity) holds, 
then the result can be improved by weakening the $\mathcal{H}$ norm of the error on the right-hand side.

\begin{theorem}[General version of Theorem \ref{thm:intro2}]
	\label{thm:WDGS2}
		Given positive constants $\Ccont$, $\Ccoer$, $\ccoer$, $\Cell$, $\Ckappa$, $\Cp$, $\Cinvk$, $\Cpw$, $\Csuperk$, $\Ccom, \es{\Capprox}> 0$, and some $C_0 > 0$, there exists a constant $C_\star > 0$ such that the following holds. Let $(a_k)_{k \geq k_0}$ and $V_h \subset \mathcal{Z}_k$ satisfy Assumptions \ref{ass:cont_coer}, \ref{ass:er}, \ref{ass:commut}, \ref{ass:ap}, \ref{ass:sa}, \ref{ass:ii}, and \ref{ass:loc}, with the constants above. Let $\Omega_0 \subset \Omega_1 \subset \Omega$ be arbitrary subsets such that 
\beqs
d:= \partial_<(\Omega_0,\Omega_1)\geq \frac{C_0}{ k} \quad \tand \quad
\max_{K \cap \Omega_1 \neq \emptyset} h_K \leq \frac{C_1}{k}
\eeqs
where
\beqs
C_1:=
\frac{1}{48(\ell+2)\max \big\{1,8\Ckappa\big\}}
\min\bigg\{ \frac{C_0}{2}, 2\ccoer\bigg\}. 
\eeqs
Assume further that $\mathcal{T}$ is {\em quasi-uniform on scale $k^{-1}$}, in the sense that, for every ball $B$ of radius at most $3\ccoer/(2k)$,
the bound \eqref{eq:AssQuThm} holds.
If $k \geq k_0$ and $u \in \mathcal{Z}_k$ and $u_h \in V_h$ are such that \eqref{eq:GOG1} holds, then
	\begin{equation}
\norm{u - u_h}_{\mathcal{Z}_k(\Omega_0)} 
\leq \Cstar 
		\bigg(\min_{w_h \in V_h} 
		\norm{u - w_h}_{\mathcal{Z}_k(\Omega_1)} 
		+\Wjdual{u-u_h}{\newell+1}{\Omega_1}\bigg)\label{eq:WDGS3}
	\end{equation}
where
	$\newell:= \min\{\ell,p-1\}$.
\end{theorem}

\bre[Galerkin orthogonality]
\label{rem:localGOG}
Given $G \in (\mathcal{Z}_k)^*$, 
if $u\in \mathcal{Z}_k$ satisfies 
\[a_k(u,v) = G(v) \quad \tfa v \in \mathcal{Z}_k\]
and if $u_h$ is a solution to the Galerkin equations
\beq\label{eq:PML_Galerkin}
\text{ find } u_h \in V_h \,\tst\, a(u_h,v_h) = G(v_h) \quad\tfa v_h \in V_h,
\eeq
then 
\beqs
a(u-u_h,v_h) = 0 \quad\tfa v_h \in V_h,
\eeqs
and thus
$u - u_h$ satisfies \eqref{eq:GOG1} since $V_h^<(U_1) \subset V_h$. 
\ere

\bre[The dependence of the constant $C^*$ on the subspace $V_h$]\label{rem:subspace}
We have stated Theorem \ref{thm:WDGS1} for a fixed finite-dimensional subspace $V_h \subset \mathcal{Z}_k$. The key point is that the constant $C^*$ depends quantitatively on $V_h$ only through the constants in the assumptions of Section \ref{sec:assumptions}. In general, one is interested in applying Theorem \ref{thm:WDGS1} to a \emph{family} of subspaces indexed by some parameter (e.g., the mesh parameter $h$ of a sequence of quasi-uniform triangulations $(\mathcal{T}_h)_{h > 0}$ of $\Omega$). The idea is that the bound \eqref{eq:WDGS1} will hold with a common constant $C^*$ for all subspaces of the sequence, provided that the assumptions of Section \ref{sec:assumptions} hold uniformly for all of those subspaces. 
We show in \S\ref{sec:verifySubspace} that these assumptions do \es{indeed} hold uniformly for standard choices of discretisation.
\ere

In the rest of the paper we use the letter $C$ in estimates of the form $a \leq Cb$ to represent a generic positive constant, whose numerical value is allowed to change from one place to another, but which can be expressed as a function of the constants in the assumptions of Section \ref{sec:assumptions}.

\section{Caccioppoli estimates}\label{sec:Caccioppoli}
  
The central idea in the proofs of Theorems \ref{thm:WDGS1}-\ref{thm:WDGS2} is to apply 
a discrete version of the classical Caccioppoli inequality to a solution of the Helmholtz equation at the discrete level.

In this section we always assume (without stating explicitly) that $(a_k)_{k \geq k_0}$ and $V_h$ satisfy the assumptions of Section \ref{sec:assumptions}. 

\begin{lemma}[Caccioppoli estimate in the $L^2$ norm on $\cO(k^{-1})$ balls]
	\label{lem:Caccioppoli}
	There exists a constant $\Cca > 0$ (whose value depends only on the constants appearing in the assumptions of Section \ref{sec:assumptions} apart from $\Cell$) such that, given $r,d>0$, if $U_0$ and $U_1$ are as in \eqref{eq:Uballs} (so that $d= \partial_<(U_0,U_1)$), 
	\beq\label{eq:dist4h}
d \geq 
\max\left\{1, 8\kappa\right\}  \max_{K \cap U_1 \neq \emptyset }h_K,
\eeq
and
\beq\label{eq:ccoer}
r+ \frac{d}{2} \leq \frac{\ccoer}{k},
\eeq
	then the following holds. If  $z_h \in V_h$ satisfies 
	\begin{equation}
		\label{eq:orthoCaccio}
	a_k(z_h,v_h) = 0 \quad \textup{for all } v_h \in V_h^<(U_1)
	\end{equation}
and $k\geq k_0$, then 
	\begin{equation}
		\label{eq:Caccioppoli}
		\norm{z_h}_{\mathcal{Z}_k(U_0)} \leq \frac{\Cca}{kd}\norm{z_h}_{\mathcal{H}(U_1)}.
	\end{equation}
\end{lemma}

Note that the combination of \eqref{eq:dist4h} and \eqref{eq:ccoer} imply that $\max_{K\in U_1}h_K$ is bounded by a constant multiple of $k^{-1}$.

\begin{proof}[Proof of Lemma \ref{lem:Caccioppoli}]
Let 
\beqs
U_{1/2} := B(\es{x_0}, r+d/2)\cap \Omega
\eeqs
so that 
\beq\label{eq:dist_proof1}
\partial_<(U_0,U_{1/2})  = d/2, \quad \partial_<(U_{1/2},U_{1}) = d/2.
\eeq
Let $\chi = \psi(U_0,U_{1/2})$ be the localiser defined in Assumption \ref{ass:commut}. If we can show that 
\beq\label{eq:sufficientC}
\norm{\chi z_h}^2_{\mathcal{Z}_k} \leq \frac{C}{(kd)^2} \norm{z_h}^2_{\mathcal{H}(U_1)},
\eeq
then the result follows since, by \eqref{eq:subset_norm}, 
\beqs
\norm{z_h}^2_{\mathcal{Z}_k(U_0)} = \norm{\chi z_h}^2_{\mathcal{Z}_k(U_0)} \leq \norm{\chi z_h}_{\mathcal{Z}_k}^2.
\eeqs

The inequality \eqref{eq:ccoer} and the assumption \eqref{eq:coercivity1} imply that $a_k$ is coercive on $\cZ_k^<(U_{1/2})$. 
Since $\chi z_h \in \cZ_k^< (U_{1/2})$, 
\beqs
\norm{\chi z_h}_{\mathcal{Z}_k}^2 \leq  (\Ccoer)^{-1}\big| {a}_k(\chi z_h,\chi z_h)\big|.
\eeqs
Then, by \eqref{eq:commute} with $j=1$,
\beq
\norm{\chi z_h}_{\mathcal{Z}_k}^2
\leq (\Ccoer)^{-1}\big| a_k (z_h, \chi^2 z_h)\big| + \Ccom (kd)^{-1} \N{z_h}_{\cH(U_{1/2})}\N{\chi z_h}_{\cZ_k},
\label{ineq0}
\eeq
so that, using
 	the inequality 
	\beq\label{eq:peterpaul1}
	2ab \leq \eps a^2+ b^2/\eps\quad\tfa a,b,\eps>0,
	\eeq
we have, 	for all $\varepsilon_1 > 0$,
	\begin{equation}
		\label{eq:ineq1}
	\norm{\chi z_h}^2_{\mathcal{Z}_k} \leq C\left(\abs{a_k(z_h,\chi^2 z_h)} +  \varepsilon_1^{-1}(kd)^{-2}\norm{z_h}^2_{\mathcal{H}(U_{1/2})}\right)+ {\varepsilon_1} \norm{\chi z_h}_{\mathcal{Z}_k}^2.
	\end{equation}

Let $w_h \in V_h^<(U_\es{1/2})$ be the finite-element super-approximation of $\chi^2 z_h$ provided by Assumption \ref{ass:sa} applied to the pair of sets $\es{U_0, U_{1/2}}$; note that the condition \eqref{eq:dU0U1} needed to apply Assumption \ref{ass:sa} becomes $d> 8 \kappa \max_{K\cap U_1\neq\emptyset} h_K$ by \eqref{eq:dist_proof1}, which holds by \eqref{eq:dist4h}.
				
By the Galerkin orthogonality \eqref{eq:orthoCaccio}, the  property \eqref{ass:loc}, and the fact that both $\chi^2 z_h$ and $w_h$ are supported on $U_{1/2}$,
\beq\label{eq:ineq3}
\abs{a_k(z_h,\chi^2 z_h )} =\abs{a_k(z_h,\chi^2 z_h - w_h)} \leq \Cloc\sum_{K \cap U_{1/2} \neq \emptyset} \norm{z_h}_{H^1_k(K)} \norm{\chi^2 z_h- w_h}_{H^1_k(K)}.
\eeq
Now, by \eqref{eq:saH1k}, the fact that $kd\leq 2\ccoer$ by \eqref{eq:ccoer}, \eqref{eq:iiH1k1}, and \eqref{eq:peterpaul1}, 
	\[\begin{split}
\norm{z_h}_{H^1_k(K)} \norm{\chi^2 z_h- w_h}_{H^1_k(K)} &\leq \Csuperk \norm{z_h}_{H^1_k(K)}\frac{h_K}{d}\left[\frac{\es{(1+ 2 \ccoer)}}{kd} \norm{z_h}_{L^2(K)} + \norm{\chi z_h}_{H^1_k(K)}\right] \\
		&\hspace{-2cm}\leq \frac{\Csuperk\Cinvk }{kd}\left[\frac{\es{(1+ 2 \ccoer)}}{kd}\norm{z_h}^2_{L^2(K)} + \norm{\chi z_h}_{H^1_k(K)}\norm{z_h}_{L^2(K)}\right]\\
		&\hspace{-2cm} \leq  \frac{C}{(kd)^2}(1 + \varepsilon_2^{-1})\norm{z_h}^2_{L^2(K)}+ \varepsilon_2\norm{\chi z_h}^2_{H^1_k(K)}
	\end{split}\]
for all $\varepsilon_2>0$. 
Combining this last inequality  with \eqref{eq:ineq3} and \eqref{eq:ineq1} and then using \eqref{eq:normsOnElements}, 
we obtain
	\[\begin{split}
		\norm{\chi z_h}^2_{\mathcal{Z}_k} \leq \,& \frac{C}{(kd)^2}(1 + \varepsilon_1^{-1} +  \varepsilon_2^{-1})\norm{z_h}^2_{\mathcal{H}(U_1)} + (\varepsilon_1 + \varepsilon_2)\norm{\chi z_h}^2_{\mathcal{Z}_k}.
	\end{split}\]
	Choosing $\eps_1=\eps_2 =1/4$, the last term on the right-hand side can be absorbed in the left-hand side, leading to \eqref{eq:sufficientC}, and hence the result \eqref{eq:Caccioppoli} follows.
\end{proof}

\

\es{To prove the Caccioppoli estimate with a negative norm on the right-hand side (Lemma \ref{lem:CaccioppoliNegative2}) we need the following lemma.}

\begin{lemma}
	\label{lem:abstractNegativeNorms}
	If $\Omega_0 \subset \Omega_1 \subset \Omega$ are arbitrary sets such that 
	\[\bigcup{\{K \in \mathcal{T} \textup{ s.t. } K \cap \Omega_0 \neq \emptyset\}}  \subset \Omega_1\,,\]
	then 
		\begin{equation}\label{eq:abstractNegativeNorms}
		\sum_{K \cap U_0 \neq \emptyset} \norm{u}^2_{H^{-j}_k(K)} \leq  \Zjdual{u}{j}{\Omega_1}^2 .
	\end{equation}
\end{lemma}
\begin{proof}
	The proof is very similar to that of \cite[Lemma 1.1]{ScWa:77}. Let $\varphi_K \in C^\infty_{\rm comp}(K)$ be such that $\norm{\varphi_K}_{H^j_k(K)} = 1$, let $\theta_K := \norm{u}_{{H}^{-j}_k(K)}$, and let
	\[\varphi := \sum_{K \cap \Omega_0 \neq \emptyset} \theta_K \varphi_K\,.\]
By linearity, $\varphi \in \mathcal{Z}_k^j$, $\supp\,\varphi \subset \Omega_1$, so that $\varphi \in \cZ^{j,<}_k(\Omega_1)$. 
By \eqref{eq:ZNormDisjoint} and \eqref{eq:controlHONorms},
	\[
	\norm{\varphi}_{\mathcal{Z}_k^j}^2 \leq\sum_{K \cap U_0 \neq \emptyset} \theta_K^2\,.
	\] 
	Hence,  
\[	\Zjdual{u}{j}{\Omega_1}^2 
	 \geq \frac{\abs{(u,\varphi)_{\mathcal{H}}}^2}{\norm{\varphi}^2_{\mathcal{Z}_k^j}} \geq \frac{\big|\sum_{K \cap \Omega_0 \neq \emptyset} \theta_K (u,\varphi_K)_{\mathcal{H}}\big|^2}{\sum_{K \cap \Omega_0 \neq \emptyset} \theta_K^2}
	 \es{
	 =\frac{\big|\sum_{K \cap \Omega_0 \neq \emptyset} \norm{u}_{{H}^{-j}_k(K)} (u,\varphi_K)_{\mathcal{H}}\big|^2}{\sum_{K \cap \Omega_0 \neq \emptyset} \norm{u}_{{H}^{-j}_k(K)}^2}
	 }
	 .\]
	Taking the supremum over $\varphi_K$ in the right hand side, we obtain the result. 
	\end{proof}

\begin{lemma}
	Let $B_i,i=1,\ldots,N$ be open sets and let $\es{\Cover}$ be as in \eqref{eq:defC}. Then
	\beq\label{eq:last_day2addition}
	\sum_{i = 1}^N \Wjdual{u}{j}{B_i}^2 \leq \es{\Cover} \Wjdual{u}{j}{\cup_{i=1}^N B_i}^2
	\eeq 
\end{lemma}
\begin{proof}
	The proof is very similar to that of Lemma \ref{lem:abstractNegativeNorms}, 
	with the following modifications. The function $\varphi_i$ is now an arbitrary element of $\cW_k^{j,<}(B_i)$ with unit $\cW_k^j$ norm, and $\theta_i := \norm{u}_{\cW_k^j}$. We now let 
	\beqs
	\varphi:=\sum_{i} \theta_i \varphi_i,
\quad\text{ so that } \quad 
	\norm{\varphi}^2_{\cW_k^j} \leq \es{\Cover} \sum_{i=1}^N \theta_i^2
	\eeqs
	by Assumption \ref{ass:balls}, and the rest of the proof is unchanged.
\end{proof}
\begin{remark}\label{rem:Friday}
	It is clear from the proof of Lemma \ref{lem:abstractNegativeNorms} that the bound also holds when the $\Zjdual{\cdot}{j}{U}$ norm is defined with a supremum ranging over the smaller subset of functions supported in $U$, instead of all functions in $\mathcal{Z}_k^{j,<}(U)$.
As a consequence, Theorem \ref{thm:WDGS2} also holds with this changed definition of $\Zjdual{\cdot}{j}{U}$.
\end{remark}

\begin{lemma}[Caccioppoli estimate in negative norms]
	\label{lem:CaccioppoliNegative2}
	Given $C_{\rm qu} ,c> 0$, there exists a constant $C'_{\rm ca} > 0$ (whose value depends on the constants appearing in the assumptions of Section \ref{sec:assumptions}) such that the following holds.
	Let $V_h$ satisfy the assumptions of Section \ref{sec:assumptions}, and let the sets $U_0 \subset U_1 \subset \Omega$ 
	be as in \eqref{eq:Uballs} (so that $d=\partial_<(U_0,U_1)$) with 
	\beq\label{eq:dist5h}
d\geq 12(\ell+2)
\max\left\{ 1, 8\kappa\right\}\max_{K\cap U_1\neq\emptyset}h_K  ,\qquad r,d\geq ck^{-1}
\quad\tand\quad r+ d \leq \frac{\ccoer}{k},
\eeq
and
	\begin{equation}
		\frac{\max_{K \cap U_1 \neq \emptyset} h_K}{\min_{K \cap U_1 \neq \emptyset} h_K} \leq \Cqu\,.
		\label{ass:quLem}
	\end{equation}
If $z_h\in V_h$ satisfies~\eqref{eq:orthoCaccio}, then
	\[
	\norm{z_h}_{\mathcal{Z}_k(U_0)} \leq C_{\rm ca}'\left(
	 \frac{1}{kd}\right)^{\alpha_*} \Wjdual{z_h}{\newell+1}{U_1},
	\]
	where $\newell:= \min\{\ell,p-1\}$ and $\alpha_* =(\newell+2)(\newell+3)/2$.
\end{lemma}
\begin{proof}
Let $\widetilde{U}_0:= B(x_0, r+d/2)\cap \Omega$. Let $\widetilde{d}:= \partial_< (\widetilde{U}_0,U_1)$ and note that $ \widetilde{d}=d/2$.
Later in the proof we apply Corollary \ref{cor:solution} with $U_0 \to \widetilde{U}_0$, $U_1\to U_1$, $d\to \widetilde{d}= d/2$ and $r\to r+ d/2$. Note that $r+d \to r+d$, so that the condition $r+d\leq \ccoer/k$ remains the same.

	By Lemma \ref{lem:Caccioppoli}, it suffices to show that for $j = 0,\ldots,s:=\min\{\ell,p-1\}$, 
	\begin{equation}
		\label{sufficient3}
		\Wjdual{z_h}{j}{U_j'}
		\leq C \left(
		 \frac{1}{k\widetilde{d}}\right)^{j+2}
\Wjdual{z_h}{j+1}{U_{j+1}'}
			\end{equation}
	where the sets
\beqs
U_j' := B\bigg(x_0, r+ \frac{j+1}{\ell+2}\widetilde{d}\bigg)\cap \Omega
\eeqs
so that 	
\[
U_0 \subset U'_0 \subset U'_1 \subset \ldots \subset U'_{\ell+1} = \widetilde{U}_0.
\]
\es{In proving \eqref{sufficient3}, we use five nested sets between $U'_j$ and $U'_{j+1}$; we therefore} 
let 
\beqs
U'_{j + \nu/6} := B \bigg(x_0, r+ \frac{j+1+\nu/6}{\ell+2}
\widetilde{d}\bigg)\cap \Omega, 	\quad\tfor\, \nu = 0,\ldots,5.
\eeqs	 
\es{Since $\partial_<(U'_{j + \nu/6},U'_{j+(\nu+1)/6})=(6 (\ell +2))^{-1}\widetilde{d}=(6 (\ell +2))^{-1}d/2$,}
the first condition in \eqref{eq:dist5h} implies that 
	\beq\label{eq:pierre1}
	\partial_<(U'_{j + \nu/6},U'_{j+(\nu+1)/6}) > \max\big\{8\kappa,1\big\} \max_{K\cap U_1 \neq \emptyset} h_K
	\quad\tfor\, \nu = 0,\ldots,5.
	 \eeq
	 	  We introduce the localiser $\chi = \psi(U'_{j+1/6},U'_{j+2/6})$. Let $v \in \mathcal{W}_k^{j,<}(U'_j)$ and note that $\chi v= v$. Let $\mathcal{R}^*$ be the solution operator on $U_1$ defined in Corollary \ref{cor:solution}. Then
	\[\abs{(z_h,v)_{\mathcal{H}}} = \abs{(\chi z_h,v)_{\mathcal{H}}} = \abs{a_k(\chi z_h, \mathcal{R}^*v)}\,.\]
	By the orthogonality \eqref{eq:orthoCaccio}, for all $w_h \in V_h$,
	\begin{equation}
		\label{eq:proof_commutators}
	a_k(\chi z_h,\mathcal{R}^*v) = a_k(z_h,\chi \mathcal{R}^*v - w_h) +
	\Big(a_k(\chi z_h,\mathcal{R}^*v) - a_k(z_h,\chi \mathcal{R}^*v)\Big).
	\end{equation}
By Assumption \ref{ass:ap}, \eqref{eq:pierre1}, and the fact that $\supp \chi \subset U'_{j+2/6}$, we can choose $w_h \in V_h^<(U'_{j+{3/6}})$ as an approximation of $\chi \mathcal{R}^* v$ satisfying \eqref{eq:ap}. Using (in this order) the locality of $a_k$ (\eqref{ass:loc} in Assumption \ref{ass:bn}), the Cauchy--Schwarz inequality,
the approximation property \eqref{eq:ap} (noting that, by the definition of $s$, $s+2\leq p+1$), \eqref{eq:improvedLeibniz} (with $j$ replaced by $j+2$),
and the elliptic regularity for $\mathcal{R}^*$ (Corollary \ref{cor:solution}),
we find
	\begin{align}
		\nonumber
		&\abs{a_k(z_h,\chi \mathcal{R}^*v - w_h)} \leq C\sum_{K \in \mathcal{T}} \norm{z_h}_{H^1_k(K)} \norm{\chi\mathcal{R}^* v - w_h}_{H^1_k(K)} \\
		\nonumber
		&\leq C \bigg(\sum_{K \cap U'_{j+3/6} \neq \emptyset} (h_Kk)^{2(j+1)} \norm{z_h}^2_{H^1_k(K)} \bigg)^{1/2}\bigg(\sum_{K \cap U'_{j+3/6} \neq \emptyset}  (h_Kk)^{-2(j+1)}\norm{\chi \mathcal{R}^* v - w_h}_{H^1_k(K)}^2\bigg)^{1/2}  \\ \nonumber
		&\leq  C(hk)^{j+1}\bigg(\sum_{K \cap U'_{j+
		3/6} \neq \emptyset} \norm{z_h}^2_{H^1_k(K)}\bigg)^{1/2} \norm{\chi \mathcal{R}^* v}_{\mathcal{Z}_k^{j+2}}\\
		& \leq C \left(
		\frac{1}{k\widetilde{d}}\right)^{j+2}(hk)^{j+1} \norm{z_h}_{\mathcal{Z}_k(U'_{j+4/6})} \norm{v}_{\mathcal{W}_k^{j}}\, \label{eq:temp}.
	\end{align} 
where $h := \max_{K \cap U_{j+3/6}' \neq \emptyset} h_K$ and we have used \eqref{eq:pierre1} in the last step.

We then use Lemma \ref{lem:Caccioppoli} to bound $\norm{z_h}_{\mathcal{Z}_k(U'_{j+4/6})}$ by $\norm{z_h}_{\mathcal{H}(U'_{j+5/6})}$. 
The condition \eqref{eq:dist4h} now becomes 
\beqs
\widetilde{d} \geq 6(\ell+2)\max\{1,8\kappa\}\max_{K\cap U_{5/6}'\neq \emptyset} h_K,
\eeqs
which is satisfied by the first condition in \eqref{eq:dist5h}. By the second condition in \eqref{eq:dist5h}, the condition \eqref{eq:ccoer}
 is satisfied ; i.e., the assumptions of Lemma \ref{lem:Caccioppoli} are satisfied.

We next use the 
\es{local quasi-uniformity assumption \eqref{ass:quLem}, the} inverse estimate \eqref{eq:iiH1k2} (noting that $s+1\leq p$) and  \eqref{eq:abstractNegativeNorms} to obtain 
	\begin{align}\nonumber
		(hk)^{2(j+1)} \norm{z_h}_{\mathcal{H}(U'_{j+5/6})}^2 \leq C\sum_{K \cap U'_{j+5/6} \neq \emptyset} (h_Kk)^{2(j+1)}\norm{z_h}_{L^2(K)}^2
		&\leq C\sum_{K \cap U'_{j+5/6} \neq \emptyset} \norm{z_h}^2_{H^{-(j+1)}_k(K)} \\
		& \leq C 
		\Zjdual{z_h}{j+1}{U'_{j+1}}^2,
		\label{eq:temp2}
	\end{align}
	where,
	\es{in using \eqref{eq:abstractNegativeNorms}, we have used that $\partial_<(U'_{j + 5/6},U'_{j+1}) > \max_{K\cap U_1 \neq \emptyset} h_K$ by \eqref{eq:pierre1}.} 
	Combining \eqref{eq:temp} and \eqref{eq:temp2}, we obtain the following bound on the first term on the right-hand side of \eqref{eq:proof_commutators}:
	\begin{align}\nonumber
	\abs{a_k(z_h,\chi \mathcal{R}^* v - w_h)} 
&\leq C\left(
	 \frac{1}{k\widetilde{d}}\right)^{j+2} \Zjdual{z_h}{j+1}{U'_{j+1}} \norm{v}_{\mathcal{W}_k^j},\\
&
\es{\leq C \left(
	 \frac{1}{k\widetilde{d}}\right)^{j+2} \Wjdual{z_h}{j+1}{U'_{j+1}} \norm{v}_{\mathcal{W}_k^j},}
	\label{eq:Friday_combine1}
	\end{align}
	\es{where we have used \eqref{eq:final1} in the last step.}
	
To bound the second term on the right-hand side of in~\eqref{eq:proof_commutators}, we use 
\eqref{eq:commute} (with $j$ replaced by $j+2$), 
	 and the mapping properties of $\mathcal{R}^*$ from Corollary \ref{cor:solution} to find that 
	\begin{align}\nonumber
	&\Big|a_k(\chi z_h,\mathcal{R}^*v) - a_k(z_h,\chi \mathcal{R}^*v)\Big|\\ \nonumber
	&\qquad\leq \frac{\Ccom}{k\widetilde{d}}\left(\sum_{r = 0}^{j+1} (k\widetilde{d})^{-r}\right) 
	\Wjdual{z_h}{j+1}{U'_{j+2/6}\setminus \overline{U_{j+1/6}'}} \norm{\mathcal{R}^* v}_{\mathcal{Z}_k^{j+2}(U'_{j+2/6}\setminus \overline{U_{j+1/6}'})}\\ 
		&\qquad\leq \frac{C}{k\widetilde{d}} \bigg(
		 \frac{1}{k\widetilde{d}}\bigg)^{j+1} \Wjdual{z_h}{j+1}{U'_{j+1}} \norm{v}_{\mathcal{W}_k^{j}}.\label{eq:Friday_combine2}
	\end{align}
Combining \eqref{eq:proof_commutators}, \eqref{eq:Friday_combine1}, and \eqref{eq:Friday_combine2}, 
we obtain that 
\[
\abs{(z_h,v)_{\mathcal{H}}} \leq C \left(
 \frac{1}{k\widetilde{d}}\right)^{j+2} \Wjdual{z_h}{j+1}{U_{j+1}'}\norm{v}_{\mathcal{W}_k^j}
\]
	for all $v \in \mathcal{W}_k^{j,<}(U'_{j})$, which implies the result \eqref{sufficient3}.
\end{proof}

\section{Proofs of Theorems \ref{thm:WDGS1} and \ref{thm:WDGS2}}\label{sec:proofs}

\ble[Analogue of Theorem \ref{thm:WDGS1} for small balls close together]
\label{lem:WDGS}
	Given positive constants $\Ccont$, $\Ccoer$, $\ccoer$, $\Ckappa$, $\Cp$, $\Cinvk$, $\Cpw$, $\Csuperk$, $\Ccom> 0$, there exists a constant $C_\star > 0$ such that the following holds. Let $(a_k)_{k \geq k_0}$ and $V_h \subset \mathcal{Z}_k$ satisfy Assumptions \ref{ass:cont_coer}, \ref{ass:commut}, \ref{ass:ap}, \ref{ass:sa}, \ref{ass:ii}, and \ref{ass:loc}, with the constants above. 
For $x_0\in \Omega$, let 
	\beq\label{eq:choice2}
\Omega_0= B(x_0,d/4)\cap \Omega \quad\tand\quad\Omega_1=B(x_0,3d/4)\cap \Omega
\eeq
with 	
	\begin{equation}\label{eq:lemWDGS0}
d\leq \frac{\es{4}}{3}\frac{\ccoer}{k}
		\quad \tand \quad 
		d\geq 4\max\left\{1,8\kappa\right\} \max_{K \cap \Omega_1 \neq \emptyset} h_K.
	\end{equation}
If $k \geq k_0$, $u \in \mathcal{Z}_k$, and $u_h \in V_h$ are such that 
	\beq\label{eq:GOG1_small}
	a_k(u-u_h,v_h) = 0 \quad \tfa v_h \in V_h^<(\Omega_1),
	\eeq
	then
		\begin{align}
		&\norm{u - u_h}_{\mathcal{Z}_k(\Omega_0)} \leq
		\frac{C}{kd}\bigg(\norm{u}_{\mathcal{Z}_k(\Omega_1)} + \frac{1}{kd} \norm{u}_{\mathcal{H}(\Omega_1)}
	+ \norm{u - u_h}_{\mathcal{H}(\Omega_1)}\bigg).\label{eq:sufficient}
	\end{align}
\ele

\bpf[Proof of Theorem \ref{thm:WDGS1} using Lemma \ref{lem:WDGS}]
We first show using a covering argument  that Lemma \ref{lem:WDGS} implies that the bound \eqref{eq:sufficient} holds for general sets $
\Omega_0\subset\Omega_1\subset\Omega$ satisfying the assumptions of Theorem \ref{thm:WDGS1}, i.e.  
\eqref{eq:WDGS1:assd}.
First, we find a subset $\Omega'_1 \subset \Omega_1$ such that 
\beq
d'=\partial_<(\Omega_0,\Omega'_1)
= \min\bigg\{ \frac{C_0}{2}, \frac{\es{4}\ccoer}{3}\bigg\} \frac{1}{k}.
\label{eq:last_day1}
\eeq
This definition implies that 
$d' \leq \es{4}\ccoer/(3k)$, and also that 
$d' \leq d\es{/2}$ (by the first condition in \eqref{eq:WDGS1:assd}), so that $\Omega'_1$ is indeed a subset of $\Omega_1$. Observe that \eqref{eq:last_day1} and the second condition in \eqref{eq:WDGS1:assd} imply that 
\beqs
d'\geq 4\max\left\{1,8\kappa\right\} \max_{K \cap \Omega_1 \neq \emptyset} h_K;
\eeqs
i.e., the inequalities in \eqref{eq:lemWDGS0} are satisfied with $d$ replaced by $d'$.

Next, we introduce $x_1,\ldots,x_N \in \Omega_0$ such that 
\beq\label{eq:inclusions1}
\Omega_0 \subset \bigcup_{j=1}^N \Big(B(x_j,d'/4) \cap \Omega\Big) \subset \bigcup_{j=1}^N \Big(B(x_j,3d'/4)\cap \Omega\Big) \subset \Omega'_1,
\eeq
and such that the intersection between $m$ distinct balls is empty when $m \geq C$, for some constant $C$ depending only on the space dimension $n$. Note that the intersections with $\Omega$ are needed when $\Omega_0$ is near the boundary of $\Omega$.  

We now apply Lemma \ref{lem:WDGS} with $d$ replaced by $d'$, and thus 
\beqs
\Omega_0= B(x_j,d'/4) \cap \Omega, \quad\tand\quad 
\Omega_1= B(x_j,d'/4) \cap \Omega.
\eeqs
Note that the orthogonality assumption \eqref{eq:GOG1} on the large domain $\Omega_1$ implies the analogous orthogonality  \eqref{eq:GOG1_small} on each ball. Therefore,
\begin{align*}
	&\norm{u - u_h}_{\mathcal{Z}_k(B(x_j,d'/4))} \\
	&\qquad\leq
	\frac{C}{kd'}\bigg(\norm{u}_{\mathcal{Z}_k(B(x_j,3d'/4))} + \frac{1}{kd'} \norm{u}_{\mathcal{H}(B(x_j,3d'/4))}
	+ \norm{u - u_h}_{\mathcal{H}(B(x_j,3d'/4))}\bigg).
\end{align*}
Summing with respect to $j$ and using \eqref{eq:lem_balls1}-\eqref{eq:lem_balls2},
\begin{align}
	\norm{u-u_h}_{\cZ_k(\Omega_0)} \leq \frac{C}{kd'}\bigg(\norm{u}_{\mathcal{Z}_k(\Omega'_1)} + \frac{1}{kd'} \norm{u}_{\mathcal{H}(\Omega'_1)}
	+ \norm{u - u_h}_{\mathcal{H}(\Omega'_1)}\bigg) 
	\label{eq:last_day2}
\end{align}
By \eqref{eq:last_day1}, the instances of $(kd')^{-1}$ on the right hand side of \eqref{eq:last_day2} are bounded by a constant; then, by \eqref{eq:subset_norm}, 
	\begin{align}
		&\norm{u - u_h}_{\mathcal{Z}_k(\Omega_0)} \leq
		C \Big( \norm{u}_{\mathcal{Z}_k(\Omega_1)} 
	+ \norm{u - u_h}_{\mathcal{H}(\Omega_1)}\Big).\label{eq:sufficient0}
	\end{align}
	
	To obtain \eqref{eq:WDGS1} from \eqref{eq:sufficient0}, we observe that if $u\in \cZ_k$ and $u_h\in V_h$ satisfy the assumptions of the theorem, then so do $\widetilde{u}:= u-w_h \in\cZ_k$ and $\widetilde{u}_h:= u_h -w_h \in V_h$, where $w_h \in V$ is arbitrary. Indeed, the key point is that $u$ and $u_h$ enter the assumptions of the theorem only via $u-u_h$, and $u-u_h = \widetilde{u}-\widetilde{u}_h$.
		Therefore, in \eqref{eq:sufficient0}, the norms of $u$ on the right-hand side can be replaced by the norms of $u-w_h$ for arbitrary $w_h\in V_h$, and this gives \eqref{eq:WDGS1}.
\epf

\

\bpf[Proof of Lemma \ref{lem:WDGS}]
Let 
\beqs
\Omega_{1/2}:= B(x_0,d/2)\cap \Omega
\eeqs
and let $\chi = \psi(\Omega_{1/2},\Omega_1)$ be the localiser associated to $\Omega_{1/2},\Omega_1$ via Assumption \ref{ass:commut} (with $(r,d)= (d/2,d/4)$).

	Let the operator $\Pi_h: \mathcal{Z}_k \to V_h^{<}(\Omega_1)$ be defined as the solution of the variational problem
	\beq\label{eq:proj1}
a_k(\Pi_h \zeta,v_h) = a_k(\zeta,v_h)\quad\tfa v_h \in V_h^<(\Omega_1).
	\eeq
Since $3d/4 < \ccoer k^{-1}$ (by \eqref{eq:lemWDGS0}),
$a_k$ is continuous and coercive on $V_h^<(\Omega_1)$ by \eqref{eq:contAb} and \eqref{eq:coercivity1}, and thus $\Pi_h$ is well-defined by the Lax-Milgram lemma.

	By the definition of $\chi$ and the triangle inequality, 
\begin{align}\nonumber
\norm{u - u_h}_{\mathcal{Z}_k(\Omega_0)} = \norm{\chi u - u_h}_{\mathcal{Z}_k(\Omega_0)} &\leq \norm{\chi u - \Pi_h(\chi u)}_{\mathcal{Z}_k(\Omega_0)} + \norm{\Pi_h(\chi u) - u_h}_{\mathcal{Z}_k(\Omega_0)} \\
& \leq \norm{(\Id - \Pi_h)(\chi u)}_{\mathcal{Z}_k}+ \norm{z_h}_{\mathcal{Z}_k(\Omega_0)},\label{eq:ineq6}
	\end{align}
	where $z_h := \Pi_h(\chi u) - u_h $. 
	To bound the first term on the right-hand side of \eqref{eq:ineq6}, we use C\'ea's lemma, which follows from the continuity and coercivity of $a_k$ on $V_h^<(\Omega_1)$, to obtain
		\begin{align}
			\norm{(\Id - \Pi_h)(\chi u)}_{\mathcal{Z}_k} \leq C\inf_{w_h \in V_h^<(\Omega_1)} \norm{\chi u - w_h}_{\mathcal{Z}_k}
			& \leq  C \norm{\chi u}_{\mathcal{Z}_k} \leq  C\Big(\norm{u}_{\mathcal{Z}_k(\Omega_1)} + \frac{1}{kd} \norm{u}_{\mathcal{H}(\Omega_1)}\Big),\label{eq:ineq4}
		\end{align}
where we have used \eqref{eq:improvedLeibniz} in the last inequality.

We now bound $\norm{z_h}_{\mathcal{Z}_k(\Omega_0)}$ in \eqref{eq:ineq6} using the Caccioppoli inequality \eqref{eq:Caccioppoli} applied with $U_0=\Omega_0$ and $U_1=\Omega_{1/2}$. The distance between these two sets is $d/4$, and so the condition \eqref{eq:dist4h} is ensured by the second  condition in \eqref{eq:lemWDGS0} since $K\cap \Omega_{1/2}\subset K\cap \Omega_1$.  
The condition \eqref{eq:ccoer} becomes that $d/4 +d/8 \leq \ccoer/ k$, and is thus ensured by the first condition in \eqref{eq:lemWDGS0}. Then the definition $z_h:= \Pi_h(\chi u) - u_h$, the definition of $\Pi_h$ \eqref{eq:proj1},
the fact that $\chi\equiv 1$ on $\Omega_{1/2}$, 
\es{the locality of $a_k$ (\eqref{ass:loc} in Assumption \ref{ass:bn})},
 and the orthogonality \eqref{eq:GOG1_small} imply that
\beqs
a(z_h,v_h)=a\big(\Pi_h(\chi u)-u_h, v_h\big)=a(\chi u-u_h, v_h)= 0 \quad\tfa v_h \in V_h^<(\Omega_{1/2}),
\eeqs
i.e., \eqref{eq:orthoCaccio} holds (with $U_1=\Omega_{1/2}$). Therefore, by \eqref{eq:Caccioppoli}, 
		\beq\label{eq:zh1}
		\norm{z_h}_{\mathcal{Z}_k(\Omega_0)} \leq \frac{C}{kd} \norm{z_h}_{\mathcal{H}(\Omega_{1/2})}.
		\eeq
Using (in this order) the definition of $z_h$, the triangle inequality, 
the fact that the $\mathcal{Z}_k$ norm is stronger than the $\mathcal{H}$ norm, and \eqref{eq:ineq4}, we find that 
		\begin{align}\nonumber
			\norm{z_h}_{\mathcal{H}(\Omega_{1/2})} \leq \norm{u - u_h}_{\mathcal{H}(\Omega_{1/2})} + \norm{u - \Pi_h(\chi u)}_{\mathcal{H}(\Omega_{1/2})}
			&\leq  C\Big(\norm{u - u_h}_{\mathcal{H}(\Omega_1)} +
			\norm{(\Id - \Pi_h)(\chi u)}_{\mathcal{Z}_k}\Big)\\ 
			&\hspace{-1cm} \leq C \Big(\norm{u - u_h}_{\mathcal{H}(\Omega_1)} + \norm{u}_{\mathcal{Z}_k(\Omega_1)} + \frac{1}{kd}\norm{u}_{\mathcal{H}(\Omega_1)}\Big). \label{eq:zh2}
		\end{align}
		
The bound \eqref{eq:sufficient} then follows from combining \eqref{eq:ineq6}, \eqref{eq:ineq4}, \eqref{eq:zh1}, and \eqref{eq:zh2}.		
\epf

\ble[Analogue of Theorem \ref{thm:WDGS2} for small balls close together]
\label{lem:WDGS2}
	Given positive constants $\Ccont$, $\Ccoer$, $\ccoer$, $\Ckappa$, $\Cp$, $\Cinv$, $\Cpw$, $\Csuper$, $\Ccom, \Cqu> 0$, there exists a constant $C_\star > 0$ such that the following holds. Let $(a_k)_{k \geq k_0}$ and $V_h \subset \mathcal{Z}_k$ satisfy Assumptions \ref{ass:cont_coer}, \ref{ass:er}, \ref{ass:commut}, \ref{ass:ap}, \ref{ass:sa}, \ref{ass:ii}, and \ref{ass:loc}, with the constants above. 
For $x_0\in \Omega$, let $\Omega_0$ and $\Omega_1$ be as in \eqref{eq:choice2}
with 	
	\begin{equation}\label{eq:lemWDGS2}
d\leq \frac{2\ccoer}{k}
		\quad \tand \quad 
		d\geq 48 (\ell+2)\max\left\{1,8\kappa\right\} \max_{K \cap \Omega_1 \neq \emptyset} h_K.
	\end{equation}
Assume further that
\beqs
		\frac{\max_{K \cap \Omega_1 \neq \emptyset} h_K}{\min_{K \cap \Omega_1 \neq \emptyset} h_K} \leq \Cqu.
\eeqs
If $k \geq k_0$, $u \in \mathcal{Z}_k$, and $u_h \in V_h$ are such that \eqref{eq:GOG1_small} holds,
	then
		\begin{align}
		&\norm{u - u_h}_{\mathcal{Z}_k(\Omega_0)} \leq
		 \frac{C}{kd}\bigg(\norm{u}_{\mathcal{Z}_k(\Omega_1)} + \frac{1}{kd} \norm{u}_{\mathcal{H}(\Omega_1)}
	+ \Wjdual{u-u_h}{\newell+1}{\Omega_1}\bigg),
	\end{align}
where $\newell:= \min\{\ell,p-1\}$
\ele

\bpf[Proof of Theorem \ref{thm:WDGS2} using Lemma \ref{lem:WDGS2}]
This exactly parallels the proof of Theorem \ref{thm:WDGS1} using Lemma \ref{lem:WDGS}, with the first step choosing
$\Omega'_1 \subset \Omega_1$ such that 
\beqs
d'=\partial_<(\Omega_0,\Omega'_1)
= \min\bigg\{ \frac{C_0}{2}, {2\ccoer}\bigg\} \frac{1}{k}.
\eeqs
\epf

\bpf[Proof of Lemma \ref{lem:WDGS2}]
This exactly parallels the proof of Lemma \ref{lem:WDGS}, except that now (i)  we use Lemma \ref{lem:CaccioppoliNegative2} instead of 
Lemma \ref{lem:Caccioppoli}, and (ii) when obtaining the analogue of \eqref{eq:last_day2} we use additionally \eqref{eq:last_day2addition}.
\epf

\section{Examples of Helmholtz problems fitting in the abstract framework}\label{sec:examples}

\paragraph{Summary.}
In \S\ref{sec:scattering}-\ref{sec:verifySesqui}, we show that the general framework in which Theorems \ref{thm:WDGS1} and \ref{thm:WDGS2} hold includes 
\bit
\item truncation of the unbounded exterior domain by \emph{either} a PML, \emph{or} an impedance boundary condition, \emph{or} the exact Dirichlet-to-Neumann map for the exterior of a ball,
\item scattering by Dirichlet or Neumann impenetrable obstacles, and
\item scattering by penetrable obstacles.
\eit
In \S\ref{sec:verifySubspace} we show that the assumptions on the finite-dimension subspace $V_h$ are satisfied for shape-regular Lagrange finite elements.

\subsection{Definitions of the sesquilinear forms $a(\cdot,\cdot)$ and spaces $\cZ_k$} \label{sec:scattering}

\subsubsection{The geometry and coefficients for scattering by a combination of an impenetrable Dirichlet or Neumann obstacle and a penetrable obstacle}

Let $\Ot,\OI\subset B_{R_0}:= \{ x : |x| < R_0\}\subset \Rea^d$, $d=2,3$, be bounded open sets with Lipschitz boundaries, $\Gt$ and $\GI$, respectively, such that $\Gt\cap \GI=\emptyset$ and $\Omega_+:=\Rea^d\setminus\overline{\OI}$ is connected. Let 
$\Omegaout:= \Omega_+ \setminus \overline{\Ot}$ and $\Omegain:= \Omega_+\cap \Ot$.

\es{The obstacle $\Ot$ is the penetrable obstacle, across whose boundary we impose transmission conditions, and $\Omega_-$ is the impenetrable obstacle, on which we impose either a zero Dirichlet or a zero Neumann condition. 
The condition $\Gt\cap \GI=\emptyset$ allows two configurations:~the first, illustrated in Figure \ref{fig:two_cases}(a), is when the penetrable and impenetrable obstacles are disjoint. The second, illustrated in Figure \ref{fig:two_cases}(b), is when the impenetrable obstacle is inside the penetrable obstacle.}

For simplicity, we do not cover the case when 
$\Omega_-$ is disconnected,  with Dirichlet boundary conditions on some connected components and Neumann boundary conditions on others, but the main results hold for this problem too (at the cost of introducing more notation).

\begin{figure}[h!]
    \centering
    \begin{subfigure}{0.5\textwidth}
        \centering
        \scalebox{0.8}{
\begin{tikzpicture}

\def\radiusA{2cm}
\def\radiusB{1.5cm}
\def\radiusC{1cm}

\coordinate (O1) at (0,0);
\coordinate (O2) at (5cm,0);

\draw[name path=circleA] (O1)node{$\Omega_-$} circle (\radiusA);

\filldraw[name path=circleB,lightgray] (O2)node{$\Ot=\Omega_{\rm in}$} circle (\radiusB);
\filldraw[name path=circleB,pattern=north east lines] (O2)node{$\Ot=\Omega_{\rm in}$} circle (\radiusB);




\end{tikzpicture}
}
\caption{}
    \end{subfigure}%
    ~ 
    \begin{subfigure}{0.5\textwidth}
        \centering

        \scalebox{0.8}{
\hspace{1cm}
\begin{tikzpicture}

\pgfdeclarepatternformonly{wide lines}{\pgfpoint{-1cm}{-1cm}}{\pgfpoint{1cm}{1cm}}{\pgfpoint{1cm}{1cm}}
{
    \pgfsetdash{{3pt}{3pt}}{0pt} 
    \pgfsetlinewidth{0.8pt}
    \pgfpathmoveto{\pgfpoint{-0.5cm}{-0.5cm}}
    \pgfpathlineto{\pgfpoint{0.5cm}{0.5cm}}
    \pgfusepath{stroke}
}

\def\radiusA{2cm}
\def\radiusB{1cm}

\coordinate (O1) at (0,0);
\coordinate (O2) at (0,0);

\filldraw[name path=circleA, lightgray, opacity=1] (O1)node{$\Omega_-$} circle (\radiusA);
\filldraw[name path=circleB,white] (O2) circle (\radiusB);
\filldraw[name path=circleA, pattern = north east lines] (O1) circle (\radiusA);

\filldraw[name path=circleB,white, opacity=.5] (O2) circle (\radiusB);

\filldraw[name path=circleA, pattern = north east lines] (\radiusA+30,\radiusA+10) rectangle (\radiusA+10,\radiusA-10);
\draw  (\radiusA +30,\radiusA) node[right]{$\Ot$};

\draw (O2) node{$\Omega_{-}$};

\draw (O1)++(90:1.5cm) node{$\Omega_{\rm in}$};




\end{tikzpicture}
}

        \caption{}
    \end{subfigure}
    \caption{\es{The two configurations of $\Omega_-$ and $\Ot$ considered. In both cases, $\Ot$ is hatched, and $\Omegain$ is uniformly shaded.}}\label{fig:two_cases}
\end{figure}

Let  $\Ascatout \in C^{0,1} (\Omegaout , \Rea^{d\times d})$ 
and $\Ascatin \in C^{0,1}(\Omegain, \Rea^{d\times d})$ 
be symmetric positive definite, let $\cscatout \in L^\infty(\Omegaout;\Rea)$, $\cscatin \in L^\infty(\Omegain;\Rea)$ 
be strictly positive, and let $\Ascatout$ and $\cscatout$ be such that 
 there exists $R_{\rm scat}>R_0>0$ such that 
\beqs
\overline{\Omega_-} \cup {\rm supp}(I- \Ascatout) \cup {\rm supp}(1-\cscatout) \Subset B_{R_{\rm scat}}.
\eeqs
Let 
\beqs
\Ascat := 
\begin{cases}
\Ascatin 
\hspace{-1ex}
& \tin \Omegain,\\
\Ascatout 
\hspace{-1ex}
& \tin \Omegaout, 
\end{cases}
\quad\tand\quad
\frac{1}{\cscat^2} := 
\begin{cases}
\cscatin^{-2} 
\hspace{-1ex}
& \tin \Omegain,\\
\cscatout^{-2} 
\hspace{-1ex}
& \tin \Omegaout 
\end{cases}.
\eeqs

\subsubsection{The scattering problem}
Given $g\in L^2(\Omega_+)$ \es{(where recall that $\Omega_+ := \Rea^d \setminus \overline{\Omega_-}$)}
with $\supp \, g\subset B_{\Rscat}$ and $k,\gamma>0$, let $u=(\uin,\uout)$ be the solution of 
\begin{subequations}\label{eq:olddognewtricks}
\begin{align}
 k^{-2}\cscatout^{2}\nabla\cdot (\Ascatout\nabla \uout) + \uout &=- g \quad\tin \Omegaout,\\
 k^{-2}\cscatin^{2}\nabla\cdot (\Ascatin\nabla \uin) + \uin &=- g \quad\tin \Omegain,\\
&\hspace{-4cm}\uin= \uout \quad\tand\quad\partial_{n, \Ascatin} \uin = \gamma \partial_{n,\Ascatout} \uout \quad\ton \Gt, \label{eq:transmission}\\
&\hspace{-3.5cm}\text{ either } \quad\uin = 0 \quad\text{ or }\quad \partial_{n,\Ascatin}\uin=0 \quad\ton \GI \quad \text{ if } \Omega_- \subset \Omegain, \text{ or} ,\\
&\hspace{-3.5cm}\text{ either } \quad\uout = 0 \quad\text{ or }\quad \partial_{n,\Ascatout}\uout=0 \quad\ton \GI \quad \text{ if } \Omega_- \subset \Omegaout, 
\end{align}
\end{subequations}
where 
$\uin \in H^1(\Omegain)$,  $\uout \in H^1(\Omegaout\cap B_R)$ for every $R>0$, and $\uout$ satisfies the Sommerfeld radiation condition 
 \beq\label{eq:src}
k^{-1} \pdiff{\uout}{r}(x) - \ri  \uout(x) = o \Big(\frac{1}{r^{(d-1)/2}}\Big)
 \eeq
as $r:= |x|\tendi$ (uniformly in $\widehat{x}:= x/r$).
The solution of this problem exists and is unique; see, e.g., \cite{GrPeSp:19} and the references therein.

\subsubsection{The variational formulation}

Given $R> \Rscat$, let 
\beqs
\Omega:= \Omegain\cup (\Omegaout\cap B_R)
\eeqs
 and let
\beq\label{eq:Z1}
\cZ_k:= \Big\{ v \in H^1(\Omega)
\,:\, v=0 \text{ on } \Gamma_-\Big\} \quad \text{ or } \quad H^1(\Omega),
\eeq
in both cases equipped with the $H^1_k$ norm defined by \eqref{eq:defWeightedNorms}, 
with the former space corresponding to zero Dirichlet boundary conditions on $\Gamma_-$ and the latter corresponding to zero Neumann boundary conditions on $\Gamma_-$.

Let 
$\DtN: H^{1/2}(\GR)\to H^{-1/2}(\GR)$ be the Dirichlet-to-Neumann map, $u\mapsto k^{-1} \partial_r u$, for the Helmholtz equation $(k^{-2}\Delta +1)u=0$ posed in the exterior of $B_R$ and satisfying the Sommerfeld radiation condition \eqref{eq:src}; i.e., when $d=2$, 
given $g\in H^{1/2}(\partial B_R)$, 
\beq\label{eq:DtN}
\DtN g (\varphi):= \frac{1}{2\pi} \sum_{n=-\infty}^\infty \frac{H^{(1)'}_n(kR)}{ H^{(1)}_n(kR)} \exp(\ri n \varphi)\int_0^{2\pi} \exp(-\ri n\theta) g(R, \theta)\, \rd \theta;
\eeq
for the analogous expression when $d=3$, see, e.g., 
e.g., \cite[Equation 3.6]{ChMo:08}, 
\cite[Equation 3.7]{MeSa:10}. 

The variational formulation of the scattering problem \eqref{eq:olddognewtricks}-\eqref{eq:src} is 
\beq\label{eq:HelmholtzVF}
\text{ find } u \in \cZ_k \,\tst\, a(u,v) = G(v) \,\tfa v \in \cZ_k,
\eeq
where 
\beq\label{eq:sesqui1}
a(u,v) := 
\left(\int_{\Omegaout\cap B_R}  + \frac{1}{\gamma}\int_{\Omegain} \right)
\Big(k^{-2}(\Ascat\nabla u) \cdot\overline{\nabla v} - \cscat^{-2} u \overline{v}\Big)
- k^{-1}\big\langle \DtN u, v\big\rangle_{\partial B_R}
\eeq
and
\beq\label{eq:HelmholtzG}
G(v):= 
\left(\int_{\Omegaout\cap B_R}  + \frac{1}{\gamma}\int_{\Omegain} \right)
c^{-2}g \overline{v}.
\eeq

\subsubsection{Approximation of $\DtN$ by an impedance boundary condition}

A commonly-used approximation of $\DtN$ is to impose that $k^{-1}\partial_r u = \ri u$ on $\partial B_R$, i.e., impose an impedance boundary condition. In this case, one also often removes the requirement that the outer boundary is a ball. 
Let $R_\tr> \Rscat$, let $\Omega_{\tr}\subset \mathbb{R}^d$ be a bounded Lipschitz open set with $B_{\Rtr }\subset \Omega_{\tr} \subset  B_{C\Rtr }$ for some $C>0$ (i.e., $\Omega_{\tr}$ has characteristic length scale $\Rtr $), and let $\Gamma_{\tr}:=\partial\Omega_{\tr}$. 
The impedance problem is \eqref{eq:HelmholtzVF} with now 
\eqref{eq:sesqui1} replaced by 
\beq\label{eq:sesqui2}
a(u,v) := 
\left(\int_{\Omegaout\cap \Omega_{\rm tr}}  + \frac{1}{\gamma}\int_{\Omegain} \right)
\Big(k^{-2}(\Ascat\nabla u) \cdot\overline{\nabla v} - \cscat^{-2} u \overline{v}\Big)
- \ri k^{-1}\langle u, v\rangle_{\Gamma_{\tr}},
\eeq
\eqref{eq:HelmholtzG} replaced by the analogous expression with integration over $\Omegaout\cap B_R$ replaced by integration over $\Omegaout\cap \Omega_{\rm tr}$, and $\cZ_k$ still defined by  \eqref{eq:Z1}, but now with 
\beq\label{eq:OmegaImpPML}
\Omega= \Omegain\cup(\Omegaout\cap \Omega_\tr).
\eeq
See \cite{GLS1} for $k$-explicit bounds on the error incurred by this approximation (showing, in particular, how the error depends on $\partial \Omega_{\rm tr}$).

\subsubsection{Approximation of $\DtN$ by a radial PML}\label{sec:PML}

Let $\Rtr >\RPMLo>R_{\rm scat}$, let $\Omega_{\rm tr}$ and $\Gamma_{\tr}:=\partial\Omega_{\rm tr}$ be as above, and let 
 $\Omega$ be defined by \eqref{eq:OmegaImpPML}.
For $0\leq \theta<\pi/2$, let the PML scaling function $f_\theta\in C^{1}([0,\infty);\mathbb{R})$ be defined by $f_\theta(r):=f(r)\tan\theta$ for some $f$ satisfying
\begin{equation}
\label{e:fProp}
\begin{gathered}
\big\{f(r)=0\big\}=\big\{f'(r)=0\big\}=\big\{r\leq \RPMLo\big\},\quad f'(r)\geq 0,\quad f(r)\equiv r \text{ on }r\geq \RPMLt;
\end{gathered}
\end{equation}
i.e., the scaling ``turns on'' at $r=\RPMLo$, and is linear when $r\geq \RPMLt$. We emphasize that $\Rtr $ can be $<\RPMLt$, i.e., we allow truncation before linear scaling is reached. Indeed,
$\RPMLt>\RPMLo$ can be arbitrarily large and therefore, given any bounded interval $[0,R]$ and any function $\widetilde{f}\in C^{1}([0,R])$ satisfying 
\beqs
\begin{gathered}
\big\{\widetilde{f}(r)=0\big\}=\big\{\widetilde{f}'(r)=0\big\}=\big\{r\leq \RPMLo\big\},\qquad \widetilde{f}'(r)\geq 0,
\end{gathered}
\eeqs
we can choose an $f$ with $f|_{[0,R]}=\widetilde{f}$. 
Given $f_\theta(r)$, let 
\beq\label{eq:alpha_beta}
\alpha(r) := 1 + \ri f_\theta'(r) \quad \tand\quad \beta(r) := 1 + \ri f_\theta(r)/r,
\eeq
and let
\beq\label{eq:Ac}
A := 
\begin{cases}
\Ascatin 
\hspace{-1ex}
& \tin \Omegain,\\
\Ascatout 
\hspace{-1ex}
& \tin \Omegaout \cap B_{\RPMLo},\\
HDH^T 
\hspace{-1ex}
&\tin (B_{\RPMLo})^c
\end{cases}
\tand
\frac{1}{c^2} := 
\begin{cases}
\cscatin^{-2} 
\hspace{-1ex}
& \tin \Omegain,\\
\cscatout^{-2} 
\hspace{-1ex}
& \tin \Omegaout \cap B_{\RPMLo},\\
\alpha(r) \beta(r)^{d-1} 
\hspace{-1ex}
&\tin (B_{\RPMLo})^c,
\end{cases}
\eeq
where, in polar coordinates $(r,\varphi)$,
\beq\label{eq:DH2}
D =
\left(
\begin{array}{cc}
\beta(r)\alpha(r)^{-1} &0 \\
0 & \alpha(r) \beta(r)^{-1}
\end{array}
\right) 
\quad\tand\quad
H =
\left(
\begin{array}{cc}
\cos \varphi & - \sin\varphi \\
\sin \varphi & \cos\varphi
\end{array}
\right) 
\tfor d=2,
\eeq
and, in spherical polar coordinates $(r,\phi, \varphi)$,
\beq\label{eq:DH3}
D =
\left(
\begin{array}{ccc}
\beta(r)^2\alpha(r)^{-1} &0 &0\\
0 & \alpha(r) &0 \\
0 & 0 &\alpha(r)
\end{array}
\right) 
\tand
H =
\left(
\begin{array}{ccc}
\sin \phi \cos\varphi & \cos \phi \cos\varphi & - \sin \varphi \\
\sin \phi \sin\varphi & \cos \phi \sin\varphi & \cos \varphi \\
\cos \phi & - \sin \phi & 0 
\end{array}
\right) 
\eeq
for $d=3$. 
Observe that $\Ascatout=I$ and $\cscatout^{-2}=1$ when $r=\RPMLo$ and thus $A$ and $c^{-2}$ are continuous at $r=\RPMLo$.

We highlight that, in other papers on PMLs, the scaled variable, which in our case is $r+\ri f_\theta(r)$, is often written as $r(1+ \ri \widetilde{\sigma}(r))$ with $\widetilde{\sigma}(r)= \sigma_0$ for $r$ sufficiently large; see, e.g., \cite[\S4]{HoScZs:03}, \cite[\S2]{BrPa:07}. Therefore, to convert from our notation, set $\widetilde{\sigma}(r)= f_\theta(r)/r$ and $\sigma_0= \tan\theta$.

Let $\cZ_k$ be still defined by  \eqref{eq:Z1}, but now with $\Omega$ given by \eqref{eq:OmegaImpPML}.
Given $g\in L^2(\Omega)$ with $\supp \, g\subset B_{\Rscat}$, 
a variational formulation of the PML problem is then \eqref{eq:HelmholtzVF} with 
\beq\label{eq:sesqui3}
a(u,v) := 
\left(\int_{\Omega\cap \Omegaout}  + \frac{1}{\gamma}\int_{\Omega\cap\Omegain} \right)
\Big(k^{-2}(A\nabla u) \cdot\overline{\nabla v} - c^{-2} u \overline{v}\Big)
\eeq
and $G(v)$ given by \eqref{eq:HelmholtzG};
this variational formulation is obtained by multiplying the PDEs in \eqref{eq:olddognewtricks} by $c^{-2}_{\rm in/out} \alpha \beta^{d-1}$ and integrating by parts.

\begin{assumption}\label{ass:PML}
When $d=3$, $f_\theta(r)/r$ is nondecreasing.
\end{assumption}

Assumption \ref{ass:PML} is standard in the literature; e.g., in the alternative notation described above it is that 
$\widetilde{\sigma}$ is non-decreasing -- see \cite[\S2]{BrPa:07}.
We record for later the following sign property of $A$ under Assumption \ref{ass:PML}.

\begin{lemma}\label{lem:strong_elliptic}
Suppose that $f_\theta$ satisfies Assumption \ref{ass:PML}.
With $A$ defined by \eqref{eq:Ac}, given $\epsilon>0$
there exists $\Amin>0$ such that, 
for all $\epsilon \leq \theta\leq \pi/2-\epsilon$,
\beqs
\Re \big( A(x) \xi, \xi\big)_2 \geq \Amin \|\xi\|_2^2 \quad\tfa \xi \in \mathbb{C}^d \tand x \in \Omega.
\eeqs
\end{lemma}

\bpf[Reference for the proof]
See, e.g., \cite[Lemma 2.3]{GLSW1}.
\epf

\bre[Existence and uniqueness of the solution of the PML problem]
Using the fact that the solution of the true scattering problem exists and is unique with $\Ascatout,\Ascatin,\cscatout,\cscatin, \Omega_-,$ and $\Omegain$ described above, 
the solution of the PML variational formulation above exists and is unique (i) for fixed $k$ and sufficiently large $\Rtr-R_1$ by \cite[Theorem 2.1]{LaSo:98}, \cite[Theorem A]{LaSo:01}, \cite[Theorem 5.8]{HoScZs:03} and (ii) for fixed $\Rtr>R_1$ and sufficiently large $k$ by \cite[Theorem 1.5]{GLS2} under the additional assumption that $f_\theta\in C^3$.
\ere

\bre[Accuracy of the PML approximation]
For the particular data $G$ \eqref{eq:HelmholtzG} (i.e., coming from a function supported in $B_{\Rscat}$), it is well-known that, for fixed $k$, the error $\|u-v\|_{H^1_k(B_{\RPMLo}\setminus \Omega)}$ decays exponentially in $R_{\rm tr}-\RPMLo$ and $\tan\theta$; see 
\cite[Theorem 2.1]{LaSo:98}, \cite[Theorem A]{LaSo:01}, \cite[Theorem 5.8]{HoScZs:03}.
It was recently proved in \cite[Theorems 1.2 and 1.5]{GLS2} that the error $\|u-v\|_{H^1_k(B_{\RPMLo}\setminus \Omega)}$ also decreases exponentially in $k$ (again under the assumption that $f_\theta\in C^3$).
\ere

\subsubsection{Summary of the sesquilinear forms $a(\cdot,\cdot)$ and spaces $\cZ_k$}

For truncation by the exact Dirichlet-to-Neumann map and the truncation boundary equal to the boundary of a ball, the sesquilinear form $a(\cdot,\cdot)$ is defined by \eqref{eq:sesqui1} and the space $\cZ_k$ is defined by \eqref{eq:Z1} with $\Omega:= \Omegain \cup(\Omegaout\cap B_R)$.

For truncation by an impedance boundary condition, the sesquilinear form $a(\cdot,\cdot)$ is defined by \eqref{eq:sesqui2} and the space $\cZ_k$ is defined by \eqref{eq:Z1} with $\Omega:= \Omegain \cup(\Omegaout\cap \Omega_\tr)$.

For truncation by a perfectly-matched layer, the sesquilinear form $a(\cdot,\cdot)$ is defined by \eqref{eq:sesqui3} and the space $\cZ_k$ is defined by \eqref{eq:Z1}
with $\Omega:= \Omegain \cup(\Omegaout\cap \Omega_\tr)$.

\subsection{The spaces $\cZ^j_k$ and $\cW^j_k$ and Assumption \ref{ass:balls}}\label{sec:spaces}

With \emph{either} $\Omega:= \Omegain \cup(\Omegaout\cap B_R)$ (for DtN truncation) 
\emph{or} $\Omega:= \Omegain \cup(\Omegaout\cap \Omega_\tr)$ (for impedance or PML truncation), 
let 
\beq\label{eq:WZ}
\cW^j_k :=
L^2(\Omega)\cap \big( H^j_k(\Omega_{\rm in})\oplus H^j_k(\Omega_{\rm out}\cap \Omega)\big)
\quad\tand\quad
\cZ^j_k :=\cW^j_k \cap \cZ_k.
\eeq

The embedding inequality \eqref{e:embedW} immediately holds, and Assumption \ref{ass:balls} holds via standard properties of the $H^j_k$ norm.

\subsection{The assumptions on the sesquilinear form (Assumptions \ref{ass:cont_coer}, \ref{ass:er}, and \ref{ass:commut})}\label{sec:verifySesqui}


\ble[Satisfying Assumption \ref{ass:cont_coer}]
Assumption \ref{ass:cont_coer} is satisfied for the sesquilinear forms defined by \eqref{eq:sesqui1}, \eqref{eq:sesqui2}, and \eqref{eq:sesqui3}.
\ele

\bpf
We first establish the continuity property \eqref{eq:contAb}. For the sesquilinear form \eqref{eq:sesqui3} from PML truncation, 
continuity follows by the Cauchy--Schwarz inequality.
For the sesquilinear form \eqref{eq:sesqui2} from impedance truncation, continuity follows by the Cauchy--Schwarz inequality, and the weighted trace inequality 
\beqs
\N{v}_{L^2(\partial D)} \leq  C k^{1/2} \N{v}_{H^1_k(D)};
\eeqs
see, e.g., \cite[Theorem 1.5.1.10, last formula on page 41]{Gr:85}.
For the sesquilinear form \eqref{eq:sesqui1} from truncation by $\DtN$, continuity follows by the Cauchy--Schwarz inequality, and the inequality
\beqs
k^{-1}\big|\langle \DtN u,v\rangle_{\partial B_R}\big| \leq C  \N{u}_{\cZ_k}\N{v}_{\cZ_k} \quad\tfa u,v\in \cZ_k,
\eeqs
which holds by, e.g., 
\cite[Equation 3.4a]{MeSa:10} (taking into account that \cite{MeSa:10} use a different $k$-weighting in the $H^1$ norm to \eqref{eq:1knorm} -- see the comments after \eqref{eq:1knorm}).

For the local coercivity \eqref{eq:coercivity1}, we claim that, for all three sesquilinear forms, given $\Ascat,\cscat$, and $\gamma>0$, there exist $C_1, C_2>0$ such that 
\beq\label{eq:Garding}
\Re\big\{ a_k(v,v)\big\} \geq C_1 \N{v}^2_{\cZ_k} - C_2 \N{v}_{\cH}^2 \quad \tfa v\in \cZ_k \text{ and for all } k\geq k_0.
\eeq
Once \eqref{eq:Garding} is established, the local coercivity follows from the Poincar\'e--Friedrichs inequality. Indeed, 
\beqs
\N{v}^2_{\cH}\leq \CPF (kr)^2 \N{k^{-1}\nabla v}^2_{\cH}
\quad \tfa v\in \cZ^{<}_k(B(x_0,r)\cap\Omega),
\eeqs
so that if $C_1/2 \geq  C_2 \CPF(kr)^2$ then 
\beqs
\Re\big\{ a_k(v,v)\big\} \geq \frac{C_1}{2} \N{v}^2_{\cZ_k} \quad\tfa v\in \cZ_k \text{ and for all } k\geq k_0,
\eeqs
where we have used that $\|\cdot\|_{\cZ_k}= \|\cdot\|_{H^1(\Omega)}$ in all three cases. Thus, \eqref{eq:coercivity1} holds with 
\beqs
\ccoer:= \left(\frac{C_1}{2 C_2\CPF}\right)^{1/2} \quad\tand\quad \Ccoer:= \frac{C_1}{2}.
\eeqs

For the sesquilinear form \eqref{eq:sesqui2} from impedance truncation, the proof of \eqref{eq:Garding} is immediate.
For the sesquilinear form \eqref{eq:sesqui2} from PML truncation, the proof follows from Lemma \ref{lem:strong_elliptic}. 
For the sesquilinear form \eqref{eq:sesqui1} from truncation by $\DtN$, the proof follows from the inequality $\Re \langle \DtN \phi, \phi\rangle_{\partial B_R} \leq 0$ for all $\phi\in H^{1/2}(\partial B_R)$; see \cite[Second inequality in Equation 2.8]{ChMo:08}, \cite[Equation 3.4b]{MeSa:10}. 
\epf

\ble[Satisfying Assumption \ref{ass:er}]
Suppose that 
\es{$\Ascatout, \cscatout \in C^{\ell,1}(\overline{\Omegaout})$, 
$\Ascatin, \cscatin \in C^{\ell,1}(\overline{\Omegain})$, 
} 
the PML scaling function $f_\theta$ is $C^{\ell+1,1}(\overline{\Omega})$,
and both $\partial \Omega$ and $\Gt$ are $C^{\ell+1,1}$. 
Then Assumption \ref{ass:er} is satisfied for the sesquilinear forms defined by \eqref{eq:sesqui1}, \eqref{eq:sesqui2}, and \eqref{eq:sesqui3}. 
\ele

\bpf
By the definition of $\cW^j_k$ and $\cZ^j_k$, the required bound \eqref{eq:ellipticRegB} is
\begin{align}\nonumber
&\N{u}_{H^{j+2}_k(U_0 \cap \Omegain)} + \N{u}_{H_k^{j+2}(U_0 \cap \Omegaout)} \\
&\leq C \Big(\N{u}_{L^2(U_1)} + \N{ k^{-2}\nabla \cdot(A \nabla u) + c^{-2}u }_{H^j_k(U_1\cap \Omegain)} 
+ \N{k^{-2}\nabla\cdot(A \nabla u) + c^{-2}u}_{H^j_k(U_1\cap \Omegaout)}
\Big)
\label{eq:burrito_required}
\end{align}
for all $u\in \mathcal{Z}_k^<(U_1)$ and $j=0,\ldots,\ell$,
where $u$ satisfies the transmission conditions \eqref{eq:transmission} across $\Gt$, either a Dirichlet or Neumann boundary condition on $\Gamma_-$, and either a Dirichlet or impedance boundary condition on $\Gamma_{\tr}$.

By assumption, the characteristic length scales of both $U_0$ and $U_1$ are proportional to $k^{-1}$; denote this length scale (temporarily) by $L$.

We now claim that 
there exists $C>0$ such that,  for all $u\in \mathcal{Z}_k^<(U_1)$ and $j=0,\ldots,\ell$,
\begin{align}\nonumber
&\N{u}_{H^{j+2}(U_0\cap \Omegain)} + \N{u}_{H^{j+2}(U_0\cap \Omegaout)} \\
&\quad\leq C \Big(
L^{-j-1}\N{u}_{H^1(U_1)} + \N{ \nabla \cdot(A \nabla u) }_{H^j(U_1\cap \Omegain)} 
+ \N{\nabla\cdot(A \nabla u)}_{H^j(U_1\cap \Omegaout)}
\Big).
\label{eq:burrito1}
\end{align}
Indeed, this result without the $L$ dependence is proved 

(i) away from the boundary in, e.g., \cite[Theorems 4.7, 4.16]{Mc:00}, 

(ii) locally next to a Dirichlet or Neumann boundary  in \cite[Theorem 4.18]{Mc:00}, 

(iii) locally next to a transmission boundary with $\gamma=1$ in \cite[Theorem 4.20]{Mc:00} and for general $\gamma$ in 
\cite[Theorem 5.2.1(i)]{CoDaNi:10}, 

(iv) locally next to an impedance boundary in \cite[Theorem 3.4]{GS3}, and 

\es{(v) locally next to $\partial B_R$ on which $k^{-1}\partial_n u =  \DtN u $ in \cite[Theorem 3.5]{GS3}.}

In all cases, the $L$-dependence can be inserted, either by keeping track of the constants in the proofs, or by a scaling argument
(using the fact that the constant in the bound on the $\mathcal{O}(1)$ domain depends only on the $C^{\ell,1}$ norms of $\Ascatout,\Ascatin, \cscatout$, and $\cscatin$, the $C^{\ell+1,1}$ norm of $f_\theta$, and the $C^{\ell+1,1}$ norms of $\partial \Omega$ and $\Gt$).

Multiplying \eqref{eq:burrito1} by $k^{-(j+2)}$, we obtain that  
\begin{align*}\nonumber
&k^{-j-2}|u|_{H^{j+2}(U_0 \cap \Omegain)} + k^{-j-2}|u|_{H^{j+2}(U_0 \cap \Omegaout)} \\
&\quad\leq C \Big(
(kL)^{-j-1}\N{u}_{H^1_k(U_1)} + \N{ k^{-2}\nabla \cdot(A \nabla u) }_{H^j_k(U_1\cap \Omegain)} 
+ \N{k^{-2}\nabla\cdot(A \nabla u)}_{H^j_k(U_1\cap \Omegaout)}
\Big)
\end{align*}
for $j=0,\ldots,\ell$,
and thus, using that $L= Ck^{-1}$, 
\begin{align}\nonumber
&\N{u}_{H^{j+2}_k(U_0 \cap \Omegain)} + \N{u}_{H_k^{j+2}(U_0 \cap \Omegaout)} \\
&\quad\leq C \Big(\N{u}_{H^1_k(U_1)} + \N{ k^{-2}\nabla \cdot(A \nabla u) }_{H^j_k(U_1\cap \Omegain)} 
+ \N{k^{-2}\nabla\cdot(A \nabla u)}_{H^j_k(U_1\cap \Omegaout)}
\Big)
\label{eq:burrito2}
\end{align}
for $j=0,\ldots,\ell$. Since $r+d\leq \ccoer k^{-1}$ the coercivity \eqref{eq:coercivity1} holds. 
We then obtain \eqref{eq:burrito_required} from \eqref{eq:burrito2} by using the Lax--Milgram lemma, the triangle inequality, and the fact that multiplication by a $C^{\ell,1}$ function is continuous from $H^\ell$ to $H^\ell$ (see, e.g., \cite[Theorem 7.4]{BeHo:21} with 
$s_1=\ell-1, s_2 = \ell, s=\ell, p_1=\infty, p_2=2,$ and $p=2$). 
\epf

\ble[Satisfying Assumption \ref{ass:commut}]
\label{lem:ass_commut}
Suppose that \es{$\Ascatout, \cscatout \in C^{\ell,1}(\overline{\Omegaout}) \cap C^{\ell+1,1}(\Gt)$, 
$\Ascatin, \cscatin \in C^{\ell,1}(\overline{\Omegain})\cap C^{\ell+1,1}(\Gt)$, 
} 
the PML scaling function $f_\theta$ is $C^{\ell,1}(\overline{\Omega})$, 
and both $\partial \Omega$ and $\Gt$ are $C^{\ell+1,1}$. 
Then Parts (i), (ii), and (iii) of Assumption \ref{ass:commut} are satisfied for the sesquilinear forms defined by \eqref{eq:sesqui1}, \eqref{eq:sesqui2}, and \eqref{eq:sesqui3}.
\ele

\es{
To prove Lemma \ref{lem:ass_commut} we need the following lemma.

\ble\label{lem:tubular}
Given an
open set $U\subset \Rea^d$ with outward-pointing unit-normal vector $\nu$ and $C^{m+1,1}$ compact boundary, 
symmetric positive-definite functions 
$A_{\rm in}\in C^{m,1}(\partial U, \Rea^{d\times d})$, 
$A_{\rm out}\in C^{m,1}(\partial U, \Rea^{d\times d})$, 
$x_0\in \Rea^d$, and $0<r<R$, the following is true. There exists  $\psi:\Rea^d \to \R$ that is continuous on $\Rea^d$, $C^{m,1}$ on both $\overline{U}$ and $\Rea^d\setminus U$, and such that 
\beqs
\psi \equiv 0 \,\ton\, B(x_0,R)^c, \quad\quad\psi \equiv 1\,\ton\,B(x_0,r),
\eeqs
and
\beq\label{eq:derivative_condition1}
\nabla( \psi|_U) \cdot (A_{\rm in}\nu) =
\nabla (\psi|_{\Rea^d\setminus \overline{U}} )\cdot (A_{\rm out}\nu) = 0 \quad\ton \quad\partial U.
\eeq
\ele

We postpone the proof of Lemma \ref{lem:tubular} until after the proof of Lemma \ref{lem:ass_commut}. However, we highlight here that 
the construction of such a $\psi$ is possible since, by positive definiteness of $\Ascatin$ and $\Ascatout$, $(\Ascatin \bnu)\cdot \bnu$ and $(\Ascatout \bnu)\cdot\bnu$ are not zero for any $\bx \in \partial U$; the vectors $\Ascatin \bnu$ and $\Ascatout \bnu$ are therefore never tangent to $\partial U$ and so prescribing that $\psi$ is constant in these directions at $\partial U$ is possible.
}

\

\bpf[Proof of Lemma \ref{lem:ass_commut} \es{using Lemma \ref{lem:tubular}}]
Let $\psi\in C_{\rm comp}(\Rea^d;[0,1])$ satisfy the following four conditions:

\noi (1) 
\beq\label{eq:0400_1}
\supp\, \psi \subset B(x_0, r+3d/4) \quad\tand\quad \psi\equiv 1 \quad \ton B(x_0,r+d/4),
\eeq
(2)
\beqs
\psi \in C^{\es{\ell+1,1}}\big( B(x_0, r +3d/4) \cap \Omegain \big) \cap C^{\es{\ell+1,1}}\big( B(x_0, r +3d/4) \cap \Omegaout\cap \Omega_{\tr} \big)
\eeqs
with $\|\partial^j \psi\|_{L^\infty}\leq C d^{-j}$ in each of the two regions \es{for $j=1,\ldots,\ell+2$}, and 

\noi (3) with $\bnu$ the outward-pointing unit normal vector to $\Omegain$, 
\beq\label{eq:0400_3}
(\Ascatin \nabla \psi)\cdot \bnu =0 
\quad\tand\quad 
(\Ascatout \nabla \psi)\cdot \bnu =0 \quad\ton \Gt
\eeq
when the limits are taken from $\Omegain$ and $\Omegaout$, respectively, and 

\noi (4) with $\bnu$ the outward-pointing unit normal vector to $\Omega$, 
\beq\label{eq:0400_3a}
(\Ascatout \nabla \psi)\cdot \bnu =0 \quad\ton \Gamma_{\tr}
\eeq
and 
\beq\label{eq:0400_3b}
\text{ either }\quad (\Ascatin \nabla \psi)\cdot \bnu =0  \quad\ton\Gamma_- \quad\text{ or }\quad
(\Ascatout \nabla \psi)\cdot \bnu =0  \quad\ton\Gamma_-
\eeq
depending on whether $\Omega_-$ is inside $\Omegain$ or $\Omegaout$.

\es{Conditions (3) and (4) imply there are at most three interfaces across which $\nabla \psi$ jumps -- see Figure \ref{fig:final}. 
For each interface we construct such a $\psi$ by Lemma \ref{lem:tubular}, and then a $\psi$ satisfying the three jumps in Conditions (3) and (4), along with Conditions (1) and (2), is constructed using a partition of unity (where each of the three partition of unity functions is one near one of the three interfaces).}
%
%
%
%


\begin{figure}

\centering
\scalebox{0.8}{
\begin{tikzpicture}

\def\radiusA{4cm}
\def\radiusB{1.5cm}
\def\radiusC{1cm}

\def\startangle{-40}
\def\endangle{40}
\def\startangleB{-40}
\def\endangleB{40}
\def\cutangle{70}
\begin{scope}[scale=2,rotate=90]

\begin{scope}

\coordinate (O1) at (0,0);
\coordinate (O4) at (-1cm,0);
\coordinate (O2) at (3cm,0);
\coordinate (O3) at (1cm,0);
    \clip   ({\radiusA*cos(\cutangle)},{-\radiusA*sin(\startangle)})rectangle ({\radiusA*1.2},{\radiusA*sin(\startangle)});
    \draw[name path=circleB, pattern=north east lines] (O2) circle (\radiusB);
\draw[name path=circleA]  (O1) ++(\startangle:\radiusA) arc (\startangle:\endangle:\radiusA);
\draw[name path=circleD]  (O4) ++(\startangleB:\radiusA) arc (\startangleB:\endangleB:\radiusA);
            \draw[name path=circleC] (O3) circle (\radiusC);



     \draw (O1) ++ (35:\radiusA-7)node[below] {$\Omega_{\rm{out}}\cap \Omega$};
     \draw (O4) ++ (35:\radiusA-7)node[below] {$\Omega_{\rm{in}}$};
          \draw  (35:\radiusA+7)node[below] {$\Gamma_{\rm{tr}}$};
\begin{scope}[xshift=-1cm]
       \draw  (\startangleB+5:\radiusA)node[right] {$\Gamma_{\rm{p}}$};
       \end{scope}
       \begin{scope}[xshift=3cm]
       \draw  (0:\radiusB-10)node[above] {$\rm{supp}\ \psi$};
       \end{scope}
      \draw (1.5cm,-1cm)node[right]{$\Gamma_-$};
       \draw[->] (.9,0) --(.5,0);
            
\end{scope}            
\end{scope}




\end{tikzpicture}
}
\caption{
\es{The relative  locations of $\Gamma_{\tr}$, $\Gt$, $\Gamma_-$, and $\supp \,\psi$ (for particular choices of $x_0$ and large $r$ and $d$), where $\Omega_-$ and $\Ot$ are as in Figure \ref{fig:two_cases}(b) (i.e., $\Omega_-$ is inside $\Omegain$).
The condition \eqref{eq:0400_3a} is imposed on $\Gamma_{\tr}$, 
the condition \eqref{eq:0400_3} is imposed on $\Gt$, and 
the first condition in \eqref{eq:0400_3b} is imposed on $\Gamma_{-}$.
}
}
\label{fig:final}
\end{figure}

Part (i) of Assumption \ref{ass:commut} holds from \eqref{eq:0400_1}. Part (ii) of Assumption \ref{ass:commut} holds by the Leibnitz rule applied piecewise in $B(x_0, r +3d/4) \cap \Omegain$ and $B(x_0, r +3d/4) \cap (\Omegaout\cap\Omega)$.

We now check Part (iii). For all three of the sesquilinear forms, 
\begin{align}\nonumber
&k^{2}\big(a_k(\psi u,v) - a_k(u,\psi v)\big)\\ \nonumber
&\quad= \int_{\Omega\cap\Omegaout} A \nabla(\psi u)\cdot\overline{\nabla v}  - A \nabla u \cdot \nabla(\psi \overline{v})
+\gamma^{-1}\int_{\Omegain} \Ascatin \nabla(\psi u)\cdot \overline{\nabla v}  - \Ascatin \nabla u \cdot \nabla(\psi \overline{v})\\
&\quad=\int_{(\Omega\cap\Omegaout)\cap (U_1\setminus U_0)} \nabla \psi \cdot\big( u (A\overline{\nabla v}) - \overline{v}(A\nabla u)\big)
+\gamma^{-1}\int_{\Omegain\cap (U_1\setminus U_0)} \nabla \psi \cdot\big( u (\Ascatin\overline{\nabla v}) - \overline{v}(\Ascatin\nabla u)\big).
\label{eq:0400_2}
\end{align}
We now prove \eqref{eq:commute} with the first argument of the minimum on the right-hand side; the proof for the second argument follows by swapping the roles of $u$ and $v$.

We first bound the fourth term on the right-hand side of \eqref{eq:0400_2}; the analogous term over $(\Omega\cap \Omega)\cap (U_1\setminus U_0)$ (i.e., the second term on the right-hand side of \eqref{eq:0400_2}) is bounded in an identical way.
We use (in the following order) the definition of $\Wjdual{\cdot}{j-1}{U_1\setminus U_0}$, the analogue of \eqref{eq:improvedLeibniz} in the $\cW^j_k$ norms (and with $\psi$ replaced by $k^{-1}\nabla \psi$), 
the fact that $\Ascatin$ is $C^{\ell,1}$, 
and the fact that $k^{-1}\partial$  
maps $\cZ^j_k$ to $\cW^{j-1}_k$ with norm bounded independent of $k$
(with  $\cW^j_k$ and $\cZ^j_k$ given by \eqref{eq:WZ})  to obtain that, for $j=1,\ldots,\ell+1$, 
\begin{align}\nonumber
k^{-2}\left|\int_{\Omegain\cap (U_1\setminus U_0)}  \overline{v} (\nabla \psi) \cdot (\Ascatin\nabla u)\right|
&\leq \Wjdual{v}{j-1}{U_1\setminus U_0}\Wj{(k^{-1}\nabla \psi) \cdot (k^{-1}\Ascatin\nabla u)}{j-1}{\Omegain\cap (U_1\setminus U_0)}\\
&\leq \Wjdual{v}{j-1}{U_1\setminus U_0}\frac{C C_\dagger}{kd} \bigg(\sum_{m=0}^{j-1}(kd)^{-(j-1-m)}\bigg)\Zj{u}{j}{U_1\setminus U_0}.
\label{eq:0400_4}
\end{align}

It therefore remains to bound the first and third terms on the right-hand side of \eqref{eq:0400_2}. By the symmetry of $\Ascatin$ and the divergence theorem,
\begin{align*}
\int_{\Omegain \cap (U_1\setminus U_0)}
u (\nabla \psi)\cdot (\Ascatin \overline{\nabla v}) 
&=\int_{\Omegain \cap (U_1\setminus U_0)}
(\overline{\nabla v})\cdot (u\Ascatin \nabla \psi) \\
&=\int_{\partial(\Omegain \cap (U_1\setminus U_0))}
u \overline{v} (\Ascatin\nabla \psi)\cdot \bnu
- 
\int_{\Omegain \cap (U_1\setminus U_0)}
\overline{v} \nabla\cdot( u \Ascatin \nabla\psi).
\end{align*}
We now claim that the boundary integral over $\partial(\Omegain \cap (U_1\setminus U_0))$ is equal to zero. 
This boundary integral can be split into integrals over $\partial(U_1\setminus U_0)$, $\Gt$,  $\Gamma_-$, and $\Gamma_{\tr}$. 
(Note that, since $U_1$ has characteristic length scale $k^{-1}$ and $\Gamma_-, \Gt,$ and $\Gamma_{\tr}$ are all a $k$-independent distance apart, there will never be boundary integrals over any two of $\Gamma_-, \Gt,$ and $\Gamma_{\tr}$ at the same time for $k$ sufficiently large.)
The integrals over $U_1\setminus U_0$, vanish because $\nabla\psi$ is zero here, and the integrals over $\Gt$ and $\partial\Omega$ vanish because of the first conditions in \eqref{eq:0400_3} and \eqref{eq:0400_3a}.

Therefore
\begin{align*}
k^{-2}\int_{\Omegain \cap (U_1\setminus U_0)}
u (\nabla \psi)\cdot (\Ascatin \overline{\nabla v}) 
&=
k^{-2}\int_{\Omegain \cap (U_1\setminus U_0)}
\overline{v} \nabla\cdot( u \Ascatin \nabla\psi)\\
&=
k^{-2}\int_{\Omegain \cap (U_1\setminus U_0)}\Big(
\overline{v} (\nabla u)\cdot( \Ascatin \nabla\psi) + \overline{v}u \nabla\cdot (\Ascatin \nabla\psi)\Big).
\end{align*}
The first term on the right-hand side is bounded exactly as in \eqref{eq:0400_4}; the second term is bounded similarly, and the proof is complete.
\epf

\

\es{
It therefore remains to prove Lemma \ref{lem:tubular}.

\

\bpf[Proof of Lemma \ref{lem:tubular}]
Let $\chi \in C^\infty_{\rm comp}(\R^d)$ be supported in $B(x_0,R')$ and identically $1$ in $B(x_0,r')$ with $r < r' < R' < R$.

Since $\Ascatin\nu$ and $\Ascatout\nu$ are uniformly transverse to $\partial U$, 
we now claim that there exists a neighborhood $U_{\varepsilon}$ of $\partial U$ in $\R^d$ such that 
\beqs
X_{\rm in} : \partial U \times [-\varepsilon,\varepsilon] \to U_{\varepsilon},\quad X_{\rm in}(x,t):= x + t \Ascatin(x)\nu(x)
\eeqs
and 
\beqs
X_{\rm out} : \partial U \times [-\varepsilon,\varepsilon] \to U_{\varepsilon},\quad X_{\rm out}(x,t):= x + t \Ascatout(x)\nu(x)
\eeqs
are bijections. Indeed, this follows from (a) the inverse function theorem (valid for Lipschitz maps -- and hence for $A_{\rm in/out}\in C^{0,1}$ and $\nu \in C^{1,1}$ -- by \cite[Theorem 1]{Cl:76})
and (b) the fact that, for $\varepsilon$ small, $X_{\rm in/out}$ are injective.

Furthermore, since $X_{\rm in}$ maps into $U$ for $t< 0$ and $X_{\rm out}$ maps into $\Rea^d\setminus \overline{U}$ for $t> 0$, we define
\beqs
X(x,\varepsilon):=
\begin{cases}
X_{\rm in}(x,t), &t< 0,\\
X_{\rm out}(x,t), &t\geq 0.
\end{cases}
\eeqs
%
%
 Define $\widetilde\psi$ on $U_\varepsilon$ by $\widetilde\psi(X(x,t)) = \widetilde\psi(X(x,0)) = \chi(x)$ and observe that 
 \eqref{eq:derivative_condition1} holds.
 

Now, since $\chi \equiv 1$ on $\partial U \cap B(x_0,r')$, by reducing $\varepsilon$ if necessary, 
we can ensure that $\widetilde\psi \equiv 1$ on $B(x_0,r) \cap U_{\varepsilon}$. 
Similarly, since $\chi  \equiv 0$ on $\partial U \cap B(x_0,R')$, by reducing $\varepsilon$ if necessary, we can ensure that $\widetilde\psi \equiv 0$ on $U_\varepsilon \cap B(x_0,R)$.

 We extend $\widetilde\psi$ by $0$ outside $U_{\varepsilon}$, and define $\psi := \eta \widetilde\psi + (1 - \eta) \chi$ where $\eta$ is smooth, supported on $U_{\varepsilon/2}$ (so that $\psi$ is smooth away from $\partial U$ and $\equiv 1$ on $B(x_0,r)$), and identically equal to one in $U_{\varepsilon/4}$ (to preserve the condition on $\partial U$).
\epf
}

\subsection{The assumptions on $V_h$ (Assumptions \ref{ass:ap}, \ref{ass:sa}, and \ref{ass:ii})}\label{sec:verifySubspace}

\ble
Assumptions \es{\ref{ass:bn},} 
 \ref{ass:ap}, \ref{ass:sa}, and \ref{ass:ii} hold when $V_h$ is a space of Lagrange finite elements 
 \es{that resolves $\partial \Omega$ and $\Gt$}, with $p$ the polynomial degree, $\kappa =1$, the constants $\Capprox, \Csuperk,$ and $\Cinv$ independent of $h_K$ and $k$, and (for Assumption \ref{ass:ap}) the space $\cZ_k^j$ given as in \S\ref{sec:spaces}.
\ele

\bpf 
\es{In Assumption \ref{ass:bn}, 
\eqref{eq:normsOnElements} and \eqref{ass:loc} are satisfied by the definition of $\cZ_k$ \eqref{eq:Z1}, 
and \eqref{eq:controlHONorms} is satisfied by the assumption that $V_h$ resolves  $\partial \Omega$ and $\Gt$.
}

The approximation property of Assumption \ref{ass:ap} with $\Ckappa=1$ and $u_h$ equal to the Lagrange interpolant of $u$ holds by, e.g., \cite[Theorem 4.4.20]{BrSc:08} and the definition \eqref{eq:defWeightedNorms} of the $k$-weighted norms. 

When $\chi$ is a smooth function on $K$ and $v_h$ is the Lagrange interpolant of $\chi^2 u_h$, 
\begin{align}\label{eq:sa1}
\norm{\chi^2 u_h - v_h}_{L^2(K)} &\leq \Csuper \frac{h_K}{d} \norm{u_h}_{L^2(K)} 
\quad\tand
\\ \label{eq:sa2}
\norm{\nabla(\chi^2 u_h - v_h)}_{L^2(K)} &\leq \Csuper \left(\frac{h_K}{d} \norm{\nabla (\chi u_h)}_{L^2(K)} + \frac{h_K}{d^2} \norm{u_h}_{L^2(K)}\right)
\end{align}
by \cite[Theorem 2.1]{DeGuSc:11}, \cite[Theorem 1]{Br:20}. 
The bound \eqref{eq:saH1k} then follows from \eqref{eq:defWeightedNorms}, \eqref{eq:sa1}, \eqref{eq:sa2}, and the inequality 
\beq\label{eq:inequality}
\sqrt{a^2+b^2}\leq a+b
\quad\tfa a,b>0.
\eeq
That \eqref{eq:dU0U1} holds with $\kappa=1$ follows the fact that 
$\supp \chi$ and $B(0,d)$ are $d/4$ apart, and thus $v_h\in V_h^{<}(U_1)$ is ensured by $\max_{K \cap U_1\neq\emptyset} h_K<d/4$.

By, e.g., \cite[Lemma 4.5.3]{BrSc:08}, 
given $p\in \mathbb{Z}^+$, there exists a constant $\Cinv > 0$ such that
	\begin{equation}
		\label{eq:ii}
		\norm{u_h}_{H^s(K)} \leq \Cinv h_K^{-s} \norm{u_h}_{L^2(K)} 
	\end{equation}
for all $K \in \mathcal{T}$, $u_h \in V_h$, and $0\leq s\leq p$.
Then, by \eqref{eq:ii}, \eqref{eq:defWeightedNorms}, and \eqref{eq:Cppw}, there exists $\Cinvk>0$ such that, for $0\leq s\leq p$,
\beq\label{eq:scream1}
\norm{u_h}_{H^s_k(K)} \leq \frac{\Cinvk}{(h_Kk)^s} \norm{u_h}_{L^2(K)};
\eeq
the inequality \eqref{eq:iiH1k1} is then this with $s=1$ 

The inequality \eqref{eq:iiH1k2} follows by repeating the proof of 
\cite[Theorem 3.6]{GrHaSa:05} with the inverse inequality \eqref{eq:ii} replaced by \eqref{eq:scream1} in \cite[Equation 3.23]{GrHaSa:05}.
Note that in \cite{GrHaSa:05}, the $H^{-s}(K)$ norm is defined with a supremum ranging over all elements of $H^s(K)$, but, as stated in \cite[Remark 3.8]{GrHaSa:05}, the proof works without modification for the $H^{-s}(K)$ norm defined by \eqref{eq:negativeNorm} with $k=1$ and equivalent to a norm defined with a supremum over elements of $H^s(K)$ supported inside $K$ as in \eqref{eq:negativeNormEquiv}. 
\epf

\appendix

\section{Recap of Sobolev spaces weighted by $k$}\label{sec:norms}
Given $k>0$, let
\beq\label{eq:FT}
\mathcal F_k\phi(\xi) := \int_{\mathbb R^d} \exp\big( -\ri k x \cdot \xi\big)
\phi(x) \, \rd x
\eeq
and, for $s\in\Rea$, let
\beq\label{eq:Sobolev}
H_k ^ s (\Rea^d):= \Big\{ u\in \cS^*(\mathbb R^d), \; \langle \xi \rangle^s 
\mathcal F_k u \in  L^2(\mathbb R^d) \Big\}
\quad\text{ with }\quad
\norm{u}_{H_k^s(\Rea^d)} ^2 := 
\left(\frac{k}{2\pi}\right)^d
 \int_{\Rea^d} \langle \xi \rangle^{2s}
 |\mathcal F_k u(\xi)|^2 \, \rd \xi,
\eeq
where
$\langle \xi \rangle := (1+|\xi|^2)^{1/2}.$
For an open $D\subset \Rea^d$, let 
\beq\label{eq:defWeightedNorms}
\norm{v}_{H^s_k(D)} := \inf_{V\in H^s_k(\Rea^d)\,:\, V|_{D} = v} \norm{V}_{H^s_k(\R^d)}.
\eeq
Recall that, for $s\geq 0$ and $D$ a bounded Lipschitz domain, by, e.g., \cite[Page 77 and Theorem 3.30(i), Page 92]{Mc:00},
\beq\label{eq:neg_norm_intro}
\norm{v}_{H^{-s}_k(D)} \sim \sup_{\norm{w}_{H^s_k(D)} = 1\,,\, \supp w\subset D} |
v(w)|,
\eeq
where $\langle\cdot,\cdot\rangle_D$ is the duality pairing in $D$ and $\sim$ denotes norm equivalence,
and 
\beq\label{eq:neg_norm_intro2}
\norm{v}_{\widetilde{H}^{-s}_k(D)} \sim \sup_{\norm{w}_{H^s_k(D)} = 1} |
v(w)|.
\eeq
 We highlight that, for $s=m \in \mathbb{Z}^+$, \eqref{eq:defWeightedNorms} is equivalent to the norm
\beq\label{eq:weightedNormsOld}
	\vertiii{v}^2_{H^m_k(D)} := \sum_{0 \leq \abs{\alpha} \leq m} k^{-2\abs{\alpha}} \norm{\partial^\alpha v}^2_{L^2(D)},
\eeq
and thus, in particular,
\beq\label{eq:1knorm}
\norm{v}^2_{H^1_k(D)}\sim\vertiii{v}^2_{H^1_k(D)} =k^{-2}\|\nabla v \|^2_{L^2(D)} + \|v \|^2_{L^2(D)},
\eeq
where $\sim$ again denotes norm equivalence.
Many papers on numerical analysis of the Helmholtz equation use the weighted $H^1$ norm 
$\|v\|^2_{H^1_k(D)}:= \|\nabla v \|^2_{L^2(D)} + k^2\|v \|^2_{L^2(D)}$;
we use \eqref{eq:weightedNormsOld}/\eqref{eq:defWeightedNorms} instead since weighting the $j$th derivative by $k^{-j}$ is easier to keep track of than weighting it by $k^{-j+1}$ (especially for high derivatives).

\section*{Acknowledgements}

\es{The authors thanks the referees for their careful reading of the paper and numerous suggestions for improvement.}
The idea of looking at the local FEM error for the Helmholtz equation came out of discussions EAS had with Ralf Hiptmair (ETH Z\"urich). 
MA and EAS were supported by EPSRC grant EP/R005591/1 and JG was supported by EPSRC grants EP/V001760/1 and EP/V051636/1.\\

\noi The authors declared that they have no conflict of interest.

\footnotesize{
\bibliographystyle{plain}
\bibliography{biblio_combined_sncwadditions}

\begin{thebibliography}{10}

\bibitem{BeHo:21}
A.~Behzadan and M.~Holst.
\newblock {Multiplication in Sobolev spaces, revisited}.
\newblock {\em Arkiv f{\"o}r Matematik}, 59(2):275--306, 2021.

\bibitem{BrPa:07}
J.~Bramble and J.~Pasciak.
\newblock {Analysis of a finite PML approximation for the three dimensional
  time-harmonic Maxwell and acoustic scattering problems}.
\newblock {\em Mathematics of Computation}, 76(258):597--614, 2007.

\bibitem{Br:20}
S.~C. Brenner.
\newblock A general superapproximation result.
\newblock {\em Computational Methods in Applied Mathematics}, 20(4):763--767,
  2020.

\bibitem{BrSc:08}
S.~C. Brenner and L.~R. Scott.
\newblock {\em The Mathematical Theory of Finite Element Methods}, volume~15 of
  {\em Texts in Applied Mathematics}.
\newblock Springer, 3rd edition, 2008.

\bibitem{ChMo:08}
S.~N. Chandler-Wilde and P.~Monk.
\newblock {Wave-number-explicit bounds in time-harmonic scattering}.
\newblock {\em SIAM J. Math. Anal.}, 39(5):1428--1455, 2008.

\bibitem{Ci:91}
P.~G. Ciarlet.
\newblock Basic error estimates for elliptic problems.
\newblock In {\em Handbook of numerical analysis, {V}ol.\ {II}}, pages 17--351.
  North-Holland, Amsterdam, 1991.

\bibitem{Cl:76}
F.~Clarke.
\newblock On the inverse function theorem.
\newblock {\em Pacific Journal of Mathematics}, 64(1):97--102, 1976.

\bibitem{CoDaNi:10}
M.~Costabel, M.~Dauge, and S.~Nicaise.
\newblock {Corner Singularities and Analytic Regularity for Linear Elliptic
  Systems. Part I: Smooth domains.}
\newblock 2010.
\newblock
  \url{https://hal.archives-ouvertes.fr/file/index/docid/453934/filename/CoDaNi_Analytic_Part_I.pdf}.

\bibitem{DeGuSc:11}
A.~Demlow, J.~Guzm\'{a}n, and A.~H. Schatz.
\newblock Local energy estimates for the finite element method on sharply
  varying grids.
\newblock {\em Mathematics of Computation}, 80(273):1--9, 2011.

\bibitem{De:75}
J.~Descloux.
\newblock {Interior regularity and local convergence of Galerkin finite element
  approximations for elliptic equations}.
\newblock In J.J.H. Miller, editor, {\em Topics in Numerical Analysis II},
  pages 27--41. Academic Press, 1975.

\bibitem{GLS1}
J.~Galkowski, D.~Lafontaine, and E.~A. Spence.
\newblock {Local absorbing boundary conditions on fixed domains give order-one
  errors for high-frequency waves}.
\newblock {\em IMA. J. Num. Anal.},
  \url{https://doi.org/10.1093/imanum/drad058}, 2023.

\bibitem{GLS2}
J.~Galkowski, D.~Lafontaine, and E.~A. Spence.
\newblock Perfectly-matched-layer truncation is exponentially accurate at high
  frequency.
\newblock {\em SIAM J. Math. Anal.}, 55(4):3344--3394, 2023.

\bibitem{GLSW1}
J.~Galkowski, D.~Lafontaine, E.A. Spence, and J.~Wunsch.
\newblock {The $hp$-FEM applied to the Helmholtz equation with PML truncation
  does not suffer from the pollution effect}.
\newblock {\em Comm. Math. Sci.}, to appear, 2024.

\bibitem{GS3}
J.~Galkowski and E.~A. Spence.
\newblock {Sharp preasymptotic error bounds for the Helmholtz $h$-FEM}.
\newblock {\em arXiv 2301.03574}, 2023.

\bibitem{GrHaSa:05}
I.~G. Graham, W.~Hackbusch, and S.~A. Sauter.
\newblock Finite elements on degenerate meshes: inverse-type inequalities and
  applications.
\newblock {\em IMA Journal of Numerical Analysis}, 25(2):379--407, 2005.

\bibitem{GrPeSp:19}
I.~G. Graham, O.~R. Pembery, and E.~A. Spence.
\newblock {The Helmholtz equation in heterogeneous media: a priori bounds,
  well-posedness, and resonances}.
\newblock {\em Journal of Differential Equations}, 266(6):2869--2923, 2019.

\bibitem{Gr:85}
P.~Grisvard.
\newblock {\em Elliptic problems in nonsmooth domains}.
\newblock Pitman, Boston, 1985.

\bibitem{He:12}
F.~Hecht.
\newblock New development in freefem++.
\newblock {\em J. Numer. Math.}, 20(3-4):251--265, 2012.

\bibitem{HoScZs:03}
T.~Hohage, F.~Schmidt, and L.~Zschiedrich.
\newblock {Solving time-harmonic scattering problems based on the pole
  condition II: convergence of the PML method}.
\newblock {\em SIAM Journal on Mathematical Analysis}, 35(3):547--560, 2003.

\bibitem{IhBa:97}
F.~Ihlenburg and I.~Babuska.
\newblock {Finite element solution of the Helmholtz equation with high wave
  number part II: the $hp$ version of the FEM}.
\newblock {\em SIAM J. Numer. Anal.}, 34(1):315--358, 1997.

\bibitem{LaSo:98}
M.~Lassas and E.~Somersalo.
\newblock {On the existence and convergence of the solution of PML equations}.
\newblock {\em Computing}, 60(3):229--241, 1998.

\bibitem{LaSo:01}
M.~Lassas and E.~Somersalo.
\newblock {Analysis of the PML equations in general convex geometry}.
\newblock {\em Proceedings of the Royal Society of Edinburgh Section A:
  Mathematics}, 131(5):1183--1207, 2001.

\bibitem{MGSS1}
P.~Marchand, J.~Galkowski, E.~A. Spence, and A.~Spence.
\newblock {Applying GMRES to the Helmholtz equation with strong trapping: how
  does the number of iterations depend on the frequency?}
\newblock {\em Advances in Computational Mathematics}, 48(4):1--63, 2022.

\bibitem{Mc:00}
W.~McLean.
\newblock {\em {Strongly elliptic systems and boundary integral equations}}.
\newblock Cambridge University Press, 2000.

\bibitem{MeSa:10}
J.~M. Melenk and S.~Sauter.
\newblock Convergence analysis for finite element discretizations of the
  {H}elmholtz equation with {D}irichlet-to-{N}eumann boundary conditions.
\newblock {\em Math. Comp}, 79(272):1871--1914, 2010.

\bibitem{NiSc:74}
J.~A. Nitsche and A.~H. Schatz.
\newblock {Interior estimates for Ritz-Galerkin methods}.
\newblock {\em Mathematics of Computation}, 28(128):937--958, 1974.

\bibitem{ScWa:77}
A.~H. Schatz and L.~B. Wahlbin.
\newblock Interior maximum norm estimates for finite element methods.
\newblock {\em Mathematics of Computation}, 31(138):414--442, 1977.

\bibitem{ScWa:82}
A.~H. Schatz and L.~B. Wahlbin.
\newblock {On the quasi-optimality in $L_\infty$ of the
  $\overset{\circ}{H}_1$-projection into finite element spaces}.
\newblock {\em Mathematics of Computation}, 38(157):1--22, 1982.

\bibitem{Wa:91}
L.~B. Wahlbin.
\newblock Local behavior in finite element methods.
\newblock In {\em Handbook of numerical analysis, {V}ol. {II}}, pages 353--522.
  North-Holland, Amsterdam, 1991.

\bibitem{wilson2007computing}
H.~B. Wilson and R.~W. Scharstein.
\newblock Computing elliptic membrane high frequencies by mathieu and galerkin
  methods.
\newblock {\em Journal of Engineering Mathematics}, 57(1):41--55, 2007.

\end{thebibliography}
}

\end{document}